\documentclass[10pt]{amsart}
\usepackage{amsmath}
\usepackage{amsxtra}
\usepackage{amscd}
\usepackage{amsthm}
\usepackage{amsfonts}
\usepackage{amssymb}
\usepackage{eucal}
\usepackage[all]{xy}
\usepackage{graphicx}
\usepackage{comment}
\usepackage{epsfig}
\usepackage{psfrag}
\usepackage{mathrsfs}
\usepackage{amscd}
\usepackage{rotating}
\usepackage{lscape}
\usepackage{amsbsy}
\usepackage{verbatim}
\usepackage{moreverb}
\usepackage{mathabx}

\newtheorem{cor}[subsubsection]{Corollary}
\newtheorem{lem}[subsubsection]{Lemma}
\newtheorem{prop}[subsubsection]{Proposition}

\newtheorem{conj}[subsubsection]{Conjecture}
\newtheorem{thm}[subsubsection]{Theorem}

\theoremstyle{remark}
\newtheorem{remark}[subsubsection]{Remark}
\newtheorem{rem}[subsubsection]{Remark}

\newcommand{\thmref}[1]{Theorem~\ref{#1}}

\newcommand{\secref}[1]{Sect.~\ref{#1}}
\newcommand{\lemref}[1]{Lemma~\ref{#1}}
\newcommand{\propref}[1]{Proposition~\ref{#1}}
\newcommand{\corref}[1]{Corollary~\ref{#1}}

\numberwithin{equation}{section}

\newcommand{\nc}{\newcommand}
\nc{\renc}{\renewcommand}
\nc{\ssec}{\subsection}
\nc{\sssec}{\subsubsection}
\nc{\on}{\operatorname}

\nc\ol{\overline}
\nc\wt{\widetilde}
\nc\tboxtimes{\wt{\boxtimes}}
\nc\tstar{\wt{\star}}
\nc{\alp}{\alpha}

\nc{\ZZ}{{\mathbb Z}}
\nc{\NN}{{\mathbb N}}
\nc{\OO}{{\mathbb O}}
\renc{\SS}{{\mathbb S}}
\nc{\DD}{{\mathbb D}}
\nc{\GG}{{\mathbb G}}

\nc{\Fq}{{\mathbb F}_q}
\nc{\Fqb}{\ol{{\mathbb F}_q}}
\nc{\Ql}{\ol{{\mathbb Q}_\ell}}
\nc{\id}{\text{id}}
\nc\X{\mathcal X}

\nc{\Hom}{\on{Hom}}
\nc{\Lie}{\on{Lie}}
\nc{\Loc}{\on{Loc}}
\nc{\Pic}{\on{Pic}}
\nc{\Bun}{\on{Bun}}
\nc{\IC}{\on{IC}}
\nc{\Aut}{\on{Aut}}
\nc{\rk}{\on{rk}}
\nc{\Sh}{\on{Sh}}
\nc{\Perv}{\on{Perv}}
\nc{\pos}{{\on{pos}}}
\nc{\Conv}{\on{Conv}}
\nc{\Sph}{\on{Sph}}
\nc{\Sym}{\on{Sym}}
\nc{\BunBb}{\overline{\Bun}_B}
\nc{\BunNb}{\overline{\Bun}_N}
\nc{\BunTb}{\overline{\Bun}_T}
\nc{\BunBbm}{\overline{\Bun}_{B^-}}
\nc{\BunBbel}{\overline{\Bun}_{B,el}}
\nc{\BunBbmel}{\overline{\Bun}_{B^-,el}}
\nc{\Buno}{\overset{o}{\Bun}}
\nc{\BunPb}{{\overline{\Bun}_P}}
\nc{\BunBM}{\Bun_{B(M)}}
\nc{\BunBMb}{\overline{\Bun}_{B(M)}}
\nc{\BunPbw}{{\widetilde{\Bun}_P}}
\nc{\BunBP}{\widetilde{\Bun}_{B,P}}
\nc{\GUb}{\overline{G/U}}
\nc{\GUPb}{\overline{G/U(P)}}

\nc{\Hhom}{\underline{\on{Hom}}}
\nc\syminfty{\on{Sym}^{\infty}}
\nc\lal{\ol{\lambda}}
\nc\xl{\ol{x}}
\nc\thl{\ol{\theta}}
\nc\nul{\ol{\nu}}
\nc\mul{\ol{\mu}}
\nc\Sum\Sigma
\nc{\oX}{\overset{\circ}{X}{}}
\nc{\hl}{\overset{\leftarrow}h{}}
\nc{\hr}{\overset{\rightarrow}h{}}
\nc{\M}{{\mathcal M}}
\nc{\N}{{\mathcal N}}
\nc{\F}{{\mathcal F}}
\nc{\D}{{\mathcal D}}
\nc{\Q}{{\mathcal Q}}
\nc{\Y}{{\mathcal Y}}
\nc{\G}{{\mathcal G}}
\nc{\E}{{\mathcal E}}
\nc{\CalC}{{\mathcal C}}
\nc\Dh{\widehat{\D}}

\nc{\C}{{\mathcal C}}
\nc{\K}{{\mathcal K}}
\renewcommand{\H}{{\mathcal H}}

\nc{\T}{{\mathcal T}}
\nc{\V}{{\mathcal V}}
\renc{\P}{{\mathcal P}}
\nc{\A}{{\mathcal A}}
\nc{\B}{{\mathcal B}}
\nc{\U}{{\mathcal U}}

\nc{\Gr}{{\on{Gr}}}

\nc{\frn}{{\check{\mathfrak u}(P)}}

\nc{\fC}{\mathfrak C}
\nc{\p}{\mathfrak p}
\nc{\q}{\mathfrak q}
\nc\f{{\mathfrak f}}

\nc{\qo}{{\mathfrak q}}
\nc{\po}{{\mathfrak p}}
\nc{\s}{{\mathfrak s}}
\nc\w{\text{w}}

\nc\mathi\iota
\nc\Spec{\on{Spec}}
\nc\Mod{\on{Mod}}
\nc{\tw}{\widetilde{\mathfrak t}}
\nc{\pw}{\widetilde{\mathfrak p}}
\nc{\qw}{\widetilde{\mathfrak q}}
\nc{\jw}{\widetilde j}

\nc{\grb}{\overline{\Gr_{X^{\fset}}}}
\nc{\I}{\mathcal I}

\nc{\lambdach}{{\check\lambda}}
\nc{\Lambdach}{{\check\Lambda}{}}
\nc{\much}{{\check\mu}}
\nc{\omegach}{{\check\omega}}
\nc{\nuch}{{\check\nu}}
\nc{\etach}{{\check\eta}}
\nc{\alphach}{{\check\alpha}}

\nc{\rhoch}{{\check\rho}}
\nc{\ch}{{\check h}}

\nc{\Hb}{\overline{\H}}

\nc{\BA}{{\mathbb{A}}}
\nc{\BC}{{\mathbb{C}}}
\nc{\BG}{{\mathbb{G}}}
\nc{\BM}{{\mathbb{M}}}
\nc{\BO}{{\mathbb{O}}}
\nc{\BD}{{\mathbb{D}}}
\nc{\BN}{{\mathbb{N}}}
\nc{\BP}{{\mathbb{P}}}
\nc{\BQ}{{\mathbb{Q}}}
\nc{\BR}{{\mathbb{R}}}
\nc{\BZ}{{\mathbb{Z}}}
\nc{\BS}{{\mathbb{S}}}
\nc{\BV}{{\mathbb{V}}}

\nc{\CA}{{\mathcal{A}}}
\nc{\CB}{{\mathcal{B}}}

\nc{\CE}{{\mathcal{E}}}
\nc{\CF}{{\mathcal{F}}}
\nc{\CH}{{\mathcal{H}}}

\nc{\CL}{{\mathcal{L}}}
\nc{\CC}{{\mathcal{C}}}
\nc{\CG}{{\mathcal{G}}}
\nc{\CM}{{\mathcal{M}}}
\nc{\CN}{{\mathcal{N}}}
\nc{\CK}{{\mathcal{K}}}
\nc{\CO}{{\mathcal{O}}}
\nc{\CP}{{\mathcal{P}}}
\nc{\CQ}{{\mathcal{Q}}}
\nc{\CR}{{\mathcal{R}}}
\nc{\CS}{{\mathcal{S}}}
\nc{\CT}{{\mathcal{T}}}
\nc{\CU}{{\mathcal{U}}}
\nc{\CV}{{\mathcal{V}}}
\nc{\CW}{{\mathcal{W}}}
\nc{\CX}{{\mathcal{X}}}
\nc{\CY}{{\mathcal{Y}}}
\nc{\CZ}{{\mathcal{Z}}}
\nc{\CI}{{\mathcal{I}}}

\nc{\csM}{{\check{\mathcal A}}{}}
\nc{\oM}{{\overset{\circ}{\mathcal M}}{}}
\nc{\obM}{{\overset{\circ}{\mathbf M}}{}}
\nc{\oCA}{{\overset{\circ}{\mathcal A}}{}}
\nc{\obA}{{\overset{\circ}{\mathbf A}}{}}
\nc{\ooM}{{\overset{\circ}{M}}{}}
\nc{\osM}{{\overset{\circ}{\mathsf M}}{}}
\nc{\vM}{{\overset{\bullet}{\mathcal M}}{}}
\nc{\nM}{{\underset{\bullet}{\mathcal M}}{}}
\nc{\oD}{{\overset{\circ}{\mathcal D}}{}}
\nc{\obC}{{\overset{\circ}{\mathbf C}}{}}
\nc{\obD}{{\overset{\circ}{\mathbf D}}{}}
\nc{\oA}{{\overset{\circ}{\mathbb A}}{}}
\nc{\op}{{\overset{\bullet}{\mathbf p}}{}}
\nc{\cp}{{\overset{\circ}{\mathbf p}}{}}
\nc{\oU}{{\overset{\bullet}{\mathcal U}}{}}
\nc{\oZ}{{\overset{\circ}{\mathcal Z}}{}}
\nc{\ofZ}{{\overset{\circ}{\mathfrak Z}}{}}
\nc{\oF}{{\overset{\circ}{\fF}}}

\nc{\fa}{{\mathfrak{a}}}
\nc{\fb}{{\mathfrak{b}}}
\nc{\fd}{{\mathfrak{d}}}
\nc{\ff}{{\mathfrak{f}}}
\nc{\fg}{{\mathfrak{g}}}
\nc{\fgl}{{\mathfrak{gl}}}
\nc{\fh}{{\mathfrak{h}}}
\nc{\fj}{{\mathfrak{j}}}
\nc{\fl}{{\mathfrak{l}}}
\nc{\fm}{{\mathfrak{m}}}
\nc{\fn}{{\mathfrak{n}}}
\nc{\fu}{{\mathfrak{u}}}
\nc{\fp}{{\mathfrak{p}}}
\nc{\fr}{{\mathfrak{r}}}
\nc{\fs}{{\mathfrak{s}}}
\nc{\ft}{{\mathfrak{t}}}
\nc{\fz}{{\mathfrak{z}}}
\nc{\fsl}{{\mathfrak{sl}}}
\nc{\hsl}{{\widehat{\mathfrak{sl}}}}
\nc{\hgl}{{\widehat{\mathfrak{gl}}}}
\nc{\hg}{{\widehat{\mathfrak{g}}}}
\nc{\chg}{{\widehat{\mathfrak{g}}}{}^\vee}
\nc{\hn}{{\widehat{\mathfrak{n}}}}
\nc{\chn}{{\widehat{\mathfrak{n}}}{}^\vee}

\nc{\fA}{{\mathfrak{A}}}
\nc{\fB}{{\mathfrak{B}}}
\nc{\fD}{{\mathfrak{D}}}
\nc{\fE}{{\mathfrak{E}}}
\nc{\fF}{{\mathfrak{F}}}
\nc{\fG}{{\mathfrak{G}}}
\nc{\fK}{{\mathfrak{K}}}
\nc{\fL}{{\mathfrak{L}}}
\nc{\fM}{{\mathfrak{M}}}
\nc{\fN}{{\mathfrak{N}}}
\nc{\fP}{{\mathfrak{P}}}
\nc{\fU}{{\mathfrak{U}}}
\nc{\fV}{{\mathfrak{V}}}
\nc{\fZ}{{\mathfrak{Z}}}

\nc{\bb}{{\mathbf{b}}}
\nc{\bc}{{\mathbf{c}}}
\nc{\bd}{{\mathbf{d}}}
\nc{\bbf}{{\mathbf{f}}}
\nc{\be}{{\mathbf{e}}}
\nc{\bg}{{\mathbf{g}}}
\nc{\bi}{{\mathbf{i}}}
\nc{\bj}{{\mathbf{j}}}
\nc{\bn}{{\mathbf{n}}}
\nc{\bp}{{\mathbf{p}}}
\nc{\bq}{{\mathbf{q}}}
\nc{\bu}{{\mathbf{u}}}
\nc{\bv}{{\mathbf{v}}}
\nc{\bx}{{\mathbf{x}}}
\nc{\bs}{{\mathbf{s}}}
\nc{\by}{{\mathbf{y}}}
\nc{\bw}{{\mathbf{w}}}
\nc{\bA}{{\mathbf{A}}}
\nc{\bK}{{\mathbf{K}}}
\nc{\bB}{{\mathbf{B}}}
\nc{\bC}{{\mathbf{C}}}
\nc{\bG}{{\mathbf{G}}}
\nc{\bD}{{\mathbf{D}}}
\nc{\bH}{{\mathbf{H}}}
\nc{\bM}{{\mathbf{M}}}
\nc{\bN}{{\mathbf{N}}}
\nc{\bV}{{\mathbf{V}}}
\nc{\bW}{{\mathbf{W}}}
\nc{\bX}{{\mathbf{X}}}
\nc{\bZ}{{\mathbf{Z}}}
\nc{\bS}{{\mathbf{S}}}

\nc{\sA}{{\mathsf{A}}}
\nc{\sB}{{\mathsf{B}}}
\nc{\sC}{{\mathsf{C}}}
\nc{\sD}{{\mathsf{D}}}
\nc{\sF}{{\mathsf{F}}}
\nc{\sG}{{\mathsf{G}}}
\nc{\sK}{{\mathsf{K}}}
\nc{\sM}{{\mathsf{M}}}
\nc{\sO}{{\mathsf{O}}}
\nc{\sW}{{\mathsf{W}}}
\nc{\sQ}{{\mathsf{Q}}}
\nc{\sP}{{\mathsf{P}}}
\nc{\sV}{{\mathsf{V}}}
\nc{\sS}{{\mathsf{S}}}
\nc{\sT}{{\mathsf{T}}}
\nc{\sZ}{{\mathsf{Z}}}
\nc{\sfp}{{\mathsf{p}}}
\nc{\sll}{{\mathsf{l}}}
\nc{\sr}{{\mathsf{r}}}
\nc{\bk}{{\mathsf{k}}}
\nc{\sg}{{\mathsf{g}}}

\nc{\sff}{{\mathsf{f}}}
\nc{\sfb}{{\mathsf{b}}}
\nc{\sfc}{{\mathsf{c}}}
\nc{\sd}{{\mathsf{d}}}
\nc{\se}{{\mathsf{e}}}

\nc{\BK}{{\bar{K}}}

\nc{\tA}{{\widetilde{\mathbf{A}}}}
\nc{\tB}{{\widetilde{\mathcal{B}}}}
\nc{\tg}{{\widetilde{\mathfrak{g}}}}
\nc{\tG}{{\widetilde{G}}}
\nc{\TM}{{\widetilde{\mathbb{M}}}{}}
\nc{\tO}{{\widetilde{\mathsf{O}}}{}}
\nc{\tU}{{\widetilde{\mathfrak{U}}}{}}
\nc{\TZ}{{\tilde{Z}}}
\nc{\tx}{{\tilde{x}}}
\nc{\tbv}{{\tilde{\bv}}}
\nc{\tfP}{{\widetilde{\mathfrak{P}}}{}}
\nc{\tz}{{\tilde{\zeta}}}
\nc{\tmu}{{\tilde{\mu}}}

\nc{\urho}{\underline{\rho}}
\nc{\uB}{\underline{B}}
\nc{\uC}{{\underline{\mathbb{C}}}}
\nc{\ui}{\underline{i}}
\nc{\uj}{\underline{j}}
\nc{\ofP}{{\overline{\mathfrak{P}}}}
\nc{\oB}{{\overline{\mathcal{B}}}}
\nc{\og}{{\overline{\mathfrak{g}}}}
\nc{\oI}{{\overline{I}}}

\nc{\eps}{\varepsilon}
\nc{\hrho}{{\hat{\rho}}}

\nc{\one}{{\mathbf{1}}}
\nc{\two}{{\mathbf{t}}}

\nc{\Rep}{{\mathop{\operatorname{\rm Rep}}}}
\nc{\Tot}{{\mathop{\operatorname{\rm Tot}}}}
\nc{\Ker}{{\mathop{\operatorname{\rm Ker}}}}
\nc{\Hilb}{{\mathop{\operatorname{\rm Hilb}}}}
\nc{\End}{{\mathop{\operatorname{\rm End}}}}
\nc{\Ext}{{\mathop{\operatorname{\rm Ext}}}}
\nc{\CHom}{{\mathop{\operatorname{{\mathcal{H}}\it om}}}}
\nc{\GL}{{\mathop{\operatorname{\rm GL}}}}
\nc{\gr}{{\mathop{\operatorname{\rm gr}}}}
\nc{\Id}{{\mathop{\operatorname{\rm Id}}}}
\nc{\de}{{\mathop{\operatorname{\rm def}}}}
\nc{\length}{{\mathop{\operatorname{\rm length}}}}
\nc{\supp}{{\mathop{\operatorname{\rm supp}}}}

\nc{\Cliff}{{\mathsf{Cliff}}}
\nc{\Fl}{\on{Fl}}
\nc{\Fib}{{\mathsf{Fib}}}
\nc{\Coh}{{\on{Coh}}}
\nc{\QCoh}{{\on{QCoh}}}
\nc{\IndCoh}{{\on{IndCoh}}}
\nc{\FCoh}{{\mathsf{FCoh}}}

\nc{\reg}{{\text{\rm reg}}}

\nc{\cplus}{{\mathbf{C}_+}}
\nc{\cminus}{{\mathbf{C}_-}}
\nc{\cthree}{{\mathbf{C}_*}}
\nc{\Qbar}{{\bar{Q}}}
\nc\Eis{\on{Eis}}
\nc\Eisb{\ol\Eis{}}
\nc\Eisr{\on{Eis}^{rat}{}}
\nc\wh{\widehat}
\nc{\Def}{\on{Def_{\check{\fb}}(E)}}
\nc{\barZ}{\overline{Z}{}}
\nc{\barbarZ}{\overline{\barZ}{}}
\nc{\barpi}{\overline\pi}
\nc{\barbarpi}{\overline\barpi}
\nc{\barpip}{\overline\pi{}^+}
\nc{\barpim}{\overline\pi{}^-}

\nc{\fq}{\mathfrak q}

\nc{\fqb}{\ol{\fq}{}}
\nc{\fpb}{\ol{\fp}{}}
\nc{\fpr}{{\fp^{rat}}{}}
\nc{\fqr}{{\fq^{rat}}{}}

\nc{\hattimes}{\wh\otimes}

\nc{\bh}{{\bar{h}}}
\nc{\bOmega}{{\overline{\Omega(\check \fn)}}}

\nc{\seq}[1]{\stackrel{#1}{\sim}}

\nc{\cT}{{\check{T}}}
\nc{\cG}{{\check{G}}}
\nc{\cM}{{\check{M}}}
\nc{\cB}{{\check{B}}}

\nc{\ct}{{\check{\mathfrak t}}}
\nc{\cg}{{\check{\fg}}}
\nc{\cb}{{\check{\fb}}}
\nc{\cn}{{\check{\fn}}}

\nc{\cLambda}{{\check\Lambda}}

\nc{\cla}{{\check\lambda}}
\nc{\cmu}{{\check\mu}}
\nc{\cnu}{{\check\nu}}
\nc{\ceta}{{\check\eta}}

\nc{\DefbE}{{\on{Def}_{\cB}(E_\cT)}}

\nc{\imathb}{{\ol{\imath}}}
\nc{\rlr}{\overset{\longrightarrow}{\underset{\longrightarrow}\longleftarrow}}

\nc{\oBun}{\overset{\circ}\Bun}
\nc{\LocSys}{\on{LocSys}}
\nc{\BunBbb}{\ol{\ol{Bun}}_B}
\nc{\BunBr}{\Bun_B^{rat}}
\nc{\BunBrsg}{\Bun_B^{rat,\on{s.g.}}}
\nc{\BunBrp}{\Bun_B^{rat,polar}}
\nc{\BunBrpbg}{\Bun_B^{rat,polar,\on{b.g.}}}
\nc{\BunBrpsg}{\Bun_B^{rat,polar,\on{s.g.}}}
\nc{\BunTrp}{\Bun_T^{rat,polar}}
\nc{\BunTrpbg}{\Bun_T^{rat,polar,\on{b.g.}}}
\nc{\BunTrpsg}{\Bun_T^{rat,polar,\on{s.g.}}}
\nc{\BunNr}{\Bun_N^{rat}}
\nc{\BunNre}{\Bun_N^{enh,rat}}
\nc{\BunTr}{\Bun_T^{rat}}
\nc{\Vect}{\on{Vect}}
\nc{\Whit}{\on{Whit}}
\nc{\CTb}{\ol{\on{CT}}}
\nc{\Ran}{\on{Ran}}
\nc{\CTr}{\on{CT}^{rat}{}}
\nc\jmathr{\jmath^{rat}{}}
\nc{\ux}{\underline{x}}
\nc{\clambda}{{\check\lambda}}
\nc{\calpha}{{\check\alpha}}
\nc{\ind}{{\mathbf{ind}}}
\nc{\oblv}{{\mathbf{oblv}}}
\nc{\StinftyCat}{\on{DGCat}}
\nc{\inftygroup}{\infty\on{-Grpd}}

\nc{\fset}{\on{fSet}}
\nc{\Sch}{\on{Sch}}
\nc{\affSch}{\on{Sch}^{\on{aff}}}
\nc{\indSch}{\on{IndSch}}
\nc{\ul}{\underline}
\nc{\dr}{\on{dR}}
\nc{\rendr}{\on{ren-dR}}
\nc{\bDelta}{\mathbf\Delta}

\begin{document}

\title[Contractibility of the space of rational maps]
{Contractibility of the space of rational maps}

\author{Dennis Gaitsgory}

\dedicatory{For Sasha Beilinson}

\date{\today}

\begin{abstract}
We define an algebro-geometric model for the space of rational maps from a smooth curve $X$
to an algebraic group $G$, and show that this space is homologically contractible. As a consequence,
we deduce that the moduli space $\Bun_G$ of $G$-bundles on $X$ is unformized by the appropriate
rational version of the affine Grassmannian, where the uniformizing map has contractible fibers. 
\end{abstract}

\maketitle

\tableofcontents

\section*{Introduction}

\ssec{The origins of the problem}

Let $X$ be a smooth, connected and complete curve, and $G$ a reductive group (over an algebraically closed 
field $k$ of characteristic $0$). A fundamental object of study in the Geometric Langlands program is the 
moduli stack of $G$-bundles on $X$, denoted $\Bun_G$.

\medskip

This paper arose from the desire to approximate $\Bun_G$ by its local cousin, namely the \emph{adelic} Grassmannian
$\Gr$. 

\sssec{}

On the first pass, we shall loosely understand $\Gr$ as the moduli space of $G$-bundles on $X$
equipped with a \emph{rational} trivialization. The local nature of $\Gr$ expresses itself in that if we specify
the locus $\oX\subset X$ over which our trivialization is regular, the corresponding subspace of $\Gr$ is
a product of copies of the affine Grassmannians $\Gr_x$ over the missing point $x\in X-\oX$, see \cite{BD2}, 4.3.14.
So, one perceives $\Gr$ as an inherently simpler object than $\Bun_G$. 

\medskip

In fact, one can think of $\Bun_G$ as the quotient of $\Gr$ by the group of rational maps 
from $X$ to $G$, denoted ${\mathbf{Maps}}(X,G)^{\on{rat}}$. 

\medskip

So, for example, line bundles on $\Bun_G$ can be thought of as line bundles on $\Gr$ equipped
with the data of equivariance with respect to the group ${\mathbf{Maps}}(X,G)^{\on{rat}}$. However,
the following crucial observation was made in \cite{BD2}: 

\smallskip

If $\CY$ is a space acted on by
${\mathbf{Maps}}(X,G)^{\on{rat}}$, and $\CL$ is a line bundle on $\CY$, then $\CL$ has a unique
equivariant structure with respect to  ${\mathbf{Maps}}(X,G)^{\on{rat}}$.

\medskip

The last observation leads one to think that the group ${\mathbf{Maps}}(X,G)^{\on{rat}}$, although
wildly infinite-dimensional, for some purposes behaves like the point-scheme. 
\footnote{The author learned this idea from conversations with A.~Beilinson, who in turn
attributed it to Drinfeld.} E.g., in
\cite{BD2}, 4.3.13 it is shown that any function on ${\mathbf{Maps}}(X,G)^{\on{rat}}$ is constant.

\sssec{}

In this paper we take up the task of establishing some of the point-like properties of 
${\mathbf{Maps}}(X,G)^{\on{rat}}$. But we will go in a direction, slightly different than the one mentioned
above. Namely, we will be interested not so much in line bundles or other quasi-coherent sheaves on 
${\mathbf{Maps}}(X,G)^{\on{rat}}$  (and, respectively, $\Bun_G$ and $\Gr$), but in D-modules.

\medskip

Namely, we will show that ${\mathbf{Maps}}(X,G)^{\on{rat}}$ is \emph{contractible}. A concise way
to formulate this is by saying that the cohomology with compact supports of the dualizing sheaf
on ${\mathbf{Maps}}(X,G)^{\on{rat}}$ is isomorphic to scalars. 

\medskip

(The word ``contractibility" is somewhat of a misnomer: in this paper we only establish that 
${\mathbf{Maps}}(X,G)^{\on{rat}}$
is homologically contractible. Full contractibilty should also include the statement that every local
system on ${\mathbf{Maps}}(X,G)^{\on{rat}}$ is trivial, see Remark \ref{r:full contr}.
The latter can also be proved, but we will not need it for the applications we have in mind in this paper.)

\ssec{Uniformization of $\Bun_G$}

We can regard the map $\pi:\Gr\to \Bun_G$ as an instance of uniformization of
a ``global" object by a ``local" one. Thus, we obtain that not only does $\Gr$ uniformize $\Bun_G$, 
but it  does so with contractible fibers. 

\medskip

As an almost immediate consequence of the contractibility of the fibers of $\pi$, we obtain that the
pullback functor from the category of D-modules on $\Bun_G$ to that on $\Gr$ is
fully faithful. The latter fact itself has multiple consequences.

\sssec{}

First, we recall that $\Gr$ is ind-proper. We will show that this implies that $\Bun_G$
also exhibits properties of proper schemes with respect to cohomology of D-modules
and quasi-coherent sheaves:  in either context we will show that coherent objects
get sent by the cohomology functor to finite-dimensional vector spaces.

\medskip

Moreover, we will show that this is true not only for $\Bun_G$ itself, but for a certain
family of its open substacks, introduced in \cite{DrGa1}.

\sssec{}

Second, we will show that if one considers the category of D-modules on $\Bun_G$, 
\emph{equivariant} with respect to the Hecke groupoid, then this category is
equivalent to $\Vect$.

\sssec{}

Third, the contractibility of the fibers of $\pi$ implies that the two spaces
have isomorphic cohomology. Using this fact, in \secref{s:Atiyah-Bott} we will show how
this allows one to re-derive the Atiyah-Bott formula for the cohomology of $\Bun_G$. 

\medskip

Furthermore, in the forthcoming work \cite{GaLu} it will be shown that the same game
can be played when $X$ is a curve over a finite field and D-modules are replaced
by $\ell$-adic sheaves. Then the expression for the trace of Frobenius on the homology
of $\Bun_G$ in terms of that on $\Gr$ leads to a geometric proof of Weil's conjecture
on the Tamagawa number of the automorphic space corresponding to $G$ and
the global field of rational functions on $X$.

\ssec{Uniformization of $\Bun_G$ in differential geometry and topology}

\sssec{} \label{sss:cl AB}

We should mention a curious discrepancy between the above way to calculate the cohomology
of $\Bun_G$, and the classical differential-geometric method (see, however, \secref{sss:non-ab}
below).

\medskip

Namely, in differential geometry (for $G$ semi-simple and simply-connected), one uniformizes
the analytic space underlying $\Bun_G$ by the space $\on{Conn}_{\ol\partial}(\CP^0_{C^\infty})$ of 
complex structures on the trivial
($C^\infty$) $G$-bundle $\CP^0_{C^\infty}$ on $X$. The space $\Bun_G$ is the quotient of
$\on{Conn}_{\ol\partial}(\CP^0_{C^\infty})$ by the group of gauge transformations
${\mathbf{Maps}}_{C^\infty}(X,G)$. 

\medskip

The space $\on{Conn}_{\ol\partial}(\CP^0_{C^\infty})$ is contractible, and hence the
cohomology of $\Bun_G$ is isomorphic to that of the classifying space 
$B{\mathbf{Maps}}_{C^\infty}(X,G)$ of ${\mathbf{Maps}}_{C^\infty}(X,G)$. 

\medskip

So, in differential geometry, the uniformizing space is topologically trivial, and the homotopy type
of $\Bun_G$ is encoded by the structure group. 

\medskip

By contrast, in algebraic geometry, the structure group is contractible,
and the homotopy type of $\Bun_G$ is encoded by the uniformizing space.

\sssec{}  \label{sss:non-ab}

A closer analogy to the uniformization of $\Bun_G$ by $\Gr$ is provided by Lurie's non-abelian
Poincar\'e duality, see \cite{Lu1}, Sect. 5.3.6.

\medskip

Taking our source to be the Riemann surface corresponding to $X$, and our target
the topological space $BG$, the assertion of {\it loc.cit.} says that the homotopy type of
${\mathbf{Maps}}_{C^\infty}(X,BG)$ is isomorphic to the ``chiral homology" of a certain factorization
algebra on $X$, whose fiber at a finite collection of points $x_1,...,x_n\in X$ is the
space of maps from $\bB_{x_1}\times...\times \bB_{x_n}$ to $BG$, trivialized
on $\bS_{x_1}\times...\times \bS_{x_n}$, where $\bB_{x}$ denoted a ball in $X$
around $x$, and $\bS_x$ is its boundary.

\medskip

Now, by \secref{sss:cl AB}, the homotopy type of $\Bun_G$ is isomorphic to that
of $B{\mathbf{Maps}}_{C^\infty}(X,G)$, which can be shown to be isomorphic to 
$${\mathbf{Maps}}_{C^\infty}(X,BG).$$ 
Further, the homotopy type of
${\mathbf{Maps}}_{C^\infty}((\bB_x;\bS_x),BG)$ is isomorphic to the fiber of $\Gr$
over $x$, denoted $\Gr_{x}$.

\medskip

Thus, non-abelian Poincar\'e duality gives another way to obtain an expression
for the homology of $\Bun_G$. In a sense, what we do in this paper is to show
that this picture goes through also in the context of algebraic geometry.

\begin{remark}
As we have just seen, the homotopy type of (the analytic space underlying) $\Bun_G$
can be realized as the space of $C^\infty$ maps from $X$ to $BG$. We should
emphasize that this is not a general fact, but rather specific to our target space being
$BG$. One cannot expect a close relation between spaces of algebraic and $C^\infty$
maps from $X$ to general targets. 
\end{remark}

\ssec{D-modules on ``spaces"}

In order to proceed to a more rigorous discussion, we need to address the following question: 
what do we mean by ``the space of $G$-bundles equipped with a rational trivialization" and ``the space of
rational maps from $X$ to $G$"?

\sssec{}

First, what do we mean by a ``space"? By definition, a ``space" is
a prestack, which in this paper is an arbitrary functor from the category affine schemes of finite type 
over $k$ to the category of $\infty$-groupoids, denoted $\inftygroup$. 
 \footnote{By a prestack we shall always understand a 
``higher" prestack, i.e., one taking values in  $\infty$-groupoids, rather than ordinary groupoids.} 
We shall denote this category by $\on{PreStk}$. 

\sssec{}   \label{sss:D-modules}

As is explained in \cite{Crystals}, there is a canonically defined functor
$$\fD^!_{\affSch}:(\affSch)^{op}\to \StinftyCat,$$
that attaches to an affine scheme $S$ of finite type the DG category $\fD(S)$ of D-modules on $S$,
and to a map $f:S_1\to S_2$ the functor
$$f^!:\fD(S_2)\to \fD(S_1).$$

\medskip

This functor automatically gives rise to a functor $\fD^!_{\on{PreStk}}:\on{PreStk}^{op}\to \StinftyCat$,
by the procedure of right Kan extension, see \cite[Sect. 2.3]{Crystals}. Namely,
$$\fD(\CY):=\underset{S\in (\affSch_{/\CY})^{op}}{lim}\, \fD(S),$$
where the limit is taken with respect to the $!$-pullback functors. 

\medskip

Informally, for $\CY\in \on{PreStk}$, the data of an object $\CM\in \fD(\CY)$ is that of 
a family of objects $\CM_S\in \fD(S)$ for every affine scheme $S$ endowed with a map 
$S\overset{\phi}\to\CY$, and a homotopy-coherent
system of isomorphisms $f^!(\CM_S)\simeq \CM_{S'}$ for every $f:S'\to S$ and an isomorphism between
the resulting maps
$$S'\overset{\phi'}\to \CY \text{ and } S'\overset{f}\to S\overset{\phi}\to\CY.$$

\medskip

In particular, for a morphism $f:\CY_1\to \CY_2$ there is a well-defined functor
$$f^!:\fD(\CY_2)\to \fD(\CY_1).$$

\medskip

When $\CY$ is a non-affine scheme, the Zariski descent property of the category of D-modules
implies that the category $\fD(\CY)$ defines above recovers the usual DG category of D-modules
on $\CY$.



\sssec{}  \label{sss:defn of contr}

For example, for any $\CY$, the category $\fD(\CY)$ contains a canonical object, denoted $\omega_\CY$,
referred to as the dualizing sheaf on $\CY$, and defined as $p_\CY^!(k)$, where $p_\CY$ is the 
map $\CY\to \on{pt}$, and $k$ is the generator of $\fD(\on{pt})\simeq \Vect$. 

\medskip

We shall say that $\CY$ is homologically contractible if the functor
\begin{equation} \label{e:pullback omega}
\Vect\to \fD(\CY),\,\,V\mapsto V\otimes \omega_\CY
\end{equation}
is fully faithful.

\medskip

The category $\fD(\CY)$ contains a full subcategory $\fD(\CY)_{\on{loc.const}}$ consisting of
objects, whose pullback on any affine scheme $S$ mapping to $\CY$ belongs to the full subcategory
of $\fD(S)$ generated by Verdier duals of lisse D-modules with regular singularities.  

\medskip

We shall say that $\CY$ is contractible if the functor \eqref{e:pullback omega} is an equivalence onto
$\fD(\CY)_{\on{loc.const}}$.

\sssec{}

When our ground field $k$ is $\BC$, there exists a canonically defined functor
$$\CY\mapsto \CY(\BC)^{\on{top}}:\on{PreStk}\to \inftygroup,$$
where $\inftygroup$ is thought of as the category of homotopy types.

\medskip

Namely, the above functor is the \emph{left} Kan extension of the functor
$$\Sch\to \inftygroup,$$
that assigns to $S$ the homotopy type of the analytic space underlying $S(\BC)$.

\medskip

One can show that $\CY$ is homologically contractible (resp., contractible)
in the sense of \secref{sss:defn of contr} if and only if $H_{\bullet}(\CY(\BC)^{\on{top}},\BQ)\simeq \BQ$
(resp., if the rational homotopy type of $\CY(\BC)^{\on{top}}$ is trivial).

\sssec{}

It is no surprise that unless we impose some 
additional conditions on $\CY$, the category $\fD(\CY)$ will be rather intractable:

\medskip

The closer a prestack $\CY$ is to being a scheme, the more manageable the category $\fD(\CY)$ is.
What is even more important is the functorial properties of $\fD(-)$ with respect to morphisms
$g:\CY_1\to \CY_2$. I.e., the closer a given morphism is given to being schematic, the more 
we can say about the behavior of the direct image functor.

\medskip

One class of prestacks for which the category $\fD(\CY)$ is really close to the case of schemes
is that of indschemes. By definition, an indscheme is a prestack that can be exhibited as a 
\emph{filtered} direct limit of prestacks representable by schemes with transition maps 
being closed embeddings. 

\medskip

A wider class is that of pseudo-indschemes, where the essential difference is that we
omit the filteredness condition. We discuss this notion in some detail in \secref{ss:pseudo}.

\medskip

On a technical note, we should remark that ``filtered'' vs. ``non-filtered" makes a huge difference.
For example, an indscheme perceived as a functor $\affSch\to \inftygroup$ takes values in
$0$-groupoids, i.e., $\on{Set}\subset \inftygroup$, whereas this is no longer true for
pseudo-indschemes. We should also add it is crucial that we consider our functors with
values in $\inftygroup$ and not truncate them to ordinary groupoids or sets: this is necessary
to obtain reasonably behaved (derived) categories of D-modules.

\ssec{Spaces of rational maps and the Ran space}

\sssec{}  \label{sss:dom approach}

The first, and perhaps, geometrically the most natural, definition of the prestack corresponding to
$\Gr$ and ${\mathbf{Maps}}(X,G)^{\on{rat}}$ is the following:

\medskip

For $\Gr$, to an affine scheme $S$ we attach the set of pairs $(\CP,\alpha)$, where $\CP$
is a $G$-bundle on $S\times X$, and $\alpha$ is a trivialization of $\CP$ defined over
an open subset $U\subset S\times X$, whose intersection with every geometric fiber of
$S\times X$ over $S$ is non-empty.

\medskip

For ${\mathbf{Maps}}(X,G)^{\on{rat}}$, to an affine scheme $S$ we attach the set of maps
$m:U\to G$, where $U$ is as above.

\medskip

The problem is, however, that the above prestacks are not indschemes. Neither is it
clear that they are pseudo-indschemes. So, although the categories of D-modules on the above
spaces are well-defined, we do not \emph{a priori} know how to make 
calculations in them,
and in particular, how to prove the contractbility of ${\mathbf{Maps}}(X,G)^{\on{rat}}$
(see, however, \secref{sss:back to rational} below). 

\sssec{}

An alternative device to handle spaces such as $\Gr$ or ${\mathbf{Maps}}(X,G)^{\on{rat}}$ 
has been suggested in \cite{BD1}, and is known as the Ran space. 

\medskip

The underlying idea is that instead of just talking about \emph{something}
being rational, we specify the finite collection of points of $X$, outside of
which this \emph{something} is regular.

\sssec{}

The basic object is the Ran space itself, denoted $\Ran X$, which is defined
as follows. By definition, $\Ran X$ is the colimit (a.k.a. direct limit) in $\on{PreStk}$
of the diagram of schemes $I\mapsto X^I$, where $I$ runs through the category 
$(\fset)^{op}$ opposite to that of finite sets and surjective maps between them. 

\medskip

Similarly, for an arbitrary target scheme $Y$, one defines the space 
${\mathbf{Maps}}(X,Y)^{\on{rat}}_{\Ran X}$ as the colimit, over the same index
category, of the spaces ${\mathbf{Maps}}(X,Y)^{\on{rat}}_{X^I}$, where the
latter is the prestack that assigns to a test scheme $S$ the data of an $S$-point
$x^I$ of $X^I$ and a map
$$m:(S\times X-\{x^I\})\to Y,$$
where $\{x^I\}$ denotes the incidence divisor in $S\times X$. 

\medskip

The point is that if $Y$ is an affine scheme, each ${\mathbf{Maps}}(X,Y)^{\on{rat}}_{X^I}$ is an indscheme,
and hence ${\mathbf{Maps}}(X,Y)^{\on{rat}}_{\Ran X}$ is a pseudo-indscheme.
\footnote{We are grateful to J.~Barlev for explaining this point of view
to us: that instead of the functor $I\mapsto {\mathbf{Maps}}(X,Y)^{\on{rat}}_{X^I}$,
for most practical purposes it is more efficient to just remember its colimit, i.e.,
${\mathbf{Maps}}(X,Y)^{\on{rat}}_{\Ran X}$.}

\medskip

In a similar fashion, we define the space $\Gr_{\Ran X}$ as the colimit over
the same index category $(\fset)^{op}$ of the indschemes 
$\Gr_{X^I}$, where the latter is the moduli space of the data $(x^I,\CP,\alpha)$
where $x^I$ is as above, $\CP$ is a principal $G$-bundle on $S\times X$, and
$\alpha$ is a trivialization of $\CP$ on $S\times X-\{x^I\}$.

\medskip

Although, as we emphasized above, in the formation of the spaces $\Ran X$,
${\mathbf{Maps}}(X,Y)^{\on{rat}}_{\Ran X}$, and $\Gr_{\Ran X}$, we must take the
colimit in the category $\inftygroup$, i.e., a priori, the value of our functor
on a test affine scheme will be an infinity-groupoid, it will turn out that in the
above cases, our functors take values in $\on{Set}\subset \inftygroup$.

\sssec{}

It is with the above version of the space of rational maps given by 
${\mathbf{Maps}}(X,G)^{\on{rat}}_{\Ran X}$ that we prove its homological contractibility.
The main result of this paper reads:

\begin{thm} \label{t:main preview}
Let $Y$ be an affine scheme, which is connected and can be covered by open subsets, 
each which is isomorphic to an open subset of the affine space $\BA^n$ for some 
fixed $n$. Then the space ${\mathbf{Maps}}(X,Y)^{\on{rat}}_{\Ran X}$ is
homologically contractible.
\end{thm}

\begin{remark} \label{r:full contr}
Although we do not prove it in this paper, as was mentioned earlier, one can show that
under the same assumptions on $Y$, the space ${\mathbf{Maps}}(X,Y)^{\on{rat}}_{\Ran X}$
is actually contractible in the sense of \secref{sss:defn of contr}. Moreover, when $k=\BC$,
one can show that the resulting homotopy is actually trivial, and not just rationally
trivial. 
\end{remark}

\medskip

It is tempting to conjecture that the corresponding facts (homological contractibility,
contractibility, and the triviality of homotopy type for $k=\BC$)
hold more generally: i.e.,
that it sufficient to require that $Y$ be smooth and birationally equivalent to $\BA^n$. 

\sssec{}

Let us briefly indicate the strategy of the proof of \thmref{t:main preview}. 

\medskip

First, we consider the
case when $Y=\BA^n$. Then the theorem is proved by a direct calculation: the corresponding space 
${\mathbf{Maps}}(X,Y)^{\on{rat}}_{\Ran X}$ is essentially comprised of affine spaces.

\medskip

We then consider the case when $Y$ can be realized as an open subset of $\BA^n$. We show that
the corresponding map
$${\mathbf{Maps}}(X,Y)^{\on{rat}}_{\Ran X}\to {\mathbf{Maps}}(X,\BA^n)^{\on{rat}}_{\Ran X}$$
is a homological equivalence. The idea is that the complement is ``of infinite codimension".

\medskip

Finally, we show that if $U_\alpha$ is a Zariski cover of $Y$, and $U_\bullet$ is its \v{C}ech nerve,
then the map 
$$|{\mathbf{Maps}}(X,U_\bullet)^{\on{rat}}_{\Ran X}|\to {\mathbf{Maps}}(X,Y)^{\on{rat}}_{\Ran X}$$
is also a homological equivalence (here $|-|$ denotes the functor of geometric realization of a 
simplicial object). 

\medskip

Since, by the above, each term in ${\mathbf{Maps}}(X,U_\bullet)^{\on{rat}}_{\Ran X}$ is 
homologically contractible, we deduce the corresponding fact for ${\mathbf{Maps}}(X,Y)^{\on{rat}}_{\Ran X}$.

\ssec{Should $\Ran X$ appear in the story?}

\sssec{}

Let us note that the appearance of the Ran space allows us to
connect spaces such as $\Gr_{\Ran X}$ to chiral/factorization algebras
of \cite{BD1} (see also \cite{FrGa}, where the derived version of 
chiral/factorization algebras is discussed in detail).

\medskip

For example, the direct image of the dualizing sheaf under the forgetful
map $$\Gr_{\Ran X}\to \Ran X$$ is a factorization algebra. It is the latter
fact that allows one to connect the cohomology of $\Gr_{\Ran X}$, and hence 
of $\Bun_G$, with the Atiyah-Bott formula. 

\medskip

And in general, the factorization property of $\Gr$ provides a convenient
tool of interpreting various cohomological questions on $\Bun_G$ as 
calculation of chiral homology of chiral algebras. 

\medskip

In particular, in a subsequent paper we will show how this approach allows
to prove that chiral homology of the integrable quotient of the
Kac-Moody chiral algebra 
 is isomorphic to the dual vector space to that of the cohomology
of the corresponding line bundle on $\Bun_G$. 

\medskip

The (dual vector space of the) $0$-th chiral homology of the above chiral algebra is known under the
name of ``conformal blocks of the WZW model". The fact that WZW conformal blocks are isomorphic to 
the space of global sections of the corresponding line bundle on $\Bun_G$ is well-known. However, the
fact that the same isomorphism holds at the derived level has been a conjecture proposed in \cite{BD1},
Sect. 4.9.10; it was proved in {\it loc.cit.} for $G$ being a torus.

\medskip

So, in a sense it is ``a good thing" to have the Ran space appearing in the picture.

\sssec{}  

However, the Ran space also brings something that can be perceived as a handicap
of our approach. 

\medskip

Namely, consider for example the space of rational maps $X\to Y$, realized
${\mathbf{Maps}}(X,Y)^{\on{rat}}_{\Ran X}$. We see that along with the the data of a rational map, 
this space retains also the data of the locus outside
of which it is defined. For example, for $Y=\on{pt}$, whereas we would like the space of
rational maps $X\to \on{pt}$ to be isomorphic to $\on{pt}$, instead we obtain $\Ran X$.

\medskip

But the above problem can be remedied. Namely, ${\mathbf{Maps}}(X,Y)^{\on{rat}}_{\Ran X}$
comes with an additional structure given by the action of $\Ran X$, viewed as a semi-group in 
$\on{PreStk}$, given by enlarging the allowed locus of irregularity. Following \cite{Ran},
we call it ``the unital structure". 

\medskip

The action of $\Ran X$ on ${\mathbf{Maps}}(X,Y)^{\on{rat}}_{\Ran X}$ gives rise to a semi-simplicial
object
$$...\Ran X\times {\mathbf{Maps}}(X,Y)^{\on{rat}}_{\Ran X}\rightrightarrows {\mathbf{Maps}}(X,Y)^{\on{rat}}_{\Ran X}$$
of $\on{PreStk}$, and we let ${\mathbf{Maps}}(X,Y)^{\on{rat}}_{\Ran X,\on{indep}}$ denote its
geometric realization.

\medskip

The effect of the passage
$${\mathbf{Maps}}(X,Y)^{\on{rat}}_{\Ran X}\rightsquigarrow {\mathbf{Maps}}(X,Y)^{\on{rat}}_{\Ran X,\on{indep}}$$
is that one gets rid of the extra data of remembering the locus of irregularity.\footnote{Hence the notation
``indep" to express the fact the resulting notion of rational map is independent of the specification of the locus
of regularity.}

\medskip

In addition, in \secref{s:unital} we will show the forgetful functor 
$$\fD\left({\mathbf{Maps}}(X,Y)^{\on{rat}}_{\Ran X,\on{indep}}\right)\to 
\fD\left({\mathbf{Maps}}(X,Y)^{\on{rat}}_{\Ran X}\right)$$
is fully faithful, so questions such as contractibility for the two models
of the space of rational maps are equivalent.

\sssec{}  \label{sss:back to rational}

Finally, recall the space ${\mathbf{Maps}}(X,Y)^{\on{rat}}$ introduced in \secref{sss:dom approach}.
There exists a natural map
\begin{equation} \label{e:Ba}
{\mathbf{Maps}}(X,Y)^{\on{rat}}_{\Ran X,\on{indep}}\to {\mathbf{Maps}}(X,Y)^{\on{rat}}.
\end{equation}

The following will be shown in \cite{Ba}:

\begin{thm} \label{t:Ba}
The pullback functor 
$$\fD\left({\mathbf{Maps}}(X,Y)^{\on{rat}}\right)\to \fD\left({\mathbf{Maps}}(X,Y)^{\on{rat}}_{\Ran X,\on{indep}}\right)$$
is an equivalence. In particular, if ${\mathbf{Maps}}(X,Y)^{\on{rat}}_{\Ran X}$ is homologically contractible,
then so is ${\mathbf{Maps}}(X,Y)^{\on{rat}}$. 
\end{thm}

In fact, the assertion of \thmref{t:Ba} is applicable not just to 
${\mathbf{Maps}}(X,Y)^{\on{rat}}$, but for a large class 
of similar problems. 

\medskip

So, although ${\mathbf{Maps}}(X,Y)^{\on{rat}}$ is not a priori a pseudo-indscheme, the category 
of D-modules on it is, after all, manageable, and in particular, assuming \thmref{t:Ba},
one can prove its homological contractibility as well. 

\ssec{Conventions and notation} \label{ss:conventions}

\sssec{}

Our conventions regarding $\infty$-categories follow those of \cite{FrGa}. Throughout the text,
whenever we say ``category", by default we shall mean an $\infty$-category. 

\sssec{}  \label{sss:dg categories}

The conventions regarding DG categories follow those of \cite{DG}. 
(Since \cite{DG} is only a survey, for a better documented theory, one can replace
the $\infty$-category of DG categories by an equivalent one of 
of stable $\infty$-categories
tensored over $k$, which has been developed by J.~Lurie in \cite{Lu1}.)

\medskip

Throughout this paper we shall be working in the category denoted in \cite{DG}
by $\StinftyCat_{\on{cont}}$: namely, all our DG categories will be assumed
presentable, and in particular, cocomplete (i.e., closed under arbitrary direct sums). Unless explicitly said
otherwise, all our functors will be assumed continuous, i.e., 
commute with arbitrary direct sums (or, equivalently, all colimits). 

\medskip

By $\Vect$ we shall denote the DG category of complexes of $k$-vector spaces.

\sssec{Schemes and prestacks}

Our general reference for prestacks is \cite{Stacks}. 

\medskip

However, the ``simplifying good news" for this paper is that, since we will only be interested in D-modules,
we can stay within the world of classical, i.e., non-derived, algebraic geometry.

\medskip

Throughout the paper, we shall only be working with schemes of finite type over $k$
(in particular, quasi-compact). We denote this category by $\Sch$. By $\affSch$
we denote the full subcategory of affine schemes.

\medskip

Thus, the category of prestacks, denoted in this paper by $\on{PreStk}$ is what is denoted
in \cite{Stacks}, Sect. 1.3 by $^{\leq 0}\!\on{PreStk}_{\on{laft}}$. 

\sssec{D-modules} 

In \cite{FrGa}, Sect. 1.4, it was explained what formalism of D-modules on schemes is required to
have a theory suitable for applications. Namely, we needed $\fD(-)$ to be a functor
$\Sch_{\on{corr}}\to \StinftyCat$, where $\Sch_{\on{corr}}$ is the category whose objects
are schemes and morphisms are correspondences. 

\medskip

Fortunately, for this paper, a more restricted formalism will suffice. Namely, we will only need the functor
$$\fD^!_{\Sch}:\Sch^{op}\to \StinftyCat.$$
I.e., we will only need to consider the !-pullback functor for morphisms. This formalism is developed
in the paper \cite{Crystals}. As was mentioned in \secref{sss:D-modules}, the functor $\fD^!_{\Sch}$
extends to a functor 
$$\fD^!_{\on{PreStk}}:\on{PreStk}\to \StinftyCat.$$

\medskip

An important technical observation is the following: let $\CY_1$ and $\CY_2$
be prestacks, such that the category $\fD(\CY_1)$ is dualizable. Then the
natural functor
\begin{equation} \label{e:prod equivalence}
\fD(\CY_1)\otimes \fD(\CY_2)\to \fD(\CY_1\times \CY_2)
\end{equation}
is an equivalence. The proof is a word-for-word repetition of the argument
in \cite{DrGa0}, Proposition 1.4.4.

\ssec{Acknowledgements}

The idea of contractibility established in this paper, as well as that the homology of rational maps 
should be insensitive to removing from the target subvarieties of positive codimension is due
to V.~Drinfeld. 

\medskip

The author would like to express his gratitude to A.~Beilinson, for teaching him most of
the mathematics that went into this paper (the above ideas of Drinfeld's, the affine Grassmannian,
the Ran space, chiral homology, etc.)

\medskip

This paper is part of a joint project with J.~Lurie, which would hopefully see light
soon in the form of \cite{GaLu}. The author is grateful to Jacob for offering his
insight in the many discussions on the subject that we have had. 

\medskip

Special thanks are due to S.~Raskin for extensive comments on the
first draft of the paper, which have greatly helped to improve the exposition.

\medskip

We would also like to thank J.~Barlev and N.~Rozenblyum for numerous
helpful conversations related to the subject of the paper.

\medskip

The author's research is supported by NSF grant DMS-1063470.
 
\section{The ``space" of rational maps as a pseudo-indscheme}

\ssec{Indschemes}

\sssec{}

Before we define pseudo-indschemes, let us recall the notion of usual indscheme
(see \cite{IndSch} for a detailed discussion). By definition,
an indscheme is an object $\CY\in \on{PreStk}$ which can be written as a 
\begin{equation} \label{e:presentation of indscheme}
\underset{I\in \CS}{colim}\, Z(I),
\end{equation}
for some functor $Z:\CS\to \Sch$, such that
\begin{itemize}

\item For $(I_1\to I_2)$, the map $Z(I_1)\to Z(I_2)$ is a closed embedding,

\item The category \footnote{Although this will be of no practical consequence,
let us note that since we are working with classical schemes, rather than derived ones, $\Sch$ is an ordinary
category, and so, we can take the category of indices $\CS$ to be an ordinary category as well.} $\CS$ of indices is \emph{filtered}.
\end{itemize}

In the above formula the colimit is taken in $\on{PreStk}$. Recall
that colimits in $\on{PreStk}$ are computed object-wise, i.e., for $S\in \affSch$,
\begin{equation} \label{e:maps into indscheme}
\on{Maps}(S,\CY)\simeq \underset{I\in \CS}{colim}\, \on{Maps}(S,Z(I)),
\end{equation}
where the latter colimit is taken in $\inftygroup$.

\medskip

We let $\indSch$ denote the full subcategory of $\on{PreStk}$ spanned by indschemes. 

\begin{remark}

The condition that $\CS$ be filtered is really crucial, as it insures some of the key properties of
indschemes. For example:

\medskip

\noindent(i) As long as we stay in the realm of classical (i.e., non-derived) algebraic geometry,
the functor $S\mapsto \on{Maps}(S,\CY)$ takes value in sets, rather than $\inftygroup$.
Indeed, a filtered colimit of $k$-truncated groupoids is $k$-truncated (and we take $k=0$). 

\medskip

\noindent(ii) An indscheme belongs to $\on{Stk}$, i.e., it satisfies fppf descent. This is so because
all $Z(I)$ satisfy descent, and the fact that finite limits (that are involved in the formulation of descent)
commute with filtered colimits, see \cite[Lemma 1.1.5]{IndSch}. Therefore, formula 
\eqref{e:maps into indscheme} holds also for $S\in \Sch$. 

\medskip

These two properties will fail for pseudo-indschemes.

\end{remark}

\sssec{}

Recall that a map $\CY_1\to \CY_2$ in $\on{PreStk}$ is called a closed embedding if
for any $S\in \affSch_{/\CY_2}$, the resulting map $S\underset{\CY_2}\times \CY_1\to S$
is a closed embedding; in particular, the left hand side is a scheme. \footnote{We emphasize
that the above definition of closed embedding is only suitable for classical, i.e., non-derived,
algebraic geometry. By contrast, in the setting of DAG, ``closed embedding" means by definition ``a closed embedding at the level
of the underlying classical prestacks", so $S\underset{\CY_2}\times \CY_1$ does not have to
be a derived scheme, but its underlying classical prestack must be a classical scheme.}

\medskip

We have:

\begin{lem} \label{l:closed emb into indscheme}
Let $\CY$ be an indscheme written as in \eqref{e:presentation of indscheme}.
For $S\in \Sch$, a map $S\to \CY$ is a closed embedding if and only if for some/any
index $I$, for which the above map factors through $Z(I)$, the resulting map $S\to Z(I)$
is a closed embedding.
\end{lem}

\begin{proof}
Let $T$ be an affine scheme mapping to $\CY$. Let $I$ be an index such that both maps $S,T\to \CY$
factor through $Z(I)$. The filteredness assumption on $\CS$ implies that 
$$T\underset{\CY}\times S\simeq \underset{I'\in \CS_{I/}}{colim}\, (S\underset{Z(I')}\times T).$$
However (in classical algebraic geometry), since $Z(I)\to Z(I')$ is a closed embedding, the map
$$S\underset{Z(I)}\times T\to S\underset{Z(I')}\times T$$ is an isomorphism, so
$T\underset{\CY}\times S\simeq S\underset{Z(I)}\times T$,
and the assertion is manifest.
\end{proof}

Thus, we obtain that the notion of closed embedding of a scheme into an indscheme is invariantly
defined. In particular, the tautological maps $\be(I):Z(I)\to \CY$ are closed embeddings. 
(None of that will be the case for pseudo-indschemes.)

\medskip

Additionally, one obtains that the category $(S\in \Sch,S\overset{\on{cl.emb.}}\hookrightarrow \CY)$
is filtered, i.e., an indscheme $\CY$
has a canonical presentation as in \eqref{e:presentation of indscheme}, where the category of indices
is that of \emph{all} schemes equipped with a closed embedding into $\CY$.

\sssec{}

For future reference, let us give the following definitions.

\medskip

Let $\CY$ be an indscheme mapping to a scheme $S$. We shall say that $\CY$ is relatively ind-proper 
(resp., ind-closed subscheme) over
$S$ if for any closed embedding $T\to \CY$, the composed map $T\to S$ is proper (resp., closed embedding).

\medskip

It is easy to see that this is equivalent to requiring that for a given presentation of $\CY$ as in 
\eqref{e:presentation of indscheme}, the composed maps $Z(I)\to S$ are proper (resp., closed embeddings).

\medskip

A map $\CY_1\to \CY_2$ in $\on{PreStk}$ is called ind-schematic, if for any $S\in \Sch_{/\CY_2}$, the
base change $S\underset{\CY_2}\times \CY_1$ is an ind-scheme.

\medskip

It is easy to see that a morphism between two indschemes is ind-schematic (again, this relies
on the filteredness assumption of the index category).

\medskip

A map $\CY_1\to \CY_2$ in $\on{PreStk}$ is called ind-proper (resp., ind-closed embedding), if it is ind-schematic and for 
every $S\in \Sch_{/\CY_2}$ as above, the indscheme $S\underset{\CY_2}\times \CY_1$ is relatively ind-proper
(resp., ind-closed subscheme) over $S$. 

\ssec{Pseudo-indschemes} \label{ss:pseudo}

\sssec{}  \label{sss:defn of pseudo}

By a pseudo-indscheme we shall mean an object $\CY\in \on{PreStk}$ that can be written as 
\begin{equation} \label{e:colimit prestack}
\CY\simeq \underset{I\in \CS}{colim}\, Z(I)
\end{equation}
for some functor $Z:\CS\to \indSch$, such that
\begin{itemize}

\item For $(I_1\to I_2)$, the map $Z(I_1)\to Z(I_2)$ is ind-proper,

\end{itemize}
where $\CS$ is an arbitrary category of indices. (As was remarked before, with no restriction of generality,
we can take $\CS$ to be an ordinary category, rather than $\infty$-category.)

\medskip

We let $\be(I)$ denote the tautological map $Z(I)\to \CY$.

\medskip

In the above expression, the colimit is taken in the category $\on{PreStk}$. Again, 
the value of $\CY$ on $S\in \affSch$ is an object of $\inftygroup$ isomorphic to
$$\underset{I\in \CS}{colim}\, \on{Maps}(S,Z(I)),$$
where the latter colimit is taken in $\inftygroup$.

\sssec{}

Several remarks are in order:

\medskip

\noindent(i) It is crucial for what follows that in the definition of pseudo-indschemes, 
the colimit is taken in $\inftygroup$ and not, naively, in the category of sets.

\medskip

\noindent(ii) Non-filtered colimits in $\inftygroup$ are somewhat unwieldy objects:
we cannot algorithmically describe the space $\on{Maps}(S,\CY)$ for $S\in \affSch$.
One manifestation
of this phenomenon is that for a pair of affine schemes $S,T$ mapping to $\CY$, 
we cannot describe their Cartesian product $S\underset{\CY}\times T$. So, many
of the properties enjoyed by indschemes will fail for pseudo-indschemes.

\medskip

\noindent(iii) It is easy to see that in the definition of pseudo-indschemes, one can
require $Z(I)$ be schemes rather than indschemes. Indeed, given a presentation as
in \eqref{e:colimit prestack}, we can define a new category of indices $\wt\CS$ 
that consists of pairs $(I\in \CS,T\overset{\on{cl.emb.}}\hookrightarrow Z(I))$, and we will have
$$\CY\simeq \underset{(I,T)\in \wt\CS}{colim}\, T.$$

\medskip

\noindent(iv) The reason for singling out pseudo-indschemes among all prestacks
is that the category of D-modules on an pseudo-indscheme is more manageable 
than on an arbitrary prestack, see \secref{ss:colimit}. 

\sssec{The Ran space}

Let us consider an example of a pseudo-indscheme that will play a prominent role
in this paper.

\medskip

We take $\CS$ to be the category $(\fset)^{op}$, where $\fset$ denotes the category of non-empty finite sets and surjective maps. 
Let $X$ be a separated scheme. We define a functor
$$X^{\fset}:(\fset)^{op}\to \Sch$$
by assigning to a finite set $I$ the scheme $X^I$, and to a surjective map $\phi:J\twoheadrightarrow I$
the corresponding diagonal map $\Delta(\phi):X^I\to X^J$.

\medskip

We shall denote the resulting pseudo-indscheme
$$\underset{(\fset)^{op}}{colim}\, X^{\fset}\in \on{PreStk}$$
by $\Ran X$, and called it ``the Ran space of $X$".

\medskip

For a finite set $I$, we let $\Delta^I$ denote the corresponding map $X^I\to \Ran X$.

\begin{remark}  \label{r:Ran discrete}
Although the colimit in the formation of $\Ran X$ was taken in $\on{PreStk}$, it is easy to see
that the functor $\Ran X:(\affSch)^{op}\to \inftygroup$ takes values in $0$-truncated groupoids,
i.e., in $\on{Set}\subset \inftygroup$. 

\medskip

In fact, $\on{Maps}(S,\Ran X)$ is the set of non-empty finite 
subsets of $\on{Maps}(S,\Ran X)$.

\medskip

Indeed, for $S\in \affSch$, we have:
$$\on{Maps}(S,\Ran X)=\underset{(\fset)^{op}}{colim}\, (\on{Maps}(S,X))^I,$$
but it is easy to see that for a set $A$, the colimit $\underset{(\fset)^{op}}{colim}\, A^I$
is discrete, and is isomorphic to set of all finite non-empty subsets of $A$. 

\end{remark}

\sssec{}  \label{sss:ind-schematic}

For future reference, let us give the following definitions. (The reader can skip this subsection
and return to it when necessary.)

\medskip

Let $\CY$ be a pseudo-indscheme mapping to a scheme $S$. We shall say that $\CY$
is pseudo ind-proper over $S$ if \emph{there exists} a presentation of $\CY$ as in \eqref{e:colimit prestack},
such that the resulting maps $Z(I)\to S$ are ind-proper. 

\medskip

We shall say that a map $\CY_1\to \CY_2$ in $\on{PreStk}$ is \emph{pseudo ind-schematic}
if for any $S\in \affSch_{/\CY_2}$, the object $S\underset{\CY_1}\times \CY_2$ is a pseudo-indscheme.

\medskip

Unlike indschemes, it is not true that any map between pseudo-indschemes is pseudo
ind-schematic. It is not even true that a map from an affine scheme to a pseudo-indscheme
is pseudo ind-schematic.

\medskip

We shall say that a map $\CY_1\to \CY_2$ in $\on{PreStk}$ is \emph{pseudo ind-proper}
if is pseudo ind-schematic and for every $S\in \affSch_{/\CY_2}$ as above, the resulting pseudo-indscheme
$S\underset{\CY_1}\times \CY_2$ is pseudo ind-proper over $S$.

\ssec{D-modules on pseudo-indschemes}  \label{ss:Dmods on diags}

\sssec{}

Recall that for any object $\CY\in \on{PreStk}$ we define the category $\fD(\CY)$
as
$$\underset{S\in (\affSch/\CY)^{op}}{lim}\, \fD(S).$$

\medskip

In other words, we define the functor 
$$\fD^!_{\on{PreStk}}:\on{PreStk}^{op}\to \StinftyCat_{\on{cont}},$$
as the right Kan extension of the functor
$$\fD^!_{\affSch}:(\affSch)^{op}\to \StinftyCat_{\on{cont}},$$
along the tautological embedding
$$(\affSch)^{op}\hookrightarrow \on{PreStk}^{op},$$
see \secref{sss:D-modules}.\footnote{This is the same as the right Kan extension of the functor
$\fD^!_{\Sch}$ along $(\Sch)^{op}\hookrightarrow \on{PreStk}^{op}$.}

\medskip

In what follows, when no confusion is likely to occur, we will suppress the subscript ``!" and the subscript ``$\on{PreStk}$"
from the notation, i.e., we shall simply write $\fD(\CY)$ rather than $\fD^!_{\on{PreStk}}(\CY)$. 

\sssec{}

The main object of study of this section is the category $\fD(\CY)$, where $\CY$
is a pseudo-indscheme.

\medskip

The above interpretation of the functor $\fD_{\on{PreStk}}$ as the right Kan extension
implies that it takes colimits in $\on{PreStk}$ to limits in $\StinftyCat_{\on{cont}}$. 

\medskip

Hence, for $\CY$ written as in \eqref{e:colimit prestack}, we have:
\begin{equation} \label{e:Dmodules on colimit}
\fD(\CY)\simeq \underset{I\in \CS^{op}}{lim}\, \fD(Z(I)),
\end{equation}
where the limit is formed using the !-pullback functors. 

\sssec{}  \label{sss:morphism of Z}

Let $f:\CY\to \CY'$ be a map between prestacks. By construction, we have
a pullback functor
$f^!:\fD(\CY')\to \fD(\CY)$.

\medskip

For example, taking $\CY'=\on{pt}$, and the tautological map
$p_\CY:\CY\to \on{pt}$, we obtain a functor
$$p_\CY^!:\Vect\to \fD(\CY).$$

\medskip

In particular, the category $\fD(\CY)$ contains the canonical object
$\omega_{\CY}:=p_\CY^!(k)$, 
which we shall refer to as ``the dualizing sheaf". 

\ssec{The category of D-modules as a colimit}  \label{ss:colimit}

A distinctive feature of pseudo-indschemes among arbitrary objects of $\on{PreStk}$
is that the category $\fD(\CY)$ can be alternatively described as a \emph{colimit}
in $\StinftyCat_{\on{cont}}$.

\sssec{}  \label{sss:left adjoint}

Let us recall the following general construction. Let $\bC$ be an $\infty$-category, and let
$$\Phi:\bC\to \StinftyCat_{\on{cont}}$$
be a functor. Suppose that for every arrow $g:\bc_1\to \bc_2$, the resulting functor
$$\Phi(g):\Phi(\bc_1)\to \Phi(\bc_2)$$
admits a left adjoint, $^L\Phi(g)$. Then the assignment $g\rightsquigarrow {}^L\Phi(g)$
canonically extends to a functor
$$^L\Phi:\bC^{op}\to \StinftyCat.$$

\sssec{}




Let $\CY$ be a pseudo-indscheme, written as in \eqref{e:colimit prestack}. Consider the functor
$$\fD^!(Z):\CS^{op}\to \StinftyCat$$
equal to the composition of $Z^{op}:\CS^{op}\to \on{PreStk}^{op}$ with the functor 
$$\fD^!_{\on{PreStk}}:\on{PreStk}^{op}\to \StinftyCat.$$

\medskip

The properness assumption on the maps $Z(I)\to Z(J)$ implies that the functor $\fD^!(Z)$
satisfies the assumption of \secref{sss:left adjoint}. Let $\fD_!(Z)$ denote the
resulting functor $\CS\to \StinftyCat$.

\medskip

By \cite{DG}, Lemma 1.3.3, we have:
\begin{equation} \label{e:limit as colimit}
\underset{\CS^{op}}{lim}\, \fD^{!}(Z)\simeq \underset{\CS}{colim}\, \fD_!(Z).
\end{equation}

In particular, this implies that the category $\fD(\CY)$ is compactly generated,
see \cite{DG}, Sect. 2.2.1. The latter observation, combined with the isomorphism \eqref{e:prod equivalence}, implies that for any
$\CY'\in \on{PreStk}$, the functor
\begin{equation} \label{e:prod equiv pseudo}
\fD(\CY)\otimes \fD(\CY')\to \fD(\CY\times \CY')
\end{equation}
is an equivalence.

\medskip

Note also, that by \cite{DG}, Sect. 2.2.1, a choice of presentation of $\CY$ as in \eqref{e:colimit prestack},
defines an equivalence $\fD(\CY)^\vee\simeq \fD(\CY)$, i.e., an anti self-equivalence
$$(\fD(\CY)^c)^{op}\simeq \fD(\CY)^c.$$

\sssec{}

For $I\in \CS$, we shall denote by $\be(I)^!$ the tautological forgetful functor 
$$\fD(\CY)=\underset{\CS^{op}}{lim}\, \fD^!(Z)\to \fD(Z(I))$$
and by $\be(I)_!$ its left adjoint, which in terms of the equivalence \eqref{e:limit as colimit}
corresponds to the tautological functor
$$\fD(Z(I))\to \underset{\CS}{colim}\, \fD_!(Z).$$

\sssec{}

The two functors in \eqref{e:limit as colimit} can be explicitly described as follows:

\medskip

The functor
$$\underset{I\in \CS}{colim}\, \fD(Z(I))\to \fD(\CY)$$
corresponds to the compatible family of functors $\fD(Z(I))\to \fD(\CY)$ given by $\be(I)_!$. 

\medskip

The functor
$$\fD(\CY)\to \underset{I\in \CS}{colim}\, \fD(Z(I))$$
sends 
\begin{equation} \label{e:from limit to colimit}
(\CF\in \fD(\CY))\rightsquigarrow \underset{I\in \CS}{colim}\, \be(I)^!(\CF)\in \underset{I\in \CS}{colim}\, \fD(Z(I)).
\end{equation}

\ssec{Direct images with compact supports}

Let us return to the situation of a morphism $f:\CY\to \CY'$. It is not, in general, true that the
functor $f^!:\fD(\CY')\to \fD(\CY)$ admits a left adjoint.  

\sssec{}  \label{sss:intr good}

In general, if $G:\bC'\to \bC$ is a functor between $\infty$-categories, one can consider the full
subcategory of $\bC$, denoted $\bC_{\on{good}}$, consisting of objects $\bc\in \bC$, for
which the functor
$$\bC'\to \inftygroup,\,\, \bc'\mapsto \on{Maps}_{\bC}(\bc,G(\bc'))$$
is co-representable.

\medskip

In this case, there exists a canonically defined functor $F:\bC_{\on{good}}\to \bC'$, equipped with an
isomorphism of functors
$$(\bC_{\on{good}})^{op}\times \bC'\rightrightarrows \inftygroup$$ that send $\bc\in \bC_{\on{good}}$
and $\bc'\in \bC'$ to
$$\on{Maps}_{\bC'}(F(\bc),\bc') \text{ and } \on{Maps}_{\bC}(\bc,G(\bc')),$$
respectively.

\medskip

In this case, we shall refer to $F$ as ``the partially defined left adjoint of $G$". For an object $\bc\in \bC$
we shall say that ``the partially defined left adjoint of $G$ is defined on $\bc$" if $\bc\in \bC_{\on{good}}$.

\sssec{}

Let
$$\fD(\CY)_{\on{good}\, \on{for}\,f}\subset \fD(\CY)$$
be the full subcategory that consists of objects, on which the partially defined 
left adjoint $f_{!}$ to $f^!$ is defined.

\medskip

For a general map $f$, it is not clear how to construct construct objects in 
$\fD(\CY)_{\on{good}\, \on{for}\,f}$. However, below we shall describe a situation
when can generate ``many" objects of this category.

\medskip

Note, however, that for any map $f:\CY\to \CY'$, for which $\omega_\CY$
belongs to $\fD(\CY)_{\on{good}\, \on{for}\,f}$, we have a canonical map
\begin{equation} \label{e:trace map omega}
\on{Tr}_{\omega}(f):f_{!}(\omega_{\CY})\to \omega_{\CY'}, 
\end{equation}
coming by adjunction
from
$$\omega_{\CY}\simeq f^!(\omega_{\CY'}).$$

\sssec{}  \label{sss:paradigm}

First, let us recall the following general paradigm of constructing maps between
colimits. 

\medskip

Let $F:\bC_1\to \bC_2$ is a functor between
$\infty$-categories, and let
$$\Phi_1:\bC_1\to \bD \text{ and } \Phi_2:\bC_2\to \bD$$
be functors, where $\bD$ is another $\infty$-category. Let us be given a natural transformation
of functors $\bC_1\rightrightarrows \bD$:
$$\Phi_1\Rightarrow \Phi_2\circ F.$$

\medskip

Then we obtain a map in $\bD$:
\begin{equation} \label{e:map of colimits}
\underset{\bC_1}{colim}\, \Phi_1  \to \underset{\bC_2}{colim}\, \Phi_2.
\end{equation}

\sssec{}  \label{sss:expl map}

Let $\CS$ and $\CS'$ be two categories of indices and $f_\CS:\CS\to \CS'$ a functor. 
Let $Z:\CS\to \indSch$ and $Z':\CS'\to \indSch$ be two functors as in \secref{sss:defn of pseudo},
and let $f_Z$ be a natural transformation between the resulting two functors
$\CS\rightrightarrows \indSch$:
$$Z\Rightarrow Z'\circ f_\CS.$$
For $I\in \CS$ we let $f(I)$ denote the resulting map of indschemes $Z(I)\to Z'(I')$,
where $I'=f_\CS(I)$.

\medskip

By the above, we obtain a map $f:\CY\to \CY'$ between the corresponding two 
pseudo-indschemes.

\sssec{}  \label{sss:good objects}

Keeping the notation of the previous subsection, for each $I\in \CS$, let 
$$\fD(Z(I))_{\on{good}\, \on{for}\,f(I)}$$ be the full subcategory of $\fD(Z(I))$
that consists of objects, on which the partially defined left adjoint $f(I)_{!}$
to $f(I)^!$ is defined.

\medskip

By adjunction, we obtain that in the next diagram the upper horizontal arrow is well-defined
(i.e., its image belongs to the indicated subcategory) and that the diagram commutes:
$$
\CD
\fD(Z(I))_{\on{good}\, \on{for}\,f(I)}    @>{\be(I)_!}>>  \fD(\CY)_{\on{good}\, \on{for}\,f} \\
@V{f(I)_{!}}VV    @VV{f_{!}}V  \\
\fD\left(Z'(I')\right)  @>{\be(I')_!}>>  \fD(\CY').
\endCD
$$

\medskip

As a consequence, from \eqref{e:from limit to colimit}, 
we obtain that if $\CF\in \fD(\CY)$ is such that for each $I$,
$$\be(I)^!(\CF)\in \fD(Z(I))_{\on{good}\, \on{for}\,f(I)},$$
then $\CF\in \fD(\CY)_{\on{good}\, \on{for}\,f}$,
and we have:
\begin{equation} \label{e:calc dir image c}
f_{!}(\CF)\simeq \underset{I\in \CS}{colim}\, \be(I')_!\circ f(I)_{!}\circ \be(I)^!(\CF)\in \fD(\CY').
\end{equation}

\sssec{}  \label{sss:when good}

The above description shows:

\medskip

\noindent(i) If all the maps $f(I)$ happen to be proper, then $\fD(\CY)_{\on{good}\, \on{for}\,f}$
equals all of $\fD(\CY)$.

\medskip

\noindent(ii) If $\CF\in \fD(\CY)$ is such that for all $I$, the object $\be(I)^!(\CF)\in \fD(Z(I))$
has holonomic cohomologies, then $\CF\in \fD(\CY)_{\on{good}\, \on{for}\,f}$.

\medskip

\noindent(iii) The object $\omega_\CY$ always belongs to $\CF\in \fD(\CY)_{\on{good}\, \on{for}\,f}$.

\ssec{(Co)homology}

\sssec{}

Let us take $\CY'=\on{pt}$ and $f=p_\CY:\CY\to \on{pt}$. It is clear that the map $p_\CY$ falls into the paradigm 
described in \secref{sss:expl map}.

\medskip

For $\CY\in \fD(\CY)_{\on{good}\, \on{for}\,p_\CY}$, we shall also
use the notation 
$$\Gamma_{\dr,c}(\CY,-):=(p_\CY)_!(-).$$

\medskip

From \eqref{e:calc dir image c}, we obtain that if $\CF\in \fD(\CY)$ is such that 
for each $I$,
$$\be(I)^!(\CF)\in \fD(Z(I))_{\on{good}\, \on{for}\,p_{Z(I)}},$$
we have $\CF\in \fD(\CY)_{\on{good}\, \on{for}\,p_\CY}$, and
\begin{equation} \label{e:calc dR c}
\Gamma_{\dr,c}(\CY,\CF)\simeq 
\underset{I\in \CS}{colim}\, \Gamma_{\dr,c}\left(Z(I),\be(I)^!(\CF)\right).
\end{equation}

\sssec{}  \label{sss:abs trace map}

Let us use the notation 
$$H_\bullet(\CY):=\Gamma_{\dr,c}(\CY,\omega_\CY).$$

\medskip

From \eqref{e:trace map omega}, we obtain a canonical map
\begin{equation} \label{e:trace map}
\on{Tr}_{H_\bullet}:H_\bullet(\CY)\to k.
\end{equation}

\medskip

Moreover, for any map $f:\CY\to \CY'$ for which $\omega_\CY\in \fD(\CY)_{\on{good}\, \on{for}\,f}$,
by applying the partially defined functor $\Gamma_{\dr,c}(\CY',-)$ to \eqref{e:trace map omega}, we obtain a map
\begin{equation} \label{e:trace on homology}
\on{Tr}_{H_\bullet}(f):H_\bullet(\CY)\to H_\bullet(\CY').
\end{equation}

\sssec{}   \label{sss:connectivity}

Let $\CY$ be written as in \eqref{e:colimit prestack}. From \eqref{e:calc dR c} we obtain:
\begin{equation} \label{e:expl homology non-proper}
H_\bullet(\CY)\simeq \underset{I\in \CS}{colim}\, H_\bullet(Z(I)).
\end{equation}







\medskip

In particular, we conclude that the object $H_\bullet(\CY) \in \Vect$
is always connective, i.e., lives in non-positive cohomological degrees.
Also, we conclude that the trace map
$$\on{Tr}_{H_\bullet}:H_\bullet(\CY) \to k$$
has the property that the map 
$$H_0(\CY)\to k$$
is non-zero whenever $\CY$ is non-empty, and if all $Z(I)$ are connected, the map
$H_0(\CY)\to k$ is an isomorphism.

\sssec{}

Finally, let us recall the following basic result of \cite{BD1}, Proposition 4.3.3, that will be crucial for this paper:

\begin{thm} \label{t:contractibility of Ran}
Let $X$ be a connected separated scheme. Then the above map
$$\on{Tr}_{H_\bullet}:H_\bullet(\Ran X)\to k$$
is an isomorphism.
\end{thm}

\begin{remark}
Strictly speaking, in \cite{BD1}, the above result is only proved when $X$ is a curve
(which is also our main case of interest in this paper). However, the proof given
in {\it loc.cit.} applies to the general case as well. For completeness, we shall include
the argument in the general case in \secref{s:proof contr Ran}, by essentially 
repeating {\it loc.cit.}
\end{remark}

\ssec{The space of rational maps}

\sssec{}

Let us recall the following construction. Let $U$ be a scheme and $Y$ an affine scheme (according
to our conventions, assumed of finite type). 

\medskip

Consider the functor 
$${\mathbf{Maps}}(U,Y):\Sch^{op}\to \on{Sets}$$
that assigns to a test scheme $S$ the set of maps $S\times U\to Y$.

\medskip

It is easy to see that this functor is representable by an indscheme (of ind-finite type).
Indeed, by representing  $Y$ as a Cartesian product
$$
\CD
Y  @>>> \BA^n \\
@VVV   @VVV  \\
\{0\}   @>>>  \BA^m 
\endCD
$$
we reduce the assertion to the case when $Y=\BA^n$, and by taking products
further to the case $Y=\BA^1$. 

\medskip

In the latter case, ${\mathbf{Maps}}(U,Y)$
is representable by the (infinite-dimensional, unless $U$ is proper) vector space $\Gamma(U,\CO_U)$,
viewed as an indscheme.

\sssec{}  \label{sss:rel maps}

In what follows we shall need a version of the above construction in the relative situation,
namely, when $U$ is a scheme flat over a base $T$. In this case we define ${\mathbf{Maps}}_T(U,Y)$ 
as a functor on the category of schemes over $T$. 

\medskip

We do not know what conditions on the map $\pi:U\to T$ guarantee this in general. However, we have
the following statement: 

\begin{lem}  \label{l:guarantee ind}
Suppose that the (derived) direct image $\pi_*(\CO_U)\in \QCoh(T)$ can be written as a filtered colimit
of objects $\CE_\alpha\in \QCoh(T)^{\on{perf}}$ such that for every index $\alpha$, the $\CO_T$-dual
$\CE_\alpha^\vee$ belongs to $\QCoh(T)^{\leq 0}$, and for every map of indices $\alpha_1\to \alpha_2$, 
the corresponding map $\CE_{\alpha_2}^\vee\to \CE_{\alpha_1}^\vee$ is a surjection on $H^0(-)$.
Then ${\mathbf{Maps}}_T(U,Y)$ is representable by an ind-scheme. 
\end{lem}

\begin{remark}
If $\pi_*(\CO_U)\in \QCoh(T)$ is concentrated in cohomological degree $0$, then the condition
of the lemma is equivalent to requiring that $\pi_*(\CO_U)$ be a Mittag-Leffler module in
the sense of \cite[Sect. 7.12.1]{BD2}.
\end{remark}

In the situation of the lemma, the indscheme representing ${\mathbf{Maps}}_T(U,Y)$ can be explicitly 
written as
$$\underset{\alpha}{colim}\, \Spec_T(\Sym(H^0(\CE_\alpha^\vee))).$$

\medskip

The conditions of the lemma are satisfied, for example, when the map $\pi$ admits a compactification $$\ol\pi:\ol{U}\to T,$$
such that $\ol\pi$ is also flat, and the complement $\ol{U}-U$ is a relatively ample divisor $D$ which
is flat over $T$. In this case we take the category of indices to be $\BN$, and for $n\gg 0$, we take 
$$\CE_n:=\ol\pi_*(\CO_{\ol{U}}(n\cdot D)).$$

\sssec{}  \label{sss:intr rat maps}

Let $X$ be a smooth, connected and complete curve, and let $Y$ be an affine scheme. 
We define a functor
$${\mathbf{Maps}}(X,Y)^{\on{rat}}_{X^{\fset}}:(\fset)^{op}\to \indSch$$
as follows. 

\medskip

For $I\in \fset$ we consider the scheme $$(X^I\times X)-\Gamma^I,$$ where
$\Gamma^I\subset X^I\times X$
is the incidence divisor. We regard it as a scheme over $X^I$. 
For future reference, for an $S$-point $x^I$ of $X^I$ we will denote by
$\{x^I\}$ the closed subscheme of $S\times X$ equal to
$(x^I\times \on{id}_X)^{-1}(\Gamma^I)$.

\medskip

We let
$${\mathbf{Maps}}(X,Y)^{\on{rat}}_{X^I}:={\mathbf{Maps}}_{X^I}((X^I\times X)-\Gamma^I,Y).$$

\medskip

By \lemref{l:guarantee ind}, ${\mathbf{Maps}}(X,Y)^{\on{rat}}_{X^I}$ is an indscheme.

\sssec{}

We define the object ${\mathbf{Maps}}(X,Y)^{\on{rat}}_{\Ran X}\in \on{PreStk}$ as 
$$\underset{(\fset)^{op}}{colim}\, {\mathbf{Maps}}(X,Y)^{\on{rat}}_{X^{\fset}}.$$
It is a pseudo-indscheme, by construction. 

\medskip

Note that also by construction, the functor 
$${\mathbf{Maps}}(X,Y)^{\on{rat}}_{X^{\fset}}:(\fset)^{op}\to \indSch$$
comes equipped with a natural transformation 
to $X^{\fset}$. We let $f$ denote the resulting map
$${\mathbf{Maps}}(X,Y)^{\on{rat}}_{\Ran X}\to \Ran X,$$
see \secref{sss:expl map}.

\begin{remark}  \label{r:maps discrete}
As in Remark \ref{r:Ran discrete}, one can show that 
$${\mathbf{Maps}}(X,Y)^{\on{rat}}_{\Ran X}:(\affSch)^{op}\to \inftygroup$$
takes values in $\on{Set}\subset \inftygroup$. 

\medskip

In fact, a data of an $S$-point of
${\mathbf{Maps}}(X,Y)^{\on{rat}}_{\Ran X}$ is equivalent to that of a non-empty finite
\emph{subset} $\ol{x}\subset \on{Maps}(S,X)$ plus a rational map
$S\times X\to Y$, which is regular on the complement to the graph of $\ol{x}$.
\end{remark}

\sssec{}

We define ``homology of the space of rational maps" as
$$H_\bullet({\mathbf{Maps}}(X,Y)^{\on{rat}}_{\Ran X}).$$

\medskip

Note that by transitivity,
$$H_\bullet({\mathbf{Maps}}(X,Y)^{\on{rat}}_{\Ran X})\simeq \Gamma_{\dr,c}\left(\Ran X,
f_{!}(\omega_{{\mathbf{Maps}}(X,Y)^{\on{rat}}_{\Ran X}})\right).$$

\begin{remark}  \label{r:not naive}
We emphasize that the functor $f_!$ is defined via the realization of the category
$\fD(\Ran X)$ as in \secref{ss:colimit} as $\underset{I\in (\fset)^{op}}{colim}\, \fD(X^I)$.
That is, if we want to ``evaluate'' the object $f_{!}(\omega_{{\mathbf{Maps}}(X,Y)^{\on{rat}}_{\Ran X}})$
on a given finite set $I$, i.e., if we are interested in 
$$(\Delta^I)^!\left(f_{!}(\omega_{{\mathbf{Maps}}(X,Y)^{\on{rat}}_{\Ran X}})\right)\in \fD(X^I),$$
we will have to apply the equivalence \eqref{e:limit as colimit}, and the result will \emph{not} be 
isomorphic to 
$$f(I)_!\left(\omega_{{\mathbf{Maps}}(X,Y)^{\on{rat}}_{X^I}})\right).$$
\end{remark}






\ssec{Statement of the the main result}

\sssec{}

The main result of this paper is:

\medskip

Consider the trace map of \eqref{e:trace map}
\begin{equation} \label{e:integral}
\on{Tr}_{H_\bullet}:H_\bullet({\mathbf{Maps}}(X,Y)^{\on{rat}}_{\Ran X})\to k.
\end{equation}

\begin{thm}  \label{t:contractibility}
Suppose that $Y$ is connected and can be covered by open subsets $U_\alpha$, each of which is isomorphic
to an open subset of the affine space $\BA^n$ (for some integer $n$). 
Then the map \eqref{e:integral} is an isomorphism.
\end{thm}

Note that for $Y=\on{pt}$, the statement of \thmref{t:contractibility} coincides
with that of \thmref{t:contractibility of Ran}.

\begin{remark}
The assumption that the curve $X$ be complete is inessential: if $\oX\subset X$ is a non-empty open subset,
the spaces ${\mathbf{Maps}}(X,Y)^{\on{rat}}_{\Ran X})$ and ${\mathbf{Maps}}(\oX,Y)^{\on{rat}}_{\Ran X})$
have isomorphic homology. This follows from \corref{c:marked points homology} and Equation 
\eqref{e:including points}.
\end{remark}

\sssec{}

A typical example of a scheme $Y$ satisfying the assumption of \thmref{t:contractibility} is a connected
affine algebraic group $G$. Then the required cover is provided by the Bruhat decomposition. 

\sssec{}

Finally, we propose:
\begin{conj}  \label{conj:contractibility}
The assertion of \thmref{t:contractibility} holds for any $Y$ which is connected, smooth and
birational to $\BA^n$. 
\end{conj}

\section{The unital setting}  \label{s:unital}

In this section we shall introduce a unital structure on the space ${\mathbf{Maps}}(X,Y)^{\on{rat}}_{\Ran X}$,
and the corresponding space ${\mathbf{Maps}}(X,Y)^{\on{rat}}_{\Ran X,\on{indep}}$. Its significance will be two-fold:
\footnote{The terminology ``unital" is motivated by the property possessed by the factorization
algebra corresponding to a unital chiral algebra, see \cite{BD1}, Sect. 3.4.5.}

\medskip

At the conceptual level, the space ${\mathbf{Maps}}(X,Y)^{\on{rat}}_{\Ran X,\on{indep}}$ gets
rid of the redundancy inherent in the definition of ${\mathbf{Maps}}(X,Y)^{\on{rat}}_{\Ran X}$, namely,
of specifying the locus where our rational map is defined. 

\medskip

Technically, certain calculations are easier to perform in the unital version; in particular,
ones that appear in the proofs of \thmref{t:contractibility} and \thmref{t:pullback fully faithful}.

\medskip

However, the reader may skip this section on the first pass: we will explain an alternative (but
equivalent) way to perform the above mentioned calculations involved in the proofs of the 
main theorems.

\medskip

Throughout this section, $X$ will be a separated connected scheme. 

\ssec{Spaces acted on by $\Ran X$}  \label{ss:spaced acted by Ran}

\sssec{}  \label{sss:monoidal structure on Ran}

Observe that the category $\fset$ (and, hence, $(\fset)^{op}$) has a natural symmetric monoidal 
(but non-unital) structure given by disjoint union. 

\medskip

The functor $X^{\fset}:(\fset)^{op}\to \Sch$ also has a natural symmetric monoidal structure. This
defines on $\Ran X=\underset{(\fset)^{op}}{colim}\, X^{\fset}$ a structure of commutative 
(but non-unital) semi-group in $\on{PreStk}$.

\medskip

Concretely, the map 
$$\on{union}_{\Ran}:\Ran X\times \Ran X\to \Ran X$$
can be described in terms of \secref{sss:expl map} as follows. It corresponds to the functor
$$\sqcup:(\fset\times \fset)^{op}\to (\fset)^{op}$$
and to the natural transformation (in fact, an isomorphism) of the resulting two functors
$(\fset\times \fset)^{op}\rightrightarrows \Sch$:
$$X^{\fset}\times X^{\fset}\Rightarrow X^{\fset}\circ \sqcup,$$
given by
$$(I_1,I_2)\rightsquigarrow \Bigl(X^{I_1}\times X^{I_2}\to X^{I_2\sqcup I_2}\Bigr).$$

\sssec{}  \label{sss:functors acted on by Ran}

Let $\CS$ be an index category, which is a module for $(\fset)^{op}$. We shall denote by
$\sqcup$ the action functor 
$$\left(J\in (\fset)^{op},I\in \CS\right)\rightsquigarrow J\sqcup I\in \CS.$$

\medskip

Let $Z:\CS\to \indSch$ be a functor with a structure of module for $X^{\fset}$. 
I.e., for every $J\in (\fset)^{op}$ and $I\in \CS$ we are given a map 
$$\on{unit}_{J,I}:X^J\times Z(I)\to Z(J\sqcup I),$$
which are functorial and associative in a natural sense. 

\medskip

We shall refer to this structure as a \emph{unital structure} on $Z$ with respect to $\Ran X$.
In this case $\CY:=\underset{\CS}{colim}\, Z$ becomes a module over $\Ran X$. We let
$\on{unit}_{\Ran}$ denote the resulting map $\Ran X\times \CY\to \CY$.

\sssec{}

We let $\CY_{\bDelta_s}$ denote the corresponding semi-simplicial object\footnote{We denote by $\bDelta_s$
the non-full subcategory of $\bDelta$ obtained by restricting $1$-morphisms to be injections of finite
ordered sets.} 
of $\on{PreStk}$:
$$...\Ran X\times \CY\rightrightarrows \CY.$$

\medskip

We define $\CY_{\on{indep}}$ to be the geometric realization of $\CY_{\bDelta_s}$, i.e.,
$$\CY_{\on{indep}}:=\underset{\bDelta_s^{op}}{colim}\, \CY_{\bDelta_s}.$$

\begin{remark}
Note that if the maps $\on{unit}_{J,I}:X^J\times Z(I)\to Z(J\sqcup I)$ are proper, then 
$\CY_{\on{indep}}$ is a pseudo-indscheme.
\end{remark}

\sssec{Example}  \label{sss:ex of maps} Let $X$ and $Y$ be be as in \secref{sss:intr rat maps}. Let $\CS=(\fset)^{op}$,
with the natural action of $(\fset)^{op}$ on itself. Let $Z$ be the functor
$${\mathbf{Maps}}(X,Y)^{\on{rat}}_{X^{\fset}}:(\fset)^{op}\to \indSch.$$

\medskip

It has a natural unital structure. Indeed for finite sets $J$ and $I$ the map
$$X^J\times {\mathbf{Maps}}(X,Y)^{\on{rat}}_{X^I}\subset
{\mathbf{Maps}}(X,Y)^{\on{rat}}_{X^{J\sqcup I}}$$
is the closed embedding corresponding to restricting a map
$$m:(S\times X-\{x^I\})\to Y$$
to a map $$m':(S\times X-(\{x^J\}\sqcup \{x^I\}))\to Y.$$

\medskip

Consider the resulting object ${\mathbf{Maps}}(X,Y)^{\on{rat}}_{\Ran X,\on{indep}}$.
We can regard it as a version of ${\mathbf{Maps}}(X,Y)^{\on{rat}}_{\Ran X}$, where we have
explicitly ``modded out" by the dependence on the locus of singularity. 

\medskip

We shall see two more classes of examples in Sects. \ref{ss:marked points} and
\ref{ss:pairs}, respectively.

\ssec{Strongly unital structures}  \label{ss:strongly unital}

\sssec{}  

Let is now assume that we have the following additional pieces of structure on the category $\CS$.
Namely, let us be given a functor $f_\CS:\CS\to (\fset)^{op}$, and a natural transformation
$$\bd:\on{Id}\Rightarrow \sqcup\circ (f_\CS\times \on{Id}),$$
i.e., map functorially assigned to every $I\in \CS$:
$$\bd(I):I\to f_\CS(I)\sqcup I.$$

\sssec{Example}  \label{sss:strong basic}

Assume that $\CS=(\fset)^{op}$, and let $f_\CS$ be the identity functor. In this case we take
$\bd$ to be the canonical map 
$$I\sqcup I\to I$$
for $I\in \fset$.

\sssec{}  \label{sss:strong axiom}

Let $Z:\CS\to \indSch$ be a unital functor. We shall say that $Z$ is \emph{strongly unital}
with respect to $(f_\CS,\bd)$ if we are given a natural transformation
$$f_Z:Z\Rightarrow X^{\fset}\circ f_\CS$$
(i.e., we have a map $f(I):Z(I)\to X^{f_\CS(I)}$ that functorially depends on $I\in \CS$),
such that for every $I\in \CS$ the composed map
\begin{equation} \label{e:axiom on unit}
Z(I)\overset{f(I)\times \on{id}}\longrightarrow X^{f_\CS(I)}\times Z(I)\overset{\on{unit}_{f_\CS(I),I}}\longrightarrow
Z(f_\CS(I)\sqcup I)
\end{equation}
equals the map 
$$Z(\bd(I)):Z(I)\to Z(f_\CS(I)\sqcup I).$$

\begin{remark}
In \secref{sss:extended funct} we shall explain the meaning of the above data in the example
of \secref{sss:strong basic}.
\end{remark}

\sssec{Example} \label{sss:strong basic Ran}
Let $(\CS,f_\CS,\bd)$ be as in \secref{sss:strong basic}, and let
us take $Z=X^{\fset}$, where $f_Z$ is the identity map. In this case the requirement of
\secref{sss:strong axiom} holds tautologically. 

\sssec{Example} Let $(\CS,f_\CS,\bd)$ be again as in \secref{sss:strong basic}.
Let us take $Z={\mathbf{Maps}}(X,Y)^{\on{rat}}_{X^{\fset}}$. 

\medskip

We let $f_Z$ be the map that assigns to $I$ the tautological projection
$${\mathbf{Maps}}(X,Y)^{\on{rat}}_{X^I}\to X^I.$$

It is easy to see that this functor satisfies the condition of \secref{sss:strong axiom}.

\sssec{}  \label{sss:right inverse}

Let $Z$ be equipped with a strongly unital structure, and consider the corresponding
prestack
$$\CY:=\underset{\CS}{colim}\, Z.$$ 

We claim that in this case there exists a canonically defined map
\begin{equation} \label{e:right inverse}
\CY\to \Ran X\times \CY,
\end{equation}
such that its compositions with both the action map $$\on{unit}_{\Ran}:\Ran X\times \CY\to \CY$$ and 
the projection $\Ran X\times \CY\to \CY$ are the identity maps on $\CY$. 

\medskip

In terms of \secref{sss:expl map}, the map \eqref{e:right inverse} corresponds to the functor
$$\CS\to (\fset)^{op}\times \CS$$
given by $f_\CS\times \on{Id}$, and the natural transformation
$$Z\Rightarrow X^{\fset}\times Z$$
given by 
$$Z(I)\overset{f(I)\times \on{id}}\longrightarrow X^{f_\CS(I)}\times Z(I).$$

\medskip

The fact that the composition of the map \eqref{e:right inverse} with the projection
$\Ran X\times \CY\to \CY$ is the identity map on $\CY$ is immediate.

\medskip

To show that the composition
\begin{equation} \label{e:left inverse composed}
\CY\to \Ran X\times \CY\overset{\on{unit}_{\Ran}}\longrightarrow \CY
\end{equation}
is the identity map, we will use the following general observation.

\sssec{} \label{sss:paradigm nat trans}

Suppose that in the paradigm of \secref{sss:paradigm}, we are given another
functor $F':\bC_1\to \bC_2$, and a natural transformation $\beta:F\Rightarrow F'$. Let
$\alpha$ denote the original natural transformation $\Phi_1\Rightarrow \Phi_2\circ F$.
Composing,
we obtain a natural transformation $\alpha':\Phi_1\to \Phi_2\circ F'$, and hence \emph{another} map
$$\underset{\bC_1}{colim}\, \Phi_1\to \underset{\bC_2}{colim}\, \Phi_2.$$
However, the resulting two maps, one coming from $\alpha$, another from $\alpha'$:
$$\underset{\bC_1}{colim}\, \Phi_1\rightrightarrows \underset{\bC_2}{colim}\, \Phi_2$$
are canonically homotopic. 

\sssec{}

Returning to the composition \eqref{e:left inverse composed}, we apply the setting of \secref{sss:paradigm nat trans}
to $\bC_1=\bC_2=\CS$, $\Phi_1=\Phi_2=Z$, with $F$ being the identity functor, and $F'$
being the functor $\on{unit}_{\Ran}\circ (f_\CS\times \on{Id})$, i.e., 
$$I\rightsquigarrow (f_\CS(I)\sqcup I).$$
We claim that the resulting natural transformation 
$$Z\Rightarrow Z\circ \left(\on{unit}_{\Ran}\circ (f_\CS\times \on{Id})\right)$$
comes from the natural transformation $\beta$
$$\on{Id}\to \left(\on{unit}_{\Ran}\circ (f_\CS\times \on{Id})\right),$$
supplied by $\bd$. Indeed, this follows from the requirement on $\bd$ expressed by
\eqref{e:axiom on unit}.

\sssec{Example}  \label{sss:sq of diag}

Let us return to the example of \secref{sss:strong basic Ran}, i.e., 
$Z=X^{\fset}$ as a functor $(\fset)^{op}\to \Sch$, with the data
of $f_Z$ being the identity map. Note that in this case, the map
$$\Ran X\to \Ran X\times \Ran X$$ 
of \eqref{e:right inverse} is the diagonal map. 
Thus, the semi-group $\Ran X$ has the feature of square map is equal to the identity.

\ssec{The ``independent" category of D-modules}

Throughout this subsection we let $Z$ be as in \secref{sss:functors acted on by Ran}.

\sssec{}

Consider the resulting category $\fD(\CY_{\on{indep}})$. By definition
$$\fD(\CY_{\on{indep}})\simeq \underset{\bDelta_s}{lim}\, \fD^!(\CY_{\bDelta_s}),$$
where $\fD^!(\CY_{\bDelta_s})$ is the functor
$$\bDelta_s\overset{\CY_{\bDelta_s}}\longrightarrow 
\on{PreStk}^{op}\overset{\fD^!_{\on{PreStk}}}\longrightarrow \StinftyCat_{\on{cont}}.$$

\medskip

Consider the forgetful functor
\begin{equation} \label{e:forget unital}
\fD(\CY_{\on{indep}})\to \fD(\CY).
\end{equation}

It turns out that the functor \eqref{e:forget unital} is often fully faithful. 

\sssec{}  

Assume that $\CY$ is such that there exists a map $\CY\to \Ran X\times \CY$,
such that its compositions with both the action map $$\on{unit}_{\Ran}:\Ran X\times \CY\to \CY$$ and 
the projection $\Ran X\times \CY\to \CY$ are the identity maps on $\CY$. As we have seen
in \secref{sss:right inverse}, this happens if the unital structure on $Z$ can be upgraded to
a strongly unital one. 

\medskip

\begin{prop} \label{p:incl of unital}
Under the above circumstances,  the forgetful functor 
$$\fD(\CY_{\on{indep}})\to \fD(\CY)$$
is fully faithful.
\end{prop}

\begin{proof} \hfill

\medskip

\noindent{\bf Step 1.} Consider the functor
$$\on{unit}_{\Ran}^!:\fD(\CY)\to \fD(\Ran X\times \CY)\simeq
\fD(\Ran X)\otimes \fD(\CY)$$
(the last isomorphism is due to \eqref{e:prod equiv pseudo}).

\medskip

Consider also the map $(p_{\Ran X}\times \on{id}):\Ran X\times \CY\to \CY$
and the corresponding functor
$$
(p_{\Ran X}\times \on{id})^!:\fD(\CY)\simeq \Vect\otimes {}\,\fD(\CY)\to 
\fD(\Ran X\times \CY) \simeq \fD(\Ran X)\otimes \,\fD(\CY).$$

Note, however, that by \thmref{t:contractibility of Ran}, the latter functor is fully faithful.

\medskip

Let $\fD(\CY)'$ be the \emph{full subcategory} of
$\fD(\CY)$ spanned by objects, whose image under $\on{unit}_{\Ran}^!$
lies in the essential image of $(p_{\Ran X}\times \on{id})^!$. 

\medskip

It is clear that the forgetful functor $\fD(\CY_{\on{indep}})\to \fD(\CY)$ factors as
$$\fD(\CY_{\on{indep}})\to \fD(\CY)'\to \fD(\CY).$$
We will show that the above functor $\fD(\CY_{\on{indep}})\to \fD(\CY)'$
is an equivalence.

\medskip

\noindent{\bf Step 2.} It is clear that the assignment 
$$[n]\mapsto \fD(\CY)'=\Vect^{\otimes n}\otimes \,\fD(\CY)'
\subset \fD(\Ran X)^{\otimes n}\otimes \fD(\CY)$$
extends to a functor
$$\fD^!(\CY_{\bDelta_s})':\bDelta_s\to \StinftyCat,$$
equipped with a natural transformation
$$\fD^!(\CY_{\bDelta_s})'\Rightarrow \fD^!(\CY_{\bDelta_s}),$$
which is fully faithful for every $[n]\in \bDelta_s$. 

\medskip

In particular, the resulting functor
\begin{equation} \label{e:tot of prime}
\fD(\CY_{\on{indep}})':=\underset{\bDelta_s}{lim}\, \fD^!(\CY_{\bDelta_s})'\to
\underset{\bDelta_s}{lim}\, \fD^!(\CY_{\bDelta_s})=\fD(\CY_{\on{indep}}).
\end{equation}
is also fully faithful. 

\medskip

However, it is easy to see that for every $n$ and every object $\fD(\CY_{\on{indep}})$, its evaluation 
on $[n]\in \bDelta_s$ belongs to $\fD^!(\CY_{[n]})'$. So, the functor \eqref{e:tot of prime} is an equivalence.

\medskip

The composition
$$\fD(\CY_{\on{indep}})'\to \fD(\CY_{\on{indep}})\to \fD(\CY)'$$
is functor of evaluation on $[0]\in \bDelta_s$. Thus, we obtain that is sufficient to show
that the above evaluation functor is an equivalence. 

\medskip

\noindent{\bf Step 3.} We claim that for any map $\phi:[0]\to [n]$ in $\bDelta_s$, and the corresponding
map $\phi_\CY:(\Ran X)^{\times n}\times \CY\to \CY$,
the functor
$$\phi_\CY^!:\fD^!(\CY_{[0]})'\to \fD^!(\CY_{[n]})',$$
is an equivalence. In fact, both categories in question are identified with $\fD(\CY)$, and we claim 
that the above functor is canonically isomorphic to the identity functor.

\medskip

There are two cases: if $\phi$ is the map $0\mapsto 0\in \{0,1,...,n\}$, then $\phi_\CY$
is the projection map $(\Ran X)^{\times n}\times \CY\to \CY$, and there
is nothing to prove. 

\medskip

If $\phi$ is any other map, we claim that the map $\phi_\CY$
admits a canonical right inverse, denoted $\psi_\CY$. Namely, $\psi_\CY$
is composition of the map
$\CY\to \Ran X\times \CY$ of \eqref{e:right inverse} with the $n$-fold diagonal map
$$\Ran X\times \CY\to (\Ran X)^{\times n}\times \CY.$$
By construction, the composition of $\psi_\CY$ with the projection of $(\Ran X)^{\times n}\times \CY$ 
onto $\CY$ is the identity
map on $\CY$. \footnote{Despite this fact, the structure of semi-simplicial object on
$\CY_{\bDelta_s}$ does not upgrade to a simplicial one.}

\medskip

The latter property implies that the functor $\psi_\CY^!:\fD^!(\CY_{[n]})\to \fD(\CY)$
induces the identity functor
$$\fD(\CY)=\Vect^{\otimes n}\otimes \,\fD(\CY)
\hookrightarrow \fD^!(\CY_{[n]})\overset{\psi_\CY^!}\longrightarrow \fD(\CY),$$
and hence, the composition
$$\fD^!(\CY_{[n]})'\hookrightarrow \fD^!(\CY_{[n]})\overset{\psi_\CY^!}\longrightarrow \fD(\CY)$$
is also the identity map of $\fD^!(\CY_{[n]})'$ onto $\fD(\CY)'\subset \fD(\CY)$. 

\medskip

Hence, it is enough to show that the composition
$$\fD(\CY)'=\fD^!(\CY_{[0]})'\overset{\phi_\CY^!}\longrightarrow\fD^!(\CY_{[n]})'\hookrightarrow \fD^!(\CY_{[n]})
\overset{\psi_\CY^!}\longrightarrow \fD(\CY)$$
is also isomorphic to the identity map onto $\fD(\CY)'\subset \fD(\CY)$. However, this follows
from the fact that $\psi_\CY^!\circ \phi_\CY^!=\on{Id}_{\fD(\CY)}$.

\medskip

Thus, we obtain that the semi-simplicial category $\fD^!(\CY_{\bDelta_s})'$ consists
of equivalences, and since the index category $\bDelta_s$ is contractible, we obtain
that evaluation on $[0]$ is an equivalence.

\end{proof}

\sssec{}

As a corollary of \propref{p:incl of unital} we obtain the following:

\medskip

Let $\CF_{\on{indep}}$ be an object of $\fD(\CY_{\on{indep}})$, and let $\CF\in \fD(\CY)$ be its image 
under the forgetful functor $\fD(\CY_{\on{indep}})\to \fD(\CY)$.  Assume that $\CF\in \fD(\CY)_{\on{good}\, \on{for}\,p_\CY}$,
see \secref{sss:intr good} for the notation. 

\medskip

Under these circumstances we have:

\begin{cor} \label{c:unital vs nonunital}
The natural map
$$\Gamma_{\dr,c}(\CY,\CF)\to
\Gamma_{\dr,c}(\CY_{\on{indep}},\CF_{\on{indep}})$$
is an isomorphism; in particular, the right-hand side is defined.
\end{cor}

As a particular case, we obtain:

\begin{cor}
For $X$ and $Z$ as in \propref{p:incl of unital}, the trace map
$$H_\bullet(\CY) \to H_\bullet(\CY_{\on{indep}})$$
is an isomorphism.
\end{cor}

\sssec{Example}

Let us calculate the category $\fD(Z_{\Ran X,\on{indep}})$ for $Z=X^{\fset}$; we shall denote it
by $\fD(\Ran X,\on{indep})$. 

\medskip

Note that we expect this category to be $\Vect$: the idea of the ``indep"
category was to get rid of the dependence on the Ran space, so we are
supposed to be dealing with D-modules on the Ran space that ``do not depend on the Ran
space variable". So, let us see that this is indeed the case. 

\medskip

We have an obvious functor 
\begin{equation} \label{e:Vect to unital}
\Vect\to \fD(\Ran X,\on{indep}) 
\end{equation}
that sends $k\mapsto \omega_{\Ran X}$.

\medskip

By \thmref{t:contractibility of Ran}, the composed functor 
$$\Vect\to \fD(\Ran X,\on{indep})\to \fD(\Ran X)$$ is 
fully faithful. Hence, by \propref{p:incl of unital}, the above functor 
\eqref{e:Vect to unital} is also fully faithful. Thus, it remains to see
that it is essentially surjective. 

\medskip

For $\CF\in \fD(\Ran X)$ consider the object
$$\on{union}^!(\CF)\in \fD(\Ran X\times \Ran X).$$
If $\CF$ lies in the essential image of the forgetful
functor $\fD(\Ran X,\on{indep})\to \fD(\Ran X)$, then
$$\on{union}^!(\CF)\simeq \omega_{\Ran X}\boxtimes \CF.$$
However, since the multiplication on $\Ran X$ is commutative,
we also have
$$\on{union}^!(\CF)\simeq \CF\boxtimes \omega_{\Ran X},$$
so we obtain an isomorphism
\begin{equation} \label{e:sym}
\omega_{\Ran X}\boxtimes \CF\simeq \CF\boxtimes \omega_{\Ran X}.
\end{equation}

\medskip

By \thmref{t:contractibility of Ran}, for an object $\CF_1\in \fD(\Ran X\times \Ran X)$ in the
essential image of the functor 
\begin{multline*}
(\on{id}\times p_{\Ran X})^!:\fD(\Ran X)\simeq \fD(\Ran X)\otimes \Vect\overset{\on{Id}\boxtimes \,p^!_{\Ran X}}
\longrightarrow \fD(\Ran X)\otimes \fD(\Ran X)\simeq \\
\simeq \fD(\Ran X\times \Ran X),
\end{multline*}
we have 
$$\CF_1\simeq (\on{id}\times p_{\Ran X})_!(\CF_1)\boxtimes \omega_{\Ran X}.$$

\medskip

Hence, from \eqref{e:sym}, we obtain that
$$\omega_{\Ran X}\boxtimes \CF\simeq \omega_{\Ran X\times \Ran X}\otimes \Gamma_{\dr,c}(\Ran X,\CF).$$
Since the functor $(p_{\Ran X}\times \on{id})^!$ is fully faithful, the latter isomorphism implies
that $$\CF\simeq \omega_{\Ran X}\otimes \Gamma_{\dr,c}(\Ran X,\CF),$$ as required.

\ssec{Spaces over $\Ran X$}  \label{ss:over Ran} \hfill

\medskip

In this subsection we specialize to the setting of the Example in \secref{sss:strong basic}.
Namely, $\CS=(\fset)^{op}$, $f_\CS=\on{Id}$ and the natural transformation $\bd$ 
being the canonical map 
$$(\on{id}\sqcup \on{id}):I\sqcup I\to I$$
for $I\in \fset$. 

\medskip

As we have seen above, $X^{\fset}$ and 
${\mathbf{Maps}}(X,Y)^{\on{rat}}_{X^{\fset}}$ provide examples of functors $Z$
equipped with a strong unital structure.  

\sssec{}  \label{sss:over Ran}

We shall consider functors $Z:(\fset)^{op}\to \indSch$ equipped with a strongly unital structure 
with respect to $(f_\CS,\bd)$ specified above. I.e., these are functors
$$I\rightsquigarrow Z(I),$$
equipped with a functorial assignment
to every pair of non-empty finite sets $J$ and $I$ of a map
$$\on{unit}_{J,I}:X^J\times Z(I)\to Z(J\sqcup I),$$
and a natural transformation 
$$I\rightsquigarrow (Z(I)\overset{f(I)}\to X^I),$$
such that the condition from \secref{sss:strong axiom} holds. Explicitly, this condition says that
for every $I\in \fset$, the diagram
$$
\CD
Z(I)   @>{f(I)\times \on{id}}>>  X^I\times Z(I)  \\
@V{\on{id}}VV   @VV{\on{unit}_{I,I}}V \\
Z(I)    @>>>   Z(I\sqcup I)
\endCD
$$
commutes, where the bottom arrow corresponds to the map $(\on{id}\sqcup \on{id}):I\sqcup I\to I$
in $\fset$.  

\medskip

We shall impose the following two additional conditions:

\begin{enumerate}

\item For $I,J\in \fset$, the diagram
$$
\CD
X^J\times Z(I)  @>{\on{unit}_{J,I}}>>  Z(J\sqcup I) \\
@V{\on{id}\times f(I)}VV  @VV{f(J\sqcup I)}V  \\
X^J\times X^I  @>{\sim}>>  X^{J\sqcup I}
\endCD
$$
commutes. 

\item For every arrow in $(\fset)^{op}$,
i.e., a surjective map of finite sets $I_1\twoheadrightarrow I_2$, the resulting map
$$Z(I_2)\to X^{I_2}\underset{X^{I_1}}\times Z(I_1)$$
is an isomorphism.

\end{enumerate} 

The examples of $Z$ being $X^{\fset}$ and ${\mathbf{Maps}}(X,Y)^{\on{rat}}_{X^{\fset}}$ 
satisfy these conditions.

\medskip

We let $Z_{\Ran X}$ denote the corresponding space
$$\underset{(\fset)^{op}}{colim}\, Z.$$
By \secref{sss:functors acted on by Ran}, the space $Z_{\Ran X}$ is acted on by the semi-group
$\Ran X$.

\begin{remark}
As in Remark \ref{r:maps discrete}, the prestack $Z_{\Ran X}$, considered
as a functor $$(\affSch)^{op}\to \inftygroup,$$ actually takes values in $\on{Set}\subset \inftygroup$.
Indeed, for $S\in \affSch$, the infinity-groupoid $\on{Maps}(S,Z_{\Ran X})$ maps to the \emph{set}
of non-empty finite subsets of $\on{Maps}(S,X)$, and for $\ol{x}\subset \on{Maps}(S,X)$, the fiber
over it is $Z(\ol{x})\underset{X^{\ol{x}}}\times \on{pt}$, where $\on{pt}\to X^{\ol{x}}$ is the
tautological point of $X^{\ol{x}}$ corresponding to $\ol{x}$.

\medskip

For this property it was not
necessary that the functor $Z$ take values in indschemes, we only need that each $Z(I)$,
considered as a functor $(\affSch)^{op}\to \inftygroup$, take values in the subcategory 
$\on{Set}\subset \inftygroup$.
\end{remark}

\sssec{}  \label{sss:epic diagram}

We have:

\begin{lem} \label{l:huge}
For any (not necessarily surjective) map of finite sets $\phi:J\to I$, the 
diagram
$$
\CD
Z(I)   @>{\Delta(\phi)\circ f(I)\times \on{id}}>>  X^J\times Z(I)  \\
@V{\on{id}}VV   @VV{\on{unit}_{J,I}}V \\
Z(I)    @>>>   Z(J\sqcup I)
\endCD
$$
commutes as well, where the bottom arrow corresponds to the map $(\phi\sqcup \on{id}):J\sqcup I\to I$
in $(\fset)^{op}$.
\end{lem}

\begin{proof}

It is enough to show that the two maps $Z(I) \rightrightarrows Z(J\sqcup I)$ 
in question, composed with the map $Z(J\sqcup I)\to Z(J\sqcup I\sqcup I)$, corresponding to the 
obvious surjection $J\sqcup I\sqcup I\to J\sqcup I$, coincide. The latter follows by chasing
through the following diagram, in which every quadrangle and triangle are commutative
(the dotted arrows are the original arrows in the lemma):

\begin{gather}  
\xy
(-5,0)*+{X^I\times Z(I)}="A";
(40,0)*+{Z(J\sqcup I\sqcup I)}="B";
(0,20)*+{Z(I\sqcup I)}="C";
(0,40)*+{Z(I)}="D";
(0,60)*+{Z(J\sqcup I)}="J";
(0,-20)*+{X^{J\sqcup I}\times Z(I)}="E";
(40,-20)*+{X^J\times Z(I\sqcup I)}="F";
(-60,-40)*+{Z(I)}="G";
(70,-40)*+{X^J\times Z(I)}="H";
(70,-20)*+{Z(J\sqcup I)}="I";
{\ar@{->} "G";"A"};
{\ar@{->} "G";"C"};
{\ar@{->} "G";"E"};
{\ar@{-->} "G";"H"};
{\ar@{->} "G";"F"};
{\ar@{->} "C";"B"};
{\ar@{->} "D";"B"};
{\ar@{->} "E";"B"};
{\ar@{->} "F";"B"};
{\ar@{->} "I";"B"};
{\ar@{->} "H";"F"};
{\ar@{-->} "H";"I"};
{\ar@{->} "A";"C"};
{\ar@{->} "A";"E"};
{\ar@{->} "E";"F"};
{\ar@{->} "D";"J"};
{\ar@{->} "D";"C"};
{\ar@{->} "G";"D"};
{\ar@{-->} "G";"J"};
{\ar@{->} "J";"B"};
\endxy
\end{gather}

\end{proof}

\begin{remark}
The compexity of the above diagram leads one to wonder, in the situation when $Z(I)$'s are
no longer indschemes, but more general prestacks (so, we can no longer talk about equality of maps
but homotopy equivalences), whether the constructed identification of the two maps 
$Z(I) \rightrightarrows Z(J\sqcup I)$ is canonical. I.e., whether a different diagram would not produce
a different identification.

\medskip

The answer is that this isomorphism is canonical. Indeed, consider the diagram
$$
\CD
X^I\underset{X^{J\sqcup I}}\times (X^J\times Z(I))  @>{\on{unit}_{J,I}}>>  X^I\underset{X^{J\sqcup I}}\times Z(J\sqcup I)  \\
@AAA   @AAA   \\
Z(I)   @>{\alpha(I)}>>   Z(I),
\endCD
$$
where the left vertical arrow is an isomorphism tautologically, and the right vertical arrow is an isomorphism by the assumption on $Z$
(see \secref{sss:over Ran}).

\medskip

The assertion of \lemref{l:huge} says that the bottom map $\alpha(I)$ is the identity map on $Z(I)$. However, it was enough to
show that $\alpha(I)$ is an isomorphism. Indeed, the associativity of the action of $X^{\fset}$ on $Z$ would then
imply that $\alpha(I)\circ \alpha(I)$ is canonically homotopic to $\alpha(I)$, implying that $\alpha(I)\sim \on{id}$.
\end{remark}

\sssec{}  \label{sss:extended funct}

One can interpret the assertion of \lemref{l:huge} as follows. Let $Z:(\fset)^{op}\to \indSch$ be 
as in \secref{sss:over Ran}. We claim that for a pair of non-empty finite sets $I$ and $J$ and \emph{any}
(i.e., not necessarily surjective) map $I\to J$, we have a well-defined map
$$X^J\underset{X^I}\times Z(I)\to Z(J),$$
compatible with the projection to $X^J$, and compatible with compositions.

\medskip

Indeed, when $I\to J$ is surjective, the map in question is the isomorphism of Condition (2) in \secref{sss:over Ran}.
When $I\to J$ is an injection $I\hookrightarrow I\sqcup K\simeq J$, the map in question is
$$X^J\underset{X^I}\times Z(I)\simeq X^K\times Z(I)\overset{\on{unit}_{K,I}}\longrightarrow Z(K\sqcup I).$$

The compatibility with compositions is assured by \lemref{l:huge}.

\ssec{Ran space with marked points}  \label{ss:marked points}

In this subsection we will consider a certain class of statements that  
a strongly unital structure allows one to prove.

\sssec{}   \label{sss:marked points}

Let $\sA$ be a finite set. Let $\fset_\sA$ be the category
whose objects are finite sets $I$, equipped with a map $\sA\to I$, and where the morphisms are
surjective maps $I_1\twoheadrightarrow I_2$, for which the diagram
$$
\CD
I_1 @>>>  I_2 \\
@AAA   @AAA  \\
\sA  @>{\on{id}}>> \sA
\endCD
$$
commutes. 

\medskip

The category $\fset_\sA$ is naturally a module over $\fset$ via operation of disjoint union. The corresponding functor
$f_{\CS}$ is the forgetful functor $\oblv_\sA:\fset_\sA\to \fset$. The natural transformation $\bd$
is given by the canonical map $(\on{id}\sqcup \on{id}):I\sqcup I\to I$. 

\sssec{}  

The first example of a functor $Z:(\fset_\sA)^{op}\to \indSch$ satisfying the requirements 
of \secref{sss:strong axiom} is constructed as follows:

\medskip

Let $x^\sA$ be a $k$-point of $X^\sA$. Let $X^{\fset_\sA}$ be the functor $(\fset_\sA)^{op}\to \Sch$ given by
$$I\mapsto X^I_{\sA}:=x^\sA\underset{X^\sA}\times X^I.$$

\medskip

Let $$\Ran X_{\sA}:=\underset{(\fset_\sA)^{op}} {colim}\, X^{\fset_\sA}.$$
We shall refer to $\Ran X_{\sA}$ as the \emph{relative version of $\Ran X$ with marked points}. 

\begin{rem}
For $S\in \affSch$, the $\infty$-groupoid $\on{Maps}(S,\Ran X_{\sA})$ is in fact a set isomorphic
to that of non-empty finite subsets $\ol{x}\subset \on{Maps}(S,X)$ that contain as a subset the image
of $\sA$ under
$$\sA\to \on{Maps}(\on{pt},X)\to \on{Maps}(S,X).$$
\end{rem}

\sssec{}  \label{sss:forget A}

Let $Z$ be a functor $(\fset)^{op}\to \indSch$ as in \secref{sss:over Ran}. For $(\sA,x^\sA)$ as above, let
$Z_\sA$ be a functor $(\fset_\sA)^{op}\to \indSch$ given by
$$I\mapsto Z(I)_\sA:=X^I_{\sA}\underset{X^I}\times Z(I).$$

The strongly unital structure on $Z$ induces one on $Z_\sA$. Set 
$$Z_{\Ran X_\sA}:=\underset{(\fset_\sA)^{op}}{colim}\, Z_\sA\in \on{PreStk}.$$
We obtain that $Z_{\Ran X_\sA}$ becomes a module over $\Ran X$. 

\sssec{}

By \secref{sss:expl map}, the forgetful functor $\oblv_\sA:\fset_\sA\to \fset$ and the natural
transformation $Z_\sA\Rightarrow Z\circ (\oblv_\sA)^{op}$, given by 
$Z(I)_\sA\to Z(I)$, define a map 
\begin{equation} \label{e:forget A}
Z_{\Ran X_\sA}\to Z_{\Ran X}.
\end{equation}

Moreover, it is easy to see that 
the map \eqref{e:forget A} is compatible with the $\Ran X$-actions. Therefore, it gives rise to a map
of the corresponding semi-simplicial objects
$$Z_{\Ran X_\sA,\bDelta_s}\to Z_{\Ran X,\bDelta_s}$$
and 
\begin{equation} \label{e:marked points}
Z_{\Ran X_\sA,\on{indep}}\to Z_{\Ran X,\on{indep}}. 
\end{equation}

\medskip

We will prove:

\begin{prop}  \label{p:marked points} 
The map \eqref{e:marked points} is an isomorphism.
\end{prop}

\begin{remark}
The meaning of this proposition is that in the unital context, constraining our
finite subset of points of $X$ to contain a given set $x^\sA$ does not
alter the resulting space.
\end{remark}

\sssec{}

Before we prove the proposition, let us discuss some corollaries:

\medskip

Let $Z$ be as in \secref{sss:over Ran}, and let $\CF$ be an object of $\fD(Z_{\Ran X})$, which lies in the essential image 
of the forgetful functor $\fD(Z_{\Ran X,\on{indep}})\to \fD(Z_{\Ran X})$, and which also lies in
$\fD(Z_{\Ran X})_{\on{good}\, \on{for}\,p_{Z_{\Ran X}}}$.

\medskip

Let $(\sA,x^\sA)$ be as in \secref{sss:marked points}. Let $\CF_\sA$ be the pullback of $\CF$ under
the map $Z_{\Ran X_\sA}\to Z_{\Ran X}$ of \eqref{e:forget A}. Combining 
\corref{c:unital vs nonunital} with \propref{p:marked points}, we obtain:

\begin{cor}  \label{c:marked points}
The trace map
$$\Gamma_{\dr,c}(Z_{\Ran X_\sA},\CF_\sA)\to \Gamma_{\dr,c}(Z_{\Ran X},\CF)$$
is an isomorphism.
\end{cor}

As a particular case, we have:

\begin{cor} \label{c:marked points homology}
Then the trace map
$$H_\bullet(Z_{\Ran X_\sA})\to H_\bullet(Z_{\Ran X})$$
is an isomorphism.
\end{cor}

\sssec{Proof of \propref{p:marked points}} \hfill

\medskip

\noindent{\bf Step 1.} Consider the functor $(\sA\sqcup-):\fset\to \fset_\sA$:
$$I\mapsto \sA\sqcup I.$$
The unital structure on $Z$ defines a natural transformation
$$Z\Rightarrow Z_\sA\circ (\sA\sqcup-)$$
between the resulting two functors $(\fset)^{op}\rightrightarrows \indSch$:
$$Z(I)\simeq x^\sA\times Z(I)\overset{\on{unit}_{\sA,I}}\longrightarrow Z(\sA\sqcup I)_{\sA}.$$
By \secref{sss:expl map}, we obtain a map
\begin{equation} \label{e:add A}
Z_{\Ran X}\to Z_{\Ran X_\sA}.
\end{equation}

\medskip

It is easy to see from the construction that this map is also compatible with the action
of the semi-group $\Ran X$. In particular, we obtain the map of semi-simplicial objects
$$Z_{\Ran X,\bDelta_s}\to Z_{\Ran X_\sA,\bDelta_s},$$
and a map 
\begin{equation} \label{e:marked points other map}
Z_{\Ran X,\on{indep}}\to Z_{\Ran X_\sA,\on{indep}}.
\end{equation}

\medskip

We will show that composition of the maps \eqref{e:forget A} and \eqref{e:add A}:
\begin{equation} \label{e:first comp A}
Z_{\Ran X_\sA}\to Z_{\Ran X}\to Z_{\Ran X_\sA}
\end{equation}
is (canonically homotopic to) the identity map on $Z_{\Ran X_\sA}$, in a way
compatible with the $\Ran X$-action. The latter compatibility will imply that the composition
$$Z_{\Ran X_\sA,\on{indep}}\to Z_{\Ran X,\on{indep}}\to Z_{\Ran X_\sA,\on{indep}}$$
is also (canonically homotopic to) the identity map. 

\medskip

We will then show that the composition
\begin{equation} \label{e:second comp A unital}
Z_{\Ran X,\on{indep}}\to Z_{\Ran X_\sA,\on{indep}}\to Z_{\Ran X,\on{indep}}
\end{equation}
is (canonically homotopic to) the identity map as well. 

\medskip

\noindent{\bf Step 2.} 
The assertion regarding the composition \eqref{e:first comp A} follows from the paradigm described
in \secref{sss:paradigm nat trans}:

\medskip

We apply it to $\bC_1=\bC_2=(\fset_\sA)^{op}$, $\Phi_1=\Phi_2=Z_\sA$, $F$ being the identity
functor and $F'$ being the functor 
$$(\sA\to I)\rightsquigarrow (\sA\to \sA\sqcup I),$$
where $\sA\overset{\on{id}}\to \sA\hookrightarrow \sA\sqcup I$. 

\medskip

The natural transformation
$\beta$ is the given by the map $\sA\sqcup I\to I$ in $\fset_\sA$. The fact that the natural transformation
$$Z_\sA\Rightarrow Z_\sA\circ (\sA\sqcup-)\circ (\oblv_\sA)^{op}$$
is isomorphic to $Z_\sA\circ \beta$ follows from \lemref{l:huge}.

\medskip

The naturality of the construction implies the compatibility with the action of the semi-group $\Ran X$. 

\medskip

\noindent{\bf Step 3.} In order to prove that the composed map \eqref{e:second comp A unital}
is (canonically homotopic to) the identity map, since the semi-group $\Ran X$ is commutative, 
 it suffices to show that the composition
$$Z_{\Ran X}\overset{\text{\eqref{e:add A}}}\longrightarrow Z_{\Ran X_\sA}
\overset{\text{\eqref{e:forget A}}}\longrightarrow Z_{\Ran X}$$
factors as
$$Z_{\Ran X}\overset{\gamma}\to \Ran X\times Z_{\Ran X}\overset{\on{unit}_{\Ran}}\longrightarrow Z_{\Ran X}$$
for some map $\gamma:Z_{\Ran X}\to \Ran X\times Z_{\Ran X}$, such that the composition
$$Z_{\Ran X}\overset{\gamma}\to \Ran X\times Z_{\Ran X}\overset{p_{\Ran X}\times \on{id}}\longrightarrow Z_{\Ran X}$$
is the identity map. 

\medskip

The required map $\gamma$ is given in terms of the paradigm of \secref{sss:expl map} as follows. Namely, 
the functor
$$\gamma_{\fset}:(\fset)^{op}\to (\fset\times \fset)^{op}$$ is 
$$I\rightsquigarrow (\sA,I),$$
and the natural transformation between the two functors $(\fset)^{op}\rightrightarrows \on{PreStk}$ 
$$Z\Rightarrow (X^{\fset}\times Z)\circ \gamma_{\fset}$$
sends $I\in \fset$ to the map 
$(x^\sA\times \on{id}):Z(I)\to X^\sA\times Z(I)$.

\qed

\sssec{}  \label{sss:relative version}

We conclude this subsection with the observation that the assertion of Proposition \ref{p:marked points}
(with the same proof) holds also in the relative situation:

\medskip

Let $S$ be a scheme and let $x^\sA$ be an $S$-point of $X^\sA$. Then we can consider
a relative version of the functor $X^{\fset_\sA}$, which now maps $(\fset_\sA)^{op}$
to $\Sch_{/S}$:
$$I\rightsquigarrow X^I_\sA:=S\underset{X^\sA}\times X^I,$$
and the corresponding prestack $\Ran X_\sA$ over $S$. 

\medskip

Given a functor $Z:(\fset)^{op}\to \indSch_{/S}$ satisfying the conditions of \secref{sss:over Ran}
(with the target category $\indSch$ replaced by $\indSch_{/S}$), we let $Z_\sA$ be the functor
$(\fset_\sA)^{op}\to \indSch_{/S}$ given by
$$I\rightsquigarrow S\underset{S\times X^\sA}\times Z(I)\simeq X^I_\sA\underset{S\times X^I}\times Z(I).$$

\medskip

Set
$$Z_{\Ran X_\sA}:=\underset{(\fset_\sA)^{op}}{colim}\, Z_\sA.$$
Let $p_\sA$ (resp., $p$) denote the map $Z_{\Ran X_\sA}\to S$
(resp., $Z_{\Ran X}\to S$). 

\medskip

We have the following relative version of \corref{c:marked points homology}:

\begin{cor} \label{c:marked points homology rel}
For $\CF_S\in \fD(S)$, the map
$$(p_\sA)_!\circ (p_\sA)^!(\CF_S)\to p_!\circ p^!(\CF_S)$$
is an isomorphism.
\end{cor}

\ssec{Products over $\Ran X$}  \label{ss:products}

In this subsection we will give another example of a statement that a unital structure allows
to prove.

\sssec{}

Let $Z^1$ and $Z^2$ be two functors $(\fset)^{op}\to \indSch$ as in \secref{sss:over Ran}.
Let $Z$ be yet another functor defined by
$$Z(I):=Z^1(I)\underset{X^I}\times Z^2(I).$$
Then $Z$ is also a functor satisfying the requirements of \secref{sss:over Ran}.

\medskip

Consider the corresponding spaces $Z^1_{\Ran X}$, $Z^2_{\Ran X}$, and $Z_{\Ran X}$.

\sssec{}

Note that we have a naturally defined map
\begin{equation} \label{e:to prod}
Z_{\Ran X}\to Z^1_{\Ran X}\times Z^2_{\Ran X}.
\end{equation}

\medskip

In terms of \secref{sss:expl map}, the map in question corresponds to the diagonal
functor $${\mathbf {diag}}^{op}:(\fset)^{op}\to (\fset\times \fset)^{op},$$ and the natural transformation
between the resulting two functors $(\fset)^{op}\rightrightarrows \indSch$:
$$Z\to (Z^1\times Z^2)\circ {\mathbf {diag}}^{op},$$
given by
$$I\rightsquigarrow \Bigl(Z(I)=Z^1(I)\underset{X^I}\times Z^2(I)\to Z^1(I)\times Z^2(I)\Bigr).$$

\medskip

It is easy to see that the map \eqref{e:to prod} is compatible with the action of
$\Ran X$ via the diagonal homomorphism $\Ran X\to \Ran_X\times \Ran _X$ and the action 
of $\Ran_X\times \Ran _X$ on $Z^1_{\Ran X}\times Z^2_{\Ran X}$. 

\sssec{}

We claim now that there exists a canonically defined map
\begin{equation} \label{e:from prod}
Z^1_{\Ran X}\times Z^2_{\Ran X}\to Z_{\Ran X}.
\end{equation}

\medskip

Namely, in terms of \secref{sss:expl map}, the map in question corresponds to the functor 
$$\on{union}^{op}:(\fset\times \fset)^{op}\to (\fset)^{op},$$
and the natural transformation between the resulting two functors
$(\fset\times \fset)^{op}\rightrightarrows \indSch$:
$$(Z^1\times Z^2)\to Z\circ \on{union}^{op},$$
given by sending $(I_1,I_2)\in \fset\times \fset$ to the map
\begin{multline*}
Z^1(I_1)\times Z^2(I_2)\simeq (Z^1(I_1)\times X^{I_2})\underset{X^{I_1}\times X^{I_2}}\times
(X^{I_1}\times Z(I_2))\to \\
\to Z^1(I_1\sqcup I_2)\underset{X^{I_1}\times X^{I_2}}\times
Z^2(I_1\sqcup I_2)=Z(I_1\sqcup I_2).
\end{multline*}

\medskip

It is easy to see that the map \eqref{e:from prod} is compatible with the action of
$\Ran X$ via the homomorphism 
$$\on{union}:\Ran X\times \Ran_X\to \Ran _X$$ and the action 
of $\Ran_X\times \Ran _X$ on $Z^1_{\Ran X}\times Z^2_{\Ran X}$. 

\medskip

Thus, the maps \eqref{e:to prod} and \eqref{e:from prod} give rise to maps 
\begin{equation} \label{e:prod unital}
Z_{\Ran X,\on{indep}}\rightleftarrows Z^1_{\Ran X,\on{indep}}\times Z^2_{\Ran X,\on{indep}}.
\end{equation}

\begin{prop} \label{p:prod}
The maps \eqref{e:prod unital} are mutually inverse isomorphisms.
\end{prop}

We omit the proof, because it is completely analogous to the one of
\propref{p:marked points}.

\sssec{}

\propref{p:prod} admits corollaries analogous to those of \propref{p:marked points}:

\medskip

Let $\CF$ be an object of $\fD(Z^1_{\Ran X}\times Z^2_{\Ran X})$ that lies in the
essential image of the forgetful functor
$$\fD(Z^1_{\Ran X,\on{indep}}\times Z^2_{\Ran X,\on{indep}})\to \fD(Z^1_{\Ran X}\times Z^2_{\Ran X}).$$
Assume also that $\CF$ belongs to 
$\fD(Z^1_{\Ran X}\times Z^2_{\Ran X})_{\on{good}\, \on{for}\,p_{Z^1_{\Ran X}\times Z^2_{\Ran X}}}$.

\medskip

Let $\CF^{{\mathbf {diag}}}$ denote the pullback of $\CF$ to $Z_{\Ran X}$ under the map \eqref{e:from prod}.
We have:
\begin{cor} \label{c:prod}
The trace map
$$\Gamma_{\dr,c}\left(Z_{\Ran X},\CF^{{\mathbf {diag}}}\right)\to \Gamma_{\dr,c}\left(
Z^1_{\Ran X}\times Z^2_{\Ran X},\CF\right)$$
is an isomorphism.
\end{cor}

In particular,
\begin{cor}
The trace map
$$H_\bullet(Z_{\Ran X})\to H_\bullet\left(Z^1_{\Ran X}\times Z^2_{\Ran X}\right)$$
is an isomorphism.
\end{cor}

\ssec{Pairs of finite sets} \label{ss:pairs}

In this subsection we shall consider another family of examples of the situation
described in \secref{ss:strongly unital}. 

\sssec{} \label{sss:categ of pairs}

We consider the category $\fset^{\to}$ whose objects can be depicted as $(J\to I)$, i.e.,
pairs of finite sets, equipped with a map between them (the map $J\to I$ is arbitrary, 
i.e., it is not required to be either injective or surjective), and morphisms are
commutative diagrams
$$
\CD
J_1  @>>>  I_1 \\
@VVV   @VVV  \\
J_2  @>>>  I_2,
\endCD
$$
where vertical maps are surjective. 



\medskip

This category is acted on by $\fset$ via
$$K\sqcup (J\to I):=(K\sqcup J\to K\sqcup I).$$

\medskip

We have two naturally defined functors ${\mathbf {pr}}_{\on{source}}$ and ${\mathbf {pr}}_{\on{target}}$
from $\fset^\to$ to $\fset$, both compatible with the monoidal action. 

\sssec{}

In this subsection we shall take the index category $\CS$ to be $(\fset^\to)^{op}$. 
We will take the functor
$$f_\CS:(\fset^\to)^{op}\to (\fset)^{op}$$
to be $({\mathbf {pr}}_{\on{source}})^{op}$. 

\medskip

The natural
transformation $\bd$ is given by the map in $\fset^\to$:
$$(J\sqcup J\to J\sqcup I)\to (J\to I).$$

\sssec{}

We shall consider functors $Z:(\fset^\to)^{op}\to \indSch$ equipped with a strongly unital structure
with respect to $(f_\CS,\bd)$ specified above. I.e., these are functors
$$(J\to I)\rightsquigarrow Z(J\to I),$$
equipped with a functorial assignment
$$\on{unit}_{K,J\to I}:X^K\times Z(J\to I)\to Z(K\sqcup J\to K\sqcup I),$$
and a natural transformation 
\begin{equation} \label{e:source nat trans}
(J\to I)\rightsquigarrow (Z(J\to I)\overset{f(J\to I)}\longrightarrow X^J),
\end{equation}
such that the condition from \secref{sss:strong axiom} holds. 

\medskip

Let $Z_{\Ran^\to X}$ denote the object of $\on{PreStk}$ equal to
$$\underset{(\fset^\to)}{colim}\, Z_{\Ran^\to X}.$$

The space $Z_{\Ran^\to X}$ is acted on by the semi-group $\Ran X$.
We shall consider the corresponding semi-simplicial object $Z_{\Ran^\to X,\bDelta_s}$
and its geometric realization $Z_{\Ran^\to X,\on{indep}}$.

\sssec{}

The prime example of such functor
is $X^{\fset^{\to}}$ given by 
$$(J\to I)\rightsquigarrow X^{J\to I}:=X^I,$$
i.e., $X^{\fset^{\to}}=X^{\fset}\circ ({\mathbf {pr}}_{\on{target}})^{op}$.

\medskip

Set $$\Ran^{\to} X:=\underset{(\fset^{\to})^{op}}{colim}\, X^{\fset^{\to}}.$$

\sssec{}   \label{sss:over pairs}

We shall impose the following additional structure on our functors 
$$Z:(\fset^\to)^{op}\to \indSch:(J\to I)\rightsquigarrow Z(J\to I).$$

Namely, we assume that $Z$ be equipped with a natural transformation $f^\to_Z:Z\to X^{\fset^{\to}}$. I.e.,
for every $(J\to I)$ we have a map
$$f^\to(J\to I):Z(J\to I)\to X^I,$$
such that the map $f(J\to I)$ of \eqref{e:source nat trans} equals the composition
$$Z(J\to I)\overset{f^\to(J\to I)}\longrightarrow X^I\to X^J.$$

\medskip

We require that the following additional conditions, analogous to those of \secref{sss:over Ran}
hold:

\begin{enumerate}

\item For $I,J,K\in \fset$, the diagram
$$
\CD
X^K\times Z(J\to I)  @>{\on{unit}_{K,J\to I}}>>  Z(K\sqcup J\to K\sqcup I) \\
@V{\on{id}\times f^\to(J\to I)}VV  @VV{f^\to (K\sqcup J\to K\sqcup I)}V  \\
X^K\times X^I  @>{\sim}>>  X^{K\sqcup I}
\endCD
$$
commutes. 

\item For every arrow $(J_1\to I_1)\to (J_2\to I_2)$ in $\fset^\to$, the resulting
map
$$Z(J_2\to I_2)\to X^{I_2}\underset{X^{I_1}}\times Z(J_1\to I_1)$$
be an isomorphism. 
in $(\fset^\to)^{op}$,

\end{enumerate}



\sssec{}

Note now that we have a functor ${\mathbf{diag}}:\fset\to \fset^\to$, given by
$$I\rightsquigarrow (I\overset{\on{id}}\to I).$$

\medskip

For $Z$ as in \secref{sss:over pairs}, let $Z^{{\mathbf{diag}}}$ denote the functor
$$Z\circ {\mathbf{diag}}^{op}:(\fset)^{op}\to \indSch.$$

\medskip

Note that if $Z$ is strongly unital and satisfies the conditions of \secref{sss:over pairs},
then $Z^{{\mathbf{diag}}}$ is also unital and satisfies the conditions of
\secref{sss:over Ran}. 

\medskip

By \secref{sss:expl map}, we have a map
\begin{equation} \label{e:from diag}
Z^{{\mathbf{diag}}}_{\Ran X}\to Z_{\Ran^\to X},
\end{equation}
which is easily seen to be compatible with the action of the semi-group $\Ran X$. Hence,
it gives rise to a map
\begin{equation} \label{e:from diag unital}
Z^{{\mathbf{diag}}}_{\Ran X,\on{indep}}\to Z_{\Ran^\to X,\on{indep}}.
\end{equation}

\medskip

We shall prove:

\begin{prop} \label{p:pairs}
The map \eqref{e:from diag unital} is an isomorphism.
\end{prop}

\sssec{}

\propref{p:pairs} has corollaries analogous to those of \propref{p:marked points}:

\medskip

Let $\CF$ be an object of $\fD(Z_{\Ran^\to X})$, which lies in the essential image 
of the forgetful functor $\fD(Z_{\Ran^\to X,\on{indep}})\to \fD(Z_{\Ran^\to X})$.
Assume also that $\CF$ belongs to $\fD(Z_{\Ran^\to X})_{\on{good}\, \on{for}\,p_{Z_{\Ran^\to X}}}$.

\medskip

Let $\CF^{{\mathbf{diag}}}$ be the pullback of $\CF$ under
the map \eqref{e:from diag}. Combining 
\corref{c:unital vs nonunital} with \propref{p:pairs}, we obtain:

\begin{cor}  \label{c:pairs}
The trace map
$$\Gamma_{\dr,c}(Z^{{\mathbf{diag}}}_{\Ran X},\CF^{{\mathbf{diag}}})\to \Gamma_{\dr,c}(Z_{\Ran^\to X},\CF)$$
is an isomorphism.
\end{cor}

As a particular case, we have:

\begin{cor}  \label{c:pairs homology}
Then the trace map
$$H_\bullet(Z^{{\mathbf{diag}}}_{\Ran X})\to H_\bullet(Z_{\Ran^\to X})$$
is an isomorphism.
\end{cor}

\sssec{Proof of \propref{p:pairs}} \hfill

\medskip

\noindent{\bf Step 1.}
Consider the functor ${\mathbf {pr}}_{\on{target}}:\fset^\to\to \fset$.
We shall now construct a natural transformation
$$Z\Rightarrow Z^{{\mathbf{diag}}}\circ ({\mathbf {pr}}_{\on{target}})^{op}.$$
It is given by sending a pair $(J\to I)$ to the map
$$Z(J\to I)\simeq 
X^I\underset{X^{I\sqcup I}}\times \left(X^I\times Z(J\to I)\right)
\overset{\on{unit}_{I,J\to I}}\longrightarrow X^I\underset{X^{I\sqcup I}}\times Z(I\sqcup J\to I\sqcup I)\simeq Z^{{\mathbf{diag}}}(I),$$
where the last isomorphism corresponds to the morphism $(I\sqcup J\to I\sqcup I)\to (I\to I)$ in $\fset^\to$.

\medskip

Hence, by \secref{sss:expl map}, we obtain a map 
\begin{equation} \label{e:to diag}
Z_{\Ran^\to X}\to Z^{{\mathbf{diag}}}_{\Ran X},
\end{equation}
which is also easily seen to be compatible with the action of the semi-group $\Ran X$. Hence, we obtain a map
\begin{equation} \label{e:to diag unital}
Z_{\Ran^\to X,\on{indep}}\to Z^{{\mathbf{diag}}}_{\Ran X,\on{indep}}.
\end{equation}

\medskip

However, it follows from the condition of \secref{sss:strong axiom} that
the composition 
$$Z^{{\mathbf{diag}}}_{\Ran X}\to Z_{\Ran^\to X}\to Z^{{\mathbf{diag}}}_{\Ran X}$$
equals the identity map. Hence, so does the composition 
$$Z^{{\mathbf{diag}}}_{\Ran X,\on{indep}}\to Z_{\Ran^\to X,\on{indep}}\to Z^{{\mathbf{diag}}}_{\Ran X,\on{indep}}.$$

\medskip

\noindent{\bf Step 2.} Thus, it remains to show that the composition 
$$Z_{\Ran^\to X,\on{indep}}\to Z^{{\mathbf{diag}}}_{\Ran X,\on{indep}}\to Z_{\Ran^\to X,\on{indep}}$$
is canonically homotopic to the identity map. For that, as in the proof of \propref{p:marked points}, it
is enough to show that the composition
\begin{equation} \label{e:compos pairs one}
Z_{\Ran^\to X}\to Z^{{\mathbf{diag}}}_{\Ran X}\to Z_{\Ran^\to X}
\end{equation}
can be factored as 
$$Z_{\Ran^\to X} \overset{\gamma}\to \Ran X\times Z_{\Ran^\to X} \overset{\on{unit}_{\Ran X}}\longrightarrow
Z_{\Ran^\to X}$$
for some map $\gamma:Z_{\Ran^\to X} \overset{\gamma}\to \Ran X\times Z_{\Ran^\to X}$, such that the composition
$$Z_{\Ran^\to X} \overset{\gamma}\to \Ran X\times Z_{\Ran^\to X} \overset{p_{\Ran X}\times \on{id}}\longrightarrow
Z_{\Ran^\to X}$$
is the identity map.

\medskip

We define the map $\gamma$ as follows. It is defined in terms of \secref{sss:expl map}
by the functor $$\gamma_{\fset}:(\fset^\to)^{op}\to (\fset\times \fset^\to)^{op},$$ given by
$$(J\to I)\rightsquigarrow (I,(J\to I)),$$
and the natural transformation between the two functors $(\fset^\to)^{op}\rightrightarrows \indSch$
$$Z\Rightarrow (X^{\fset}\times Z)\circ \gamma_{\fset}$$
given by
$$Z(J\to I)\overset{f^\to(J\to I)\times \on{id}}\longrightarrow X^I\times Z(J\to I).$$

\medskip

\noindent{\bf Step 3.} Let us show that the resulting map 
\begin{equation} \label{e:compos pairs two}
\on{unit}_{\Ran X}\circ \,\gamma:Z_{\Ran^\to X}\to Z_{\Ran^\to X}
\end{equation}
is indeed homotopic to the map \eqref{e:compos pairs one}. We shall do so using the paradigm
of \secref{sss:paradigm nat trans}:

\medskip

We take $\bC_1=\bC_2=(\fset^\to)^{op}$ and $\Phi_1=\Phi_2=Z$. We take the functor
$F$ to be 
$$(J\to I)\rightsquigarrow (I\sqcup J\to I\sqcup I),$$
and the functor $F'$ to be
$$(J\to I)\rightsquigarrow (I\to I).$$

\medskip

The map \eqref{e:compos pairs two} corresponds to the natural transformation $\alpha:\Phi_1\Rightarrow \Phi_2\circ F$
that sends $(J\to I)$ to the map
$$Z(J\to I)\overset{f^\to(J\to I)\times \on{id}}\longrightarrow X^I\times Z(J\to I)
\overset{\on{unit}_{I,(J\to I)}}\longrightarrow Z(I\sqcup J,I\sqcup I).$$

\medskip

The map \eqref{e:compos pairs one} corresponds to the natural transformation $\alpha':\Phi_1\Rightarrow \Phi_2\circ F'$
that sends $(J\to I)$ to the map
$$Z(J\to I)\simeq 
X^I\underset{X^{I\sqcup I}}\times \left(X^I\times Z(J\to I)\right)
\overset{\on{unit}_{I,J\to I}}\longrightarrow X^I\underset{X^{I\sqcup I}}\times Z(I\sqcup J\to I\sqcup I)\simeq Z(I\to I).$$

\medskip

Now, the required natural transformation $\beta:F\to F'$ sends $(J\to I)$ to the map in $(\fset^\to)^{op}$
$$(I\to I)\to (I\sqcup J\to I\sqcup I)$$
opposite to the natural map in $\fset^\to$.

\qed

\section{Proof of the main theorem}






\ssec{The case of $Y=\BA^n$}



We will show directly that the assertion of \thmref{t:contractibility} holds in this case, which
is a triviality modulo \thmref{t:contractibility of Ran}.

\sssec{}

By \secref{sss:expl map}, the natural transformation of functors $(\fset)^{op}\rightrightarrows \indSch$
$${\mathbf{Maps}}(X,Y)^{\on{rat}}_{X^{\fset}}\Rightarrow X^{\fset}$$
induces a map
$$f:{\mathbf{Maps}}(X,Y)^{\on{rat}}_{\Ran X}\to \Ran X,$$
and by \eqref{e:trace map omega}, combined with \secref{sss:when good}(iii), a map
\begin{equation} \label{e:V2 to k upstairs}
\on{Tr}_\omega(f):f_{!}(\omega_{{\mathbf{Maps}}(X,Y)^{\on{rat}}_{\Ran X}})\to \omega_{\Ran X}.
\end{equation}

\medskip

Applying $\Gamma_{\dr,c}(\Ran X,-)$, we obtain a map
\begin{equation} \label{e:affine to pt}
\on{Tr}_{H_\bullet}(f):H_\bullet({\mathbf{Maps}}(X,Y)^{\on{rat}}_{\Ran X})\to H_\bullet(\Ran X).
\end{equation}

\medskip

We claim that the map \eqref{e:V2 to k upstairs} is an isomorphism. 
The map \eqref{e:V2 to k upstairs} is given by a compatible system of maps
\begin{equation} \label{e:V2 to k upstairs before colimit}
f(I)_!\left(\omega_{{\mathbf{Maps}}(X,Y)^{\on{rat}}_{X^I}}\right)\to \omega_{X^I}.
\end{equation}

We will show that for every $I$, the map \eqref{e:V2 to k upstairs before colimit} is an isomorphism.
This follows from the following general lemma:

\begin{lem}
Let $T$ be a scheme and $\CE$ a locally projective $\CO_T$-module, viewed as an indscheme
over $T$. Then the map $f_{!}(\omega_\CE)\to \omega_T$ in $\fD(T)$ is an isomorphism, where 
$f$ denotes the map $\CE\to T$.
\end{lem}

\begin{proof}
The question is Zariski local, so we can assume that $\CE$ is projective, and hence a direct
summand of a free $\CO_T$-module $\CE_0$. I.e., $\CE$, viewed as an indscheme over $T$
is a retract of $\CE_0$. Therefore $f_{!}(\omega_\CE)$ is a retract of the corresponding object for
$\CE_0$. However, since $\omega_T$ is also a retract of $f_{!}(\omega_\CE)$ via the zero-section,
we obtain that it suffices to assume that $\CE$ is free. 

\medskip

In the latter case we are dealing with the product situation, so we can assume that $T=\on{pt}$,
and $\CE$ corresponds to a vector space $E\simeq \underset{\alpha}{colim}\, E_\alpha$,
where $E_\alpha$ are finite-dimensional subspaces of $E$. In this case
$$H_\bullet(E)\simeq \underset{\alpha}{colim}\, H_\bullet(E_\alpha),$$
by \eqref{e:calc dR c}.

\medskip

However, for each $\alpha$, the canonical map
$$H_\bullet(E_\alpha)\to k$$
is an isomorphism.
\end{proof}

\sssec{}

Thus, we obtain that the map \eqref{e:V2 to k upstairs} is an isomorphism. Hence, 
\eqref{e:affine to pt} is an isomorphism as well. This implies the assertion of 
\thmref{t:contractibility} for $Y=\BA^n$  in view of \thmref{t:contractibility of Ran}.

\ssec{Subsequent strategy}

The idea of the proof is to deduce the assertion for $Y$ from that for the affine space.
We will do so using an intermediate object,
$${\mathbf{Maps}}(X,U\overset{\on{gen.}}\subset Y)^{\on{rat}}_{\Ran X},$$
introduced below, where $U\subset Y$ is a dense open subset. 

\sssec{}   \label{sss:intr gen cond}

For $I\in \fset$ we let
$$\jmath(I):{\mathbf{Maps}}(X,U\overset{\on{gen.}}\subset Y)^{\on{rat}}_{X^I}\hookrightarrow
{\mathbf{Maps}}(X,Y)^{\on{rat}}_{X^I}$$
to be the open subfunctor that assigns to $x^I:S\to X^I$ the subset of those
maps $$m:(S\times X-\{x^I\})\to Y$$ for which
for every geometric point $s\in S$, the resulting map $(X-\{x^I_s\})\to Y$
lands generically in $U\subset Y$.

\medskip

Set
$${\mathbf{Maps}}(X,U\overset{\on{gen.}}\subset Y)^{\on{rat}}_{\Ran X}:=\underset{(\fset)^{op}}{colim}\, 
{\mathbf{Maps}}(X,U\overset{\on{gen.}}\subset Y)^{\on{rat}}_{X^{\fset}}\in \on{PreStk}.$$

\medskip

By construction, ${\mathbf{Maps}}(X,U\overset{\on{gen.}}\subset Y)^{\on{rat}}_{\Ran X}$ is a pseudo-indscheme.

\sssec{}

The open embeddings $\jmath(I)$ induce a natural
transformation of functors $(\fset)^{op}\rightrightarrows \indSch$
$${\mathbf{Maps}}(X,U\overset{\on{gen.}}\subset Y)^{\on{rat}}_{X^{\fset}}\Rightarrow
{\mathbf{Maps}}(X,Y)^{\on{rat}}_{X^{\fset}}.$$

\medskip

Hence, by \secref{sss:expl map}, we obtain a map
$$\jmath:{\mathbf{Maps}}(X,U\overset{\on{gen.}}\subset Y)^{\on{rat}}_{\Ran X}\to {\mathbf{Maps}}(X,Y)^{\on{rat}}_{\Ran X}.$$

Therefore, by \eqref{e:trace on homology}, we obtain a map 
\begin{equation} \label{e:gen to all}
\on{Tr}_{H_\bullet}(\jmath):H_\bullet\left({\mathbf{Maps}}(X,U\overset{\on{gen.}}\subset Y)^{\on{rat}}_{\Ran X}\right)\to
H_\bullet\left({\mathbf{Maps}}(X,Y)^{\on{rat}}_{\Ran X}\right).
\end{equation}

The following is the only step in the proof of \thmref{t:contractibility} that involves some algebraic geometry:

\begin{prop} \label{p:gen to all}
If $Y\simeq \BA^n$, the map \eqref{e:gen to all} is an isomorphism.
\end{prop}

We do not know whether the assertion of \propref{p:gen to all} hold for more general targets $Y$.

\sssec{}

Let $U$ be affine. The open embedding $U\hookrightarrow Y$ induces also a natural
transformation of functors $$(\fset)^{op}\rightrightarrows \indSch,$$
namely,
$${\mathbf{Maps}}(X,U)^{\on{rat}}_{X^{\fset}}\Rightarrow 
{\mathbf{Maps}}(X,U\overset{\on{gen.}}\subset Y)^{\on{rat}}_{X^{\fset}},$$

\medskip

Hence, by \secref{sss:expl map}, we obtain a map
\begin{equation} \label{e:open to gen spaces}
{\mathbf{Maps}}(X,U)^{\on{rat}}_{\Ran X}\to 
{\mathbf{Maps}}(X,U\overset{\on{gen.}}\subset Y)^{\on{rat}}_{\Ran X}.
\end{equation}

\medskip

Therefore, by \eqref{e:trace on homology}, we obtain a map 
\begin{equation} \label{e:open to gen}
H_\bullet\left({\mathbf{Maps}}(X,U)^{\on{rat}}_{\Ran X}\right)\to 
H_\bullet\left({\mathbf{Maps}}(X,U\overset{\on{gen.}}\subset Y)^{\on{rat}}_{\Ran X}\right).
\end{equation}

\medskip

The next step in the proof will essentially be a formal manipulation with the Ran space:

\begin{prop} \label{p:open to gen}
Assume that $U\subset Y$ is a basic open affine (the locus of non-vanishing of a regular function).
Then the map \eqref{e:open to gen} is an isomorphism.
\end{prop}

\ssec{Conclusion of the proof}

Let us accept the assertions of the above Propositions \ref{p:gen to all} and \ref{p:open to gen}
and conclude the proof of \thmref{t:contractibility}.

\sssec{Step 1}

Let $U$ be a basic affine open subset of $\BA^n$. Combining the assertions of Propositions 
\ref{p:gen to all} and \ref{p:open to gen}, and that of \thmref{t:contractibility} for $\BA^n$, we obtain that the 
trace map 
$$H_\bullet\left({\mathbf{Maps}}(X,U)^{\on{rat}}_{\Ran X}\right)\to k$$
is an isomorphism.

\medskip

I.e., we obtain that the assertion of \thmref{t:contractibility} holds for targets $U$ that are isomorphic
to basic open affine subsets of the affine space. 

\sssec{Step 2} Let $Y$ be as in \thmref{t:contractibility}, and let $U_\alpha\subset Y$ be the 
corresponding open subsets. With no restriction of generality we can assume that each $U_\alpha$
is a basic open affine in $Y$, which, moreover, can be realized as a basic open affine in $\BA^n$. 
Hence, the same will be true for any intersection of the $U_\alpha$'s. 

\medskip

For any finite collection of indices
$\ul\alpha:=\alpha_1,...,\alpha_k$, consider the open subset
$$U_{\ul\alpha}:=\underset{j}\bigcap\, U_{\alpha_j}.$$

Consider the composition
$$H_\bullet\left({\mathbf{Maps}}(X,U_{\ul\alpha})^{\on{rat}}_{\Ran X}\right)\to 
H_\bullet\left({\mathbf{Maps}}(X,U_{\ul\alpha}\overset{\on{gen.}}\subset Y)^{\on{rat}}_{\Ran X}\right)\to k.$$

From Step 1 we obtain that the composed arrow is an isomorphism. From 
\propref{p:open to gen} we obtain that the first arrow is an isomorphism. Hence, we conclude that the trace map
$$H_\bullet\left({\mathbf{Maps}}(X,U_{\ul\alpha}\overset{\on{gen.}}\subset Y)^{\on{rat}}_{\Ran X}\right)\to k$$
is an isomorphism as well. 

\sssec{Step 3} 

For each finite set $I$ we have a Zariski cover 
\begin{equation} \label{e:Cech cover maps}
\underset{\alpha}\bigcup\,\, {\mathbf{Maps}}(X,U_{\ul\alpha}\overset{\on{gen.}}\subset Y)^{\on{rat}}_{X^I}\to
{\mathbf{Maps}}(X,Y)^{\on{rat}}_{X^I}.
\end{equation}

\medskip

Let $U_\bullet$ denote the \v{C}ech nerve of the cover $\underset{\alpha}\bigcup\, U_\alpha\to Y$. 
It is easy to see that the \v{C}ech nerve of the cover \eqref{e:Cech cover maps} identifies with the simplicial indscheme
$${\mathbf{Maps}}(X,U_\bullet\overset{\on{gen.}}\subset Y)^{\on{rat}}_{X^I}.$$

Since the category of D-modules satisfies Zariski descent, we obtain that the canonical map
$$|H_\bullet\left({\mathbf{Maps}}(X,U_\bullet\overset{\on{gen.}}\subset Y)^{\on{rat}}_{X^I}\right)|\to
H_\bullet\left({\mathbf{Maps}}(X,Y)^{\on{rat}}_{X^I}\right)$$
is an isomorphism for each $I$, where $|-|$ denotes the functor of geometric realization
$$\Vect^{{\mathbf \Delta}^{op}}\to \Vect.$$

\medskip

Consider now the corresponding simplicial object 
$${\mathbf{Maps}}(X,U_\bullet\overset{\on{gen.}}\subset Y)^{\on{rat}}_{\Ran X}$$
in $\on{PreStk}$, obtained by taking the colimit over $(\fset)^{op}$. 

\medskip

By taking the colimit over the category ${\mathbf \Delta}^{op}\times (\fset)^{op}$, we obtain that 
the resulting map
$$|H_\bullet\left({\mathbf{Maps}}(X,U_\bullet\overset{\on{gen.}}\subset Y)^{\on{rat}}_{\Ran X}\right)|\to
H_\bullet\left({\mathbf{Maps}}(X,Y)^{\on{rat}}_{\Ran X}\right)$$
is an isomorphism as well.

\medskip

We have a commutative diagram
$$
\CD 
|H_\bullet\left({\mathbf{Maps}}(X,U_\bullet\overset{\on{gen.}}\subset Y)^{\on{rat}}_{\Ran X}\right)|  @>>>  
H_\bullet\left({\mathbf{Maps}}(X,Y)^{\on{rat}}_{\Ran X}\right)  \\
@VVV    @VVV   \\
|k_\bullet|  @>>>   k,
\endCD
$$
where $k_\bullet$ is the constant simplicial object of $\Vect$ with value $k$, and where the vertical arrows
are given by the trace maps.

\medskip

As was shown above, the upper horizontal arrow is an isomorphism. The lower horizontal arrow is an isomorphism
since the category ${\mathbf \Delta}^{op}$ is contractible. The left vertical arrow is an isomorphism by Step 2. Hence,
we conclude that the right vertical arrow is also an isomorphism, as desired. 

\qed(\thmref{t:contractibility})

\ssec{Proof of \propref{p:gen to all}}

The idea of the proof is to show that 
the complement of ${\mathbf{Maps}}(X,U\overset{\on{gen.}}\subset \BA^n)^{\on{rat}}_{\Ran X}$ inside
${\mathbf{Maps}}(X,\BA^n)^{\on{rat}}_{\Ran X}$ is ``of infinite codimension", thereby implying
that the two spaces have the same homology. 

\sssec{}

We need to show that the map \eqref{e:gen to all} is an isomorphism. In fact, we will show that
it is a term-wise isomorphism, i.e., that for every $I\in \fset$, the corresponding map
$$\on{Tr}_{H_\bullet}(\jmath(I)):H_\bullet\left({\mathbf{Maps}}(X,U\overset{\on{gen.}}\subset \BA^n)^{\on{rat}}_{X^I}\right)\to
H_\bullet\left({\mathbf{Maps}}(X,\BA^n)^{\on{rat}}_{X^I}\right)$$
is an isomorphism.

\medskip

In fact, we claim that for every $I\in \fset$, the map
\begin{equation} \label{e:gen to Y upstairs}
(f(I)\circ \jmath(I))_{!}(\omega_{{\mathbf{Maps}}(X,U\overset{\on{gen.}}\subset \BA^n)^{\on{rat}}_{X^I}})\to 
f(I)_{!}(\omega_{{\mathbf{Maps}}(X,\BA^n)^{\on{rat}}_{X^I}})
\end{equation}
is an isomorphism in $\fD(X^I)$, where $f(I)$ is the projection ${\mathbf{Maps}}(X,\BA^n)^{\on{rat}}_{X^I}\to X^I$.

\sssec{}

Let us denote by $V$ the complement $\BA^n-U$ (with any scheme structure). Let us denote by $\imath(I)$ the closed 
embedding
$${\mathbf{Maps}}(X,V)^{\on{rat}}_{X^I}\hookrightarrow {\mathbf{Maps}}(X,\BA^n)^{\on{rat}}_{X^I}.$$
We have:
$${\mathbf{Maps}}(X,\BA^n)^{\on{rat}}_{X^I}-{\mathbf{Maps}}(X,V)^{\on{rat}}_{X^I}=
{\mathbf{Maps}}(X,U\overset{\on{gen.}}\subset \BA^n)^{\on{rat}}_{X^I},$$
as open sub-prestacks of ${\mathbf{Maps}}(X,\BA^n)^{\on{rat}}_{X^I}$.

\medskip

Our current goal is to represent the above indschemes explicitly as union of schemes, whose
dimensions we can control. 

\sssec{}  \label{sss:parameter d}

Since $\Gamma^I$ is a Cartier divisor in $X^I\times X$, for an integer $d$, we can consider a closed
subscheme
$${\mathbf{Maps}}(X,\BA^1)^{\on{rat},d}_{X^I}\subset {\mathbf{Maps}}(X,\BA^1)^{\on{rat}}_{X^I}={\mathbf{Maps}}_{X^I}((X^I\times X)-\Gamma^I,\BA^1)$$
that consists of maps, whose order of pole along $\Gamma^I$ is of order $\leq d$. 

\medskip

By Riemann-Roch, for 
$d>(2g-2)$, where $g$ is the genus of $X$, ${\mathbf{Maps}}(X,\BA^1)^{\on{rat},d}_{X^I}$ is
a vector bundle over $X^I$ of rank $d+1-g$. 

\medskip

For an integer $n$, let 
\begin{multline*}
{\mathbf{Maps}}(X,\BA^n)^{\on{rat},d}_{X^I}:=
\underset{n}{\underbrace{\left({\mathbf{Maps}}(X,\BA^1)^{\on{rat},d}_{X^I}\right)\underset{X^I}\times...\underset{X^I}\times
\left({\mathbf{Maps}}(X,\BA^1)^{\on{rat},d}_{X^I}\right)}}
\subset\\
\subset\underset{n}{\underbrace{\left({\mathbf{Maps}}(X,\BA^1)^{\on{rat}}_{X^I}\right)\underset{X^I}\times...\underset{X^I}\times
\left({\mathbf{Maps}}(X,\BA^1)^{\on{rat}}_{X^I}\right)}}={\mathbf{Maps}}(X,\BA^n)^{\on{rat}}_{X^I}
\end{multline*}
be the corresponding closed subscheme of ${\mathbf{Maps}}(X,\BA^n)^{\on{rat}}_{X^I}$. 

\medskip

Let $f(I)^d$ denote the restriction of the map
$f(I)$ to ${\mathbf{Maps}}(X,\BA^n)^{\on{rat},d}_{X^I}$.

\sssec{}

Set also 
$${\mathbf{Maps}}(X,U\overset{\on{gen.}}\subset \BA^n)^{\on{rat},d}_{X^I}:=
{\mathbf{Maps}}(X,U\overset{\on{gen.}}\subset \BA^n)^{\on{rat}}_{X^I}\cap {\mathbf{Maps}}(X,\BA^n)^{\on{rat},d}_{X^I} \text{ and }$$
$${\mathbf{Maps}}(X,V)^{\on{rat},d}_{X^I}:=
{\mathbf{Maps}}(X,V)^{\on{rat}}_{X^I}\cap {\mathbf{Maps}}(X,\BA^n)^{\on{rat},d}_{X^I}.$$

\medskip

Denote by $\jmath(I)^d$ and $\imath(I)^d$ the open and closed embeddings
$${\mathbf{Maps}}(X,U\overset{\on{gen.}}\subset \BA^n)^{\on{rat},d}_{X^I}\hookrightarrow {\mathbf{Maps}}(X,\BA^n)^{\on{rat},d}_{X^I} \text{ and }
{\mathbf{Maps}}(X,V)^{\on{rat},d}_{X^I}\hookrightarrow {\mathbf{Maps}}(X,\BA^n)^{\on{rat},d}_{X^I},$$
respectively, which are complementary to each other.

\medskip

Consider the projection
$$f(I)^d\circ \imath(I)^d:{\mathbf{Maps}}(X,V)^{\on{rat},d}_{X^I}\to X^I.$$

\medskip

A key observation is that the codimension of 
$${\mathbf{Maps}}(X,V)^{\on{rat},d}_{X^I}\subset 
{\mathbf{Maps}}(X,\BA^n)^{\on{rat},d}_{X^I}$$ in the fibers of $f(I)$ uniformly tends to $\infty$
as $d\to \infty$. More precisely, we will show:

\begin{lem}  \label{l:codimension estimate}
There exists a constant $C$, independent of $d$, such that the dimension of the fibers of the map $f(I)^d\circ \imath(I)^d$
is $\leq (n-1)\cdot d+C$.
\end{lem}

Let us show how this lemma implies the isomorphism in \eqref{e:gen to Y upstairs}. 

\begin{proof}

Note that the terms in \eqref{e:gen to Y upstairs} identify with
$$\underset{d}{colim}\, (f(I)^d\circ \jmath(I)^d)_{!}(\omega_{{\mathbf{Maps}}(X,U\overset{\on{gen.}}\subset \BA^n)^{\on{rat},d}_{X^I}})$$
and
$$\underset{d}{colim}\, (f(I)^d)_{!}(\omega_{{\mathbf{Maps}}(X,\BA^n)^{\on{rat},d}_{X^I}}),$$
respectively.

\medskip

Hence, the cone of the map in \eqref{e:gen to Y upstairs} is given by
$$\underset{d}{colim}\, (f(I)^d\circ \imath(I)^d)_{!}\circ (\imath(I)^d)^*(\omega_{{\mathbf{Maps}}(X,\BA^n)^{\on{rat},d}_{X^I}}).$$

\medskip

We claim that for every $l$ there exists $d_0$ large enough so that for all $d\geq d_0$, the object
\begin{equation} \label{e:truncated on complement}
(f(I)^d\circ \imath(I)^d)_{!}\circ (\imath(I)^d)^*(\omega_{{\mathbf{Maps}}(X,\BA^n)^{\on{rat},d}_{X^I}})\in \fD(X^I)
\end{equation}
lives in the cohomological degrees $\leq -l$. 

\medskip

Note that we are dealing with the subcategory of D-modules with holonomic cohomologies on a finite-dimensional scheme.
To perform the required estimate we shall work in the usual (i.e., non-perverse) t-structure, which makes sense on 
the holonomic subcategory.\footnote{I.e., this is the t-structure that for $k=\BC$ corresponds under Riemann-Hilbert 
to the usual t-structure on the derived category of sheaves with constructible cohomology.}

\medskip

By \lemref{l:codimension estimate}, the fibers of the map $(f(I)^d\circ \imath(I)^d)_{!}$ are of dimension
$\leq (n-1)\cdot d+C$ for some constant $C$ independent of $d$. Hence, it is sufficient to show that
$$(\imath(I)^d)^*(\omega_{{\mathbf{Maps}}(X,\BA^n)^{\on{rat},d}_{X^I}})\in
\fD({\mathbf{Maps}}(X,V)^{\on{rat},d}_{X^I})$$
lives in cohomological degrees $\leq -2\cdot n\cdot d+C'$ for some other constant $C'$ independent of $d$,
with respect to the usual t-structure.

\medskip

For the latter, it is sufficient to show that
$$\omega_{{\mathbf{Maps}}(X,\BA^n)^{\on{rat},d}_{X^I}}\in \fD({\mathbf{Maps}}(X,U)^{\on{rat},d}_{X^I})$$
lives in cohomological degrees $\leq -2\cdot n\cdot d+C'$, again for the usual t-structure.

\medskip

However, the latter is evident: by Riemann-Roch, for $d$ large enough, 
${\mathbf{Maps}}(X,\BA^n)^{\on{rat},d}_{X^I}$ is a vector bundle over $X^I$ of rank
$$n\cdot (d-(1+g)).$$
In particular, it is is smooth, and hence its dualizing complex lives in cohomological degree
$$-2(|I|+n\cdot (d-(1+g))),$$
as required. 

\end{proof}

\sssec{Proof of \lemref{l:codimension estimate}}

By Noether normalization, we can choose a linear map $$\pi:\BA^n\to \BA^{n-1},$$
such that $\pi|_V$ is finite. The map $\pi$ induces a map
$${\mathbf{Maps}}(X,\BA^n)^{\on{rat},d}_{X^I}\to {\mathbf{Maps}}(X,\BA^{n-1})^{\on{rat},d}_{X^I},$$
and hence a map
\begin{equation}  \label{e:map from maps to V}
{\mathbf{Maps}}(X,V)^{\on{rat},d}_{X^I}\to {\mathbf{Maps}}(X,\BA^{n-1})^{\on{rat},d}_{X^I}.
\end{equation}

As the dimensions of the fibers of ${\mathbf{Maps}}(X,\BA^{n-1})^{\on{rat},d}_{X^I}$
over $X^I$ grow as $$(n-1)\cdot d+(n-1)\cdot (1-g),$$ it suffices to show that the map 
\eqref{e:map from maps to V} is finite. 

\medskip

Since all the schemes involved are affine, it suffices
to check that the map in question is proper. We shall show that the map \eqref{e:map from maps to V} satisfies 
the valuative criterion of properness. 

\medskip

Since ${\mathbf{Maps}}(X,V)^{\on{rat},d}_{X^I}$ is closed in ${\mathbf{Maps}}(X,V)^{\on{rat}}_{X^I}$, it is
enough to show that the map 
$${\mathbf{Maps}}(X,V)^{\on{rat}}_{X^I}\to {\mathbf{Maps}}(X,\BA^{n-1})^{\on{rat}}_{X^I}$$
satisfies the valuative criterion. We shall show this for any finite map $Y_1\to Y_2$ of affine
schemes.

\medskip

Thus, let 
$$
\CD
\obC  @>>>  {\mathbf{Maps}}_{X^I}((X^I\times X)-\Gamma^I,Y_1) \\
@VVV   @VVV   \\
\bC  @>>> {\mathbf{Maps}}_{X^I}((X^I\times X)-\Gamma^I,Y_2)
\endCD
$$
be a commutative diagram, where $\bC$ is an affine regular curve, and
$\obC$ is the complement of a point $\bc\in \bC$. We would like to fill in the diagonal
arrow from $\bC$ to ${\mathbf{Maps}}_{X^I}((X^I\times X)-\Gamma^I,Y_1)$.

\medskip

Part of the data of a map
$$\bC\to {\mathbf{Maps}}_{X^I}((X^I\times X)-\Gamma^I,Y_2)$$ is a map
$x^I:\bC\to X^I$. 

\medskip 
 
So, we are given a map
$$m_2:(\bC\times X-\{x^I\})\to Y_2,$$
and its lift to a map
$$\overset{\circ}m_1:(\bC\times X-\{x^I\})\cap (\obC\times X)\to Y_1.$$
We wish to extend $\overset{\circ}m_1$ to a map 
$$m_1:(\bC\times X-\{x^I\})\to Y_1.$$

However, since $Y_1\to Y_2$ is proper, the map $\overset{\circ}m_1$ extends
to an open subscheme of $\bC\times X-\{x^I\}$
whose complement is of codimension $2$. Now, since
$\bC\times X-\{x^I\}$ is normal and $Y_1$ is affine, the above map further extends
to all of $\bC\times X-\{x^I\}$.

\qed

\ssec{Proof of \propref{p:open to gen}}

In order to prove that the map
$${\mathbf{Maps}}(X,U)^{\on{rat}}_{\Ran X}\to {\mathbf{Maps}}(X,U\overset{\on{gen.}}\subset Y)^{\on{rat}}_{\Ran X}$$
induces an isomorphism on homology, we shall introduce another intermediate space, denoted 
$${\mathbf{Maps}}(X,U\subset Y)^{\on{rat}}_{\Ran^\to X},$$
along with the maps
$${\mathbf{Maps}}(X,U)^{\on{rat}}_{\Ran X}\to 
{\mathbf{Maps}}(X,U\subset Y)^{\on{rat}}_{\Ran^\to X}\to {\mathbf{Maps}}(X,U\overset{\on{gen.}}\subset Y)^{\on{rat}}_{\Ran X},$$
and we will show that both these maps induce an isomorphism on homology.

\medskip

The idea of ${\mathbf{Maps}}(X,U\subset Y)^{\on{rat}}_{\Ran^\to X}$ vs 
${\mathbf{Maps}}(X,U\overset{\on{gen.}}\subset Y)^{\on{rat}}_{\Ran X}$ is that 
instead of asking for
a map $(X-\{x^I\})\to Y$  to generically land in $U$, we will specify
the locus outside of which it is defined as a regular map to $U$.



\sssec{}

Consider the category $\fset^{\to}$ introduced in \secref{sss:categ of pairs}. I.e., its objects are
pairs $(J\to I)$ of finite sets with an arbitrary map between them, and morphisms are commutative 
diagrams with surjective arrows.

\medskip

Recall also the functor 
$$X^{\fset^{\to}}:(\fset^{\to})^{op}\to \Sch$$
given by 
$$(J\to I)\rightsquigarrow X^{J\to I}:=X^I,$$
and the space
$$\Ran^\to X:=\underset{(\fset^\to)}{colim}\, X^{\fset^{\to}}.$$

\sssec{}

We shall now introduce a functor 
$${\mathbf{Maps}}(X,U\subset Y)^{\on{rat}}_{X^{\fset^{\to}}}:(\fset^{\to})^{op}\to \indSch,$$
satisfying the assumptions of \secref{sss:over pairs}.

\medskip

For a test scheme $S$ mapping to $X^{J\to I}$ we have two incidence divisors in $S\times X$:
one is $\{x^I\}$ corresponding to $\Gamma^I\subset X^I\times X$, and the other is
$\{x^J\}$ corresponding to $\Gamma^J\subset X^J\times X$ and the composed map
$$S\to X^{J\to I}=X^I\to X^J.$$

\medskip

For $(J\to I)\in \fset^{\to}$ we let ${\mathbf{Maps}}(X,U\subset Y)^{\on{rat}}_{X^{J\to I}}$
be the indscheme over $X^{J\to I}\simeq X^I$ equal to the following subfunctor of
$X^{J\to I}\underset{X^J}\times {\mathbf{Maps}}(X,Y)^{\on{rat}}_{X^J}$:

\medskip

We define 
$$\Hom_{X^I}\left(S,{\mathbf{Maps}}(X,U\subset Y)^{\on{rat}}_{X^{J\to I}}\right)$$
to consist of those maps
$$m:(S\times X-\{x^J\})\to Y$$
for which the restriction of $m$ to the open subset 
$$(S\times X-\{x^I\})\subset (S\times X-\{x^J\})$$
factors through $U$.

\medskip

Diagrammatically, we can write
\begin{equation} \label{e:specified maps}
{\mathbf{Maps}}(X,U\subset Y)^{\on{rat}}_{X^{J\to I}}=
\left(X^I\underset{X^J}\times {\mathbf{Maps}}(X,Y)^{\on{rat}}_{X^J}\right)
\underset{{\mathbf{Maps}}(X,Y)^{\on{rat}}_{X^I}}\times {\mathbf{Maps}}(X,U)^{\on{rat}}_{X^I},
\end{equation}
where 
$$X^I\underset{X^J}\times {\mathbf{Maps}}(X,Y)^{\on{rat}}_{X^J}\to {\mathbf{Maps}}(X,Y)^{\on{rat}}_{X^I}$$
is the natural map, obtained by restricting a map $m:S\times X-\{x^J\}\to Y$ to
a map $$m':S\times X-\{x^J\}\to Y.$$

\sssec{}

Set 
$${\mathbf{Maps}}(X,U\subset Y)^{\on{rat}}_{\Ran^{\to} X}:=
\underset{(\fset^{\to})^{op}}{colim}\, {\mathbf{Maps}}(X,U\subset Y)^{\on{rat}}_{X^{\fset^{\to}}}.$$

\medskip

By construction, ${\mathbf{Maps}}(X,U\subset Y)^{\on{rat}}_{\Ran^{\to} X}$ is a pseudo-indscheme that
comes equipped with a map
$$f^{\to}:{\mathbf{Maps}}(X,U\subset Y)^{\on{rat}}_{\Ran^{\to} X}\to \Ran^{\to} X.$$





\begin{rem}
For $S\in \affSch$, the $\infty$-groupoid $\on{Maps}\left(S,{\mathbf{Maps}}(X,U\subset Y)^{\on{rat}}_{\Ran^{\to} X}\right)$
is in fact a set described as follows. Its elements are pairs of non-empty finite subsets $\ol{x}'\subset \ol{x}$ of
$\on{Maps}(S,X)$, plus a rational map $S\times X\to Y$, which is regular on the complement to the graph of $\ol{x}'$,
and is regular as a map to $U$ on the complement to the graph of $\ol{x}$.
\end{rem}

\sssec{}

Recall (see \secref{sss:categ of pairs}) that the are two natural forgetful functors
$${\mathbf{pr}}_{\on{source}},{\mathbf{pr}}_{\on{target}}:
\fset^{\to}\to \fset,$$
that send $J\to I$ to $I$ and $J$, respectively.

\medskip

Note also that we have a natural transformation between the two
functors $$(\fset^{\to})^{op}\rightrightarrows \indSch:$$
$${\mathbf{Maps}}(X,U\subset Y)^{\on{rat}}_{X^{\fset^{\to}}}\Rightarrow 
{\mathbf{Maps}}(X,U\overset{\on{gen.}}\subset Y)^{\on{rat}}_{X^{\fset}}\circ ({\mathbf{pr}}_{\on{source}})^{op}.$$
I.e., for every $(J\to I)$ we have a naturally defined map
\begin{equation} \label{e:specified to gen on spaces before colimit}
f_{\on{source}}(J\to I):{\mathbf{Maps}}(X,U\subset Y)^{\on{rat}}_{X^{J\to I}}\to
{\mathbf{Maps}}(X,U\overset{\on{gen.}}\subset Y)^{\on{rat}}_{X^J}.
\end{equation}

\medskip

Thus, by \secref{sss:expl map}, we obtain a map
$$f_{\on{source}}:{\mathbf{Maps}}(X,U\subset Y)^{\on{rat}}_{\Ran^{\to} X}\to 
{\mathbf{Maps}}(X,U\overset{\on{gen.}}\subset Y)^{\on{rat}}_{\Ran X},$$
and by \eqref{e:trace on homology} a map
\begin{equation} \label{e:specified to gen}
\on{Tr}_{H_\bullet}(f_{\on{source}}):H_\bullet\left({\mathbf{Maps}}(X,U\subset Y)^{\on{rat}}_{\Ran^{\to} X}\right)\to
H_\bullet\left({\mathbf{Maps}}(X,U\overset{\on{gen.}}\subset Y)^{\on{rat}}_{\Ran X}\right),
\end{equation}

We will prove:

\begin{lem}  \label{l:specified to gen}
The map \eqref{e:specified to gen} is an isomorphism. 
\end{lem}

\sssec{}

As was remarked above, the functor
$${\mathbf{Maps}}(X,U\subset Y)^{\on{rat}}_{X^{\fset^{\to}}}:(\fset^\to)^{op}\to \indSch$$
falls into the paradigm of \secref{sss:over pairs}.

\medskip

Note also that the resulting functor
$${\mathbf{Maps}}(X,U\subset Y)^{\on{rat}}_{X^{\fset^{\to}}}\circ ({\mathbf{diag}})^{op}:
(\fset)^{op}\to \indSch$$
identifies with ${\mathbf{Maps}}(X,U)^{\on{rat}}_{X^{\fset}}$.

\medskip

In particular, \eqref{e:from diag} defines a map
$${\mathbf{Maps}}(X,U)^{\on{rat}}_{\Ran X}\to {\mathbf{Maps}}(X,U\subset Y)^{\on{rat}}_{\Ran^{\to} X}.$$
Note that the composed map
$${\mathbf{Maps}}(X,U)^{\on{rat}}_{\Ran X}\to {\mathbf{Maps}}(X,U\subset Y)^{\on{rat}}_{\Ran^{\to} X}\to 
{\mathbf{Maps}}(X,U\overset{\on{gen.}}\subset Y)^{\on{rat}}_{\Ran X}$$
is the map 
$${\mathbf{Maps}}(X,U)^{\on{rat}}_{\Ran X}\to {\mathbf{Maps}}(X,U\overset{\on{gen.}}\subset Y)^{\on{rat}}_{\Ran X},$$
of \eqref{e:open to gen spaces}.

\sssec{}

Now, it follows from \corref{c:pairs homology}, that the induced map 
\begin{equation} \label{e:last step homology}
H_\bullet\left({\mathbf{Maps}}(X,U)^{\on{rat}}_{\Ran X}\right)\to
H_\bullet\left({\mathbf{Maps}}(X,U\subset Y)^{\on{rat}}_{\Ran^{\to} X}\right)
\end{equation}
is an isomorphism. Combined with \lemref{l:specified to gen}, this implies that the composed
map
$$H_\bullet\left({\mathbf{Maps}}(X,U)^{\on{rat}}_{\Ran X}\right)\to
H_\bullet\left({\mathbf{Maps}}(X,U\subset Y)^{\on{rat}}_{\Ran^{\to} X}\right)\to
H_\bullet\left({\mathbf{Maps}}(X,U\overset{\on{gen.}}\subset Y)^{\on{rat}}_{\Ran X}\right)$$
is an isomorphism, as required.

\qed(\propref{p:open to gen}).

\sssec{An alternative argument}
For the reader who chose to skip \secref{s:unital}, an alternative 
(but essentially equivalent) way to deduce the fact that 
the map \eqref{e:last step homology} is an isomorphism is to apply \cite{BD1}, Proposition 4.4.9
to the assignment
$$(J\to I)\mapsto (f^\to(J\to I))_!\left(\omega_{{\mathbf{Maps}}(X,U\subset Y)^{\on{rat}}_{J\to I}}\right)\in \fD(X^I).$$

\ssec{Proof of \lemref{l:specified to gen}}

\sssec{}

Note that the map \eqref{e:specified to gen} can be interpreted as follows. Namely, it is obtained
by applying the paradigm of \secref{sss:paradigm} to
$$\bC_1:=(\fset^{\to})^{op},\,\, \bC_2:=(\fset)^{op},\,\, F:=({\mathbf{pr}}_{\on{source}})^{op},\,\,\bD=\Vect$$
and
$$\Phi_1(J\to I):=H_\bullet\left({\mathbf{Maps}}(X,U\subset Y)^{\on{rat}}_{X^{J\to I}}\right),\,\,
\Phi_2(J):=H_\bullet\left({\mathbf{Maps}}(X,U\overset{\on{gen.}}\subset Y)^{\on{rat}}_{X^J}\right),$$
and the natural transformation whose value on $(J\to I)\in (\fset^{\to})$ is
\begin{equation} \label{e:V4 to V3 before colimit}
\on{Tr}_{H_\bullet}(f_{\on{source}}(J\to I)):H_\bullet\left({\mathbf{Maps}}(X,U\subset Y)^{\on{rat}}_{X^{J\to I}}\right)\to
H_\bullet\left({\mathbf{Maps}}(X,U\overset{\on{gen.}}\subset Y)^{\on{rat}}_{X^J}\right).
\end{equation}

\sssec{}

The proof that \eqref{e:specified to gen} is an isomorphism is based on the following observation.

\medskip

Suppose that in the situation of
\secref{sss:paradigm} the functor $F$ is a co-Cartesian fibration.
Then in order to check that \eqref{e:map of colimits} is an isomorphism,
it suffices to show that for every $\bc_2\in \bC_2$, the map
\begin{equation} \label{e:objectwise colimit}
\underset{\bC_1\underset{\bC_2}\times \bc_2}{colim}\, \Phi_1|_{\bC_1\underset{\bC_2}\times \bc_2}\to \Phi_2(\bc_2)
\end{equation}
is an isomorphism. 

\sssec{}

We note that the functor $({\mathbf{pr}}_{\on{source}})^{op}:(\fset^{\to})^{op}\to (\fset)^{op}$ is a co-Cartesian fibration, i.e., 
the functor ${\mathbf{pr}}_{\on{source}}:\fset^{\to}\to \fset$ is  Cartesian fibration. Indeed for a map $J_1\twoheadrightarrow J_2$ in 
$\fset$ and an object 
$$(J_2\to I_2)\in \fset^{\to}\underset{\fset}\times J_2,$$
its pullback to $\fset^{\to}\underset{\fset}\times J_1$ is given by composition $J_1\to J_2\to I_2$.

\medskip

Thus, in order to show that \eqref{e:specified to gen} is an isomorphism, it suffices to check that the map
\eqref{e:objectwise colimit} is an isomorphism in our situation, i.e., that for every finite set $J$, the map
$$\underset{I\in (\fset_{J/})^{op}}{colim}\, H_\bullet\left({\mathbf{Maps}}(X,U\subset Y)^{\on{rat}}_{X^{J\to I}}\right)\to
H_\bullet\left({\mathbf{Maps}}(X,U\overset{\on{gen.}}\subset Y)^{\on{rat}}_{X^J}\right)$$
is an isomorphism.

\medskip

In fact, we will show that the isomorphism is taking place ``upstairs", i.e., at the level of objects
of $\fD\left({\mathbf{Maps}}(X,U\overset{\on{gen.}}\subset Y)^{\on{rat}}_{X^J}\right)$. Namely, we will show that the 
trace map
\begin{equation} \label{e:map of colimits upstairs}
\underset{I\in (\fset_{J/})^{op}}{colim}\,
f_{\on{source}}(J\to I)_{!}\left(\omega_{{\mathbf{Maps}}(X,U\subset Y)^{\on{rat}}_{X^{J\to I}}}\right)\to
\omega_{{\mathbf{Maps}}(X,U\overset{\on{gen.}}\subset Y)^{\on{rat}}_{X^J}}
\end{equation}
is an isomorphism.

\sssec{}  \label{sss:fiber is proper}

First, we claim that the map $f_{\on{source}}(J\to I)$ is ind-proper; in fact for every $J\to I$ 
the map 
$${\mathbf{Maps}}(X,U\subset Y)^{\on{rat}}_{X^{J\to I}}\hookrightarrow 
X^I\underset{X^J}\times {\mathbf{Maps}}(X,U\overset{\on{gen.}}\subset Y)^{\on{rat}}_{X^J}$$
is an ind-closed embedding.

\medskip

Indeed, it follows from \eqref{e:specified maps} that 
$${\mathbf{Maps}}(X,U\subset Y)^{\on{rat}}_{X^{J\to I}}\simeq
\left(X^I\underset{X^J}\times {\mathbf{Maps}}(X,U\overset{\on{gen.}}\subset  Y)^{\on{rat}}_{X^J}\right)
\underset{{\mathbf{Maps}}(X,U\overset{\on{gen.}}\subset Y)^{\on{rat}}_{X^I}}\times {\mathbf{Maps}}(X,U)^{\on{rat}}_{X^I},$$
so it is sufficient to show that the map
$${\mathbf{Maps}}(X,U)^{\on{rat}}_{X^I}\to {\mathbf{Maps}}(X,U\overset{\on{gen.}}\subset Y)^{\on{rat}}_{X^I}$$
is an ind-closed embedding. Since $U\subset Y$ is, by assumption, the locus of non-vanishing of a regular function, 
it is sufficient to consider the universal case of
$Y=\BA^1$ and $U=\BA^1-\{0\}$. 

\medskip

The latter situation reduces to the following one: let $S$ be a test scheme, and let $D^1$ and $D^2$ be 
two effective Cartier divisors on $S\times X$, both finite and flat over $S$. Consider the functor on $\Sch_{/S}$ 
that sends $g:T\to S$ to the point-set
if $(g\times \on{id})^{-1}(D_1)$ is \emph{set-theoretically} contained in $(g\times \on{id})^{-1}(D_2)$,
and to the empty set, otherwise. Then we claim that this functor is representable by a formal subscheme of $S$:

\medskip

Indeed, the above functor is the colimit over $n\in \BN$ of the functors, where for each $n$ we require 
that $(g\times \on{id})^{-1}(D_1)$ be contained scheme-theoretically in $(g\times \on{id})^{-1}(n\cdot D_2)$.
However, each of the latter functors is representable by a closed subscheme of $S$.

\sssec{}

We stratify the indscheme ${\mathbf{Maps}}(X,U\overset{\on{gen.}}\subset Y)^{\on{rat}}_{X^J}$ according
to the pattern of collision of points in $X$ corresponding to $X^J$
and also according to the pattern of points of $X$, away from which the rational map 
$m:X\to U$ is regular. 

\medskip

We will show that the map \eqref{e:map of colimits upstairs} is an isomorphism after
!-restriction to each stratum. However, in order to unburden the notation, we will do
so only at the level of !-stalks at $k$-points of each stratum (the proof in the general
case is the same). I.e., we will show that the map \eqref{e:map of colimits upstairs}
is an isomorphism at the level of !-stalks at $k$-points of 
${\mathbf{Maps}}(X,U\overset{\on{gen.}}\subset Y)^{\on{rat}}_{X^J}$. 

\medskip

For a fixed point $(x^J,m)\in {\mathbf{Maps}}(X,U\overset{\on{gen.}}\subset Y)^{\on{rat}}_{X^J}$ as above,
let $$\left({\mathbf{Maps}}(X,U\subset Y)^{\on{rat}}_{X^{J\to I}}\right)_{(x^J,m)}$$ denote the fiber
of the map $f_{\on{source}}(J\to I)$ over it. 

\medskip

Note that because $f_{\on{source}}(J\to I)$ is proper,
the base-change formula applies, i.e., the !-stalk of 
$$\left(f_{\on{source}}(J\to I)\right)_{!}\left(\omega_{{\mathbf{Maps}}(X,U\subset Y)^{\on{rat}}_{X^{J\to I}}}\right)$$
at $(x^J,m)$ is isomorphic to the cohomology of the !-restriction of 
$\omega_{{\mathbf{Maps}}(X,U\subset Y)^{\on{rat}}_{X^{J\to I}}}$ to 
$\left({\mathbf{Maps}}(X,U\subset Y)^{\on{rat}}_{X^{J\to I}}\right)_{(x^J,m)}$, the
latter being 
$$\omega_{\left({\mathbf{Maps}}(X,U\subset Y)^{\on{rat}}_{X^{J\to I}}\right)_{(x^J,m)}}.$$
I.e., we have to show that the trace map
\begin{equation} \label{e:map of colimits fibers}
\underset{I\in (\fset_{J/})^{op}}{colim}\,
H_\bullet\left(\left({\mathbf{Maps}}(X,U\subset Y)^{\on{rat}}_{X^{J\to I}}\right)_{(x^J,m)}\right)\to k
\end{equation}
is an isomorphism. We will deduce this from a certain variant of \thmref{t:contractibility of Ran}. 

\sssec{}  \label{sss:descr of fiber}

Let us describe the indscheme 
\begin{equation} \label{e:fiber}
\left({\mathbf{Maps}}(X,U\subset Y)^{\on{rat}}_{X^{J\to I}}\right)_{(x^J,m)}
\end{equation}
explicitly. 

\medskip

Let $\oX\subset X$ be the open subset over which the rational map $m:X\to U$ is regular. Let $\sA$ be the complementary 
finite set of points, and let $x^\sA$ be the corresponding canonical point of $X^\sA$. For a finite set
$I$, let $X^{I,\sA}$ be the following closed subscheme of $X^I$:
$$X^{I,\sA}:=\underset{\psi:\sA\to I}\cup\, \Delta(\psi)^{-1}(x^\sA).$$
I.e., $X^{I,\sA}$ consists of those $I$-tuples of points of $X$, which \emph{contain} $x^\sA$
as a subset. (If $\sA=\emptyset$, we have $X^{I,\sA}=X^I$.) 

\medskip

Let $X^I_{x^J}$ denote the preimage in $X^I$ of the point $x^J\in X^J$. 

\medskip

It follows that the indscheme \eqref{e:fiber}, which according to \secref{sss:fiber is proper},
is ind-closed in $X^I$, equals
$$X^I_{x^J}\underset{X^I}\times \wh{X}{}^{I,\sA},$$
where $\wh{X}{}^{I,\sA}$ denotes the formal completion of $X^{I,\sA}$ in $X^I$.

\sssec{}

Thus, we obtain that the two functors $(\fset/J)^{op}\to \indSch$
$$X^I_{x^J}\underset{X^I}\times X^{I,\sA} \text{ and } \left({\mathbf{Maps}}(X,U\subset Y)^{\on{rat}}_{X^{J\to I}}\right)_{(x^J,m)}$$
become isomorphic after passing to the corresponding reduced indschemes.

\medskip

Hence, we obtain that it is sufficient to show that the trace map
\begin{equation} \label{e:map of colimits fibers bis}
\underset{I\in (\fset_{J/})^{op}}{colim}\,H_\bullet\left(X^I_{x^J}\underset{X^I}\times X^{I,\sA}\right)\to k
\end{equation}
is an isomorphism.

\sssec{}

Thus, we consider the functor $X^{\fset^\sA_J}:(\fset_{J/})^{op}\to \Sch$ given by
$$I\rightsquigarrow X^I_{x^J}\underset{X^I}\times X^{I,\sA}.$$
Set $\Ran X^{\sA}_J:=\underset{(\fset_{J/})^{op}}{colim}\, X^{\fset^\sA_J}$.

\medskip

We need to show that the trace map defines an isomorphism 
$$H_\bullet(\Ran X^{\sA}_J)\to k.$$

However, this follows by repeating the proof of \thmref{t:contractibility of Ran},
see \secref{s:proof contr Ran}.

\sssec{}  \label{sss:including points}

Here is an alternative argument for the last statement. To simplify the notation,
we will consider the space 
$$\Ran X^{\sA}:=\underset{(\fset)^{op}}{colim}\, X^{\fset^\sA},$$
where $X^{\fset^\sA}$ is the functor $(\fset)^{op}\to \Sch$ that sends
$$I\rightsquigarrow X^{I,\sA}.$$

\medskip

Let $i^\sA$ denote the tautological map $\Ran X^{\sA}\to \Ran X$. We shall denote by
$i_\sA$ the map $\Ran X_{\sA}\to \Ran X$ of \eqref{e:forget A}.

\medskip

It is easy to see that for $\CF\in \fD(\Ran X)$ we have:
\begin{equation} \label{e:including points}
(i_\sA)_!\circ (i_\sA)^!(\CF)\simeq (i^\sA)_!\circ (i^\sA)^!(\CF).
\end{equation}

\medskip

In particular, 
$$H_\bullet(\Ran X_{\sA})\simeq H_\bullet(\Ran X^{\sA}).$$

Now, the required statement follows from \corref{c:marked points homology} 
which says that $$H_\bullet(\Ran X_{\sA})\simeq H_\bullet(\Ran X),$$ combined
with \thmref{t:contractibility of Ran}, which says that 
$H_\bullet(\Ran X)\simeq k$.

\section{Applications to D-modules on $\Bun_G$}

\ssec{The Beilinson-Drinfeld Grassmannian}

\sssec{}

Let $G$ be a connected linear algebraic group. Let $\Gr_{X^{\fset}}$ denote the Beilinson-Drinfeld affine Grassmannian of $G$,
viewed as a functor
$$(\fset)^{op}\to \indSch,$$
see e.g. \cite{BD2}, Sect. 5.3.11 or \cite{MV}, Sect. 5.
For a finite set $I$, we shall denote by $\Gr_{X^I}$ the corresponding indscheme over $X^I$.

\medskip

Let $\Gr_{\Ran X}$ denote the object
$\underset{(\fset)^{op}}{colim}\, \Gr_{X^{\fset}}\in \on{PreStk}$.
By construction, $\Gr_{\Ran X}$ is a pseudo-indscheme.

\medskip

We have a tautological natural transformation $\Gr_{X^{\fset}}\Rightarrow X^{\fset}$,
which gives rise to map
$$f_\Gr:\Gr_{\Ran X}\to \Ran X.$$

\medskip

In this section we will study the category $\fD\left(\Gr_{\Ran X}\right)$, i.e., 
$$\fD\left(\Gr_{\Ran X}\right):=\underset{I\in \fset}{lim}\, \fD(\Gr_{X^I}).$$

\begin{remark}
The functor $\Gr_{X^{\fset}}$ is another example of a functor satisfying the assumptions of
\secref{sss:over Ran}. Thus, we could also introduce the space $\Gr_{\Ran X,\on{indep}}$.
By \propref{p:incl of unital}, the forgetful functor
$$\fD\left(\Gr_{\Ran X,\on{indep}}\right)\to \fD\left(\Gr_{\Ran X}\right)$$
is fully faithful.
\end{remark}

\sssec{}

Let $\Bun_G$ denote the moduli stack of $G$-bundles on $X$. Note also that $\Gr_{\Ran X}$ is equipped 
with a forgetful map $\pi:\Gr_{\Ran X}\to \Bun_G$ in $\on{PreStk}$. 


\medskip

Note also that we have a Cartesian square in $\on{PreStk}$:
\begin{equation} \label{e:fibration over Bun}
\CD
{\mathbf{Maps}}(X,G)^{\on{rat}}_{\Ran X}   @>{\imath'_{\one}}>>  \Gr_{\Ran X}  \\
@VVV    @VV{\pi}V  \\
\on{pt}   @>{\imath_{\one}}>> \Bun_G,
\endCD
\end{equation}
where $\one$ refers to the unit point of $\Bun_G$ corresponding to the trivial bundle. 

\medskip

Note that the map $\pi$ is pseudo ind-schematic (see \secref{sss:ind-schematic}, where this notion is introduced). 


\sssec{}

Consider the category $\fD(\Bun_G)$. We remind that for any $\CY\in \on{PreStk}$, the category
$\fD(\CY)$ is by definition 
$$\underset{S\in \affSch_{/\CY}}{lim}\, \fD(S),$$
where the limit is taken $\StinftyCat_{\on{cont}}$.

\medskip

However, when $\CY$ is an Artin stack, we can also replace 
the index category $S\in \affSch_{/\CY}$ by its full subcategory consisting
of those $S$ that map smoothly to $\CY$, and, further, by the non-full subcategory
of the latter, where we allow only smooth maps $S_1\to S_2$. 
Further, writing $\CY$ as a union of its quasi-compact open substacks $\CY_\alpha$,
we have
$$\fD(\CY)\simeq \underset{\alpha}{lim}\, \fD(\CY_\alpha).$$

\sssec{}

We are now ready to formulate the main result of this section. The morphism $\pi$
defines a functor
$$\pi^!:\fD(\Bun_G)\to \fD\left(\Gr_{\Ran X}\right).$$



\medskip

We have:
\begin{thm} \label{t:pullback fully faithful}
The functor $\pi^!:\fD(\Bun_G)\to \fD\left(\Gr_{\Ran X}\right)$
is fully faithful.
\end{thm}

\begin{remark}
In view of Diagram \eqref{e:fibration over Bun}, the statement of
\thmref{t:pullback fully faithful} it is not surprising: once we show that the map $\pi$
behaves like a fibration, which is what the proof of \thmref{t:pullback fully faithful} will
amount to, the assertion would follow from \thmref{t:contractibility}: the property of
fully faithful pullback functor is enjoyed by fibrations with contractible fibers.
\end{remark}

\sssec{}

Note that \thmref{t:pullback fully faithful} can be reformulated as follows:

\begin{cor} \label{c:fully-faith reform}
The partially defined left adjoint $\pi_!$ to $\pi^!$ is defined on the essential
image of $\pi^!$, and for $\CF\in \fD(\Bun_G)$
the adjunction map 
\begin{equation} \label{e:adjunction Gr and Bun}
\pi_{!}\circ \pi^!(\CF)\to \CF 
\end{equation}
is an isomorphism.
\end{cor}

\medskip

Applying \eqref{e:adjunction Gr and Bun} to $\CF:=\omega_{\Bun_G}$, we
obtain:

\begin{cor} \label{c:dualizing on Bun}
The trace map
$$\on{Tr}_\omega(\pi):\pi_{!}(\omega_{\Gr_{\Ran X}})\to \omega_{\Bun_G}$$
is an isomorphism.
\end{cor}

\sssec{A variant}  \label{sss:with level structure}

For a finite subscheme $D\subset X$ we can consider the stack $\Bun_G^{\on{level}_D}$ that classifies
$G$-bundles ``with structure of level $D$", i.e., equipped with a trivialization when restricted to $D$.

\medskip

Similarly, one can consider the functor 
$$\Gr_{X^{\fset}}^{\on{level}_D}:(\fset)^{op}\to \indSch$$ equal by definition to
$$\Bun_G^{\on{level}_D}\underset{\Bun_G}\times \Gr_{X^{\fset}},$$
and the corresponding object
$$\Gr_{\Ran X}^{\on{level}_D}:=\underset{(\fset)^{op}}{colim}\, \Gr_{X^{\fset}}^{\on{level}_D}\in \on{PreStk}.$$

\medskip

Denote by $\pi_D$ the map $\Gr_{\Ran X}^{\on{level}_D}\to \Bun_G^{\on{level}_D}$.

\medskip

The proof of \thmref{t:pullback fully faithful} applies equally to the present situation, i.e.,
the functor
$$\pi^!_D:\fD(\Bun_G^{\on{level}_D})\to \fD(\Gr_{\Ran X}^{\on{level}_D})$$
is fully faithful.

\medskip

The rest of this subsection is devoted to the proof of \thmref{t:pullback fully faithful}.


\sssec{Step 1}

We claim that it suffices to show that for every affine scheme $S$ equipped with a map
$g:S\to \Bun_G$, for the Cartesian diagram
$$
\CD
S\underset{\Bun_G}\times \Gr_{\Ran X}  @>{g'}>>  \Gr_{X^{\fset}} \\
@V{\pi_S}VV   @VV{\pi}V  \\
S  @>{g}>> \Bun_G,
\endCD
$$
and $\CF_S\in \fD(S)$, the trace map
\begin{equation} \label{e:map after base change}
(\pi_S)_{!}\circ (\pi_S)^!(\CF_S)\to \CF_S
\end{equation}
is an isomorphism (and in particular, the left hand side is defined as an object of $\fD(S)$).
Let us assume that \eqref{e:map after base change} is an isomorphism, and deduce 
the assertion of \thmref{t:pullback fully faithful}.

\medskip

First, for $\CF_1,\CF_2\in \fD(\Bun_G)$, the map
$$\Hom_{\fD(\Bun_G)}(\CF_1,\CF_2)\to \underset{(S,g)\in \affSch_{/\Bun_G}}{lim}\, 
\Hom_{\fD(S)}\left(g^!(\CF_1),g^!(\CF_2)\right)$$
is an isomorphism. 

\medskip

However, for any map $\CY_1\to \CY_2$ in $\on{PreStk}$, the restriction map
$$\fD(\CY_1)\to \underset{S\in \affSch_{/\CY_2}}{lim}\, \fD(S\underset{\CY_2}\times \CY_1)$$
is an equivalence.

\medskip

Therefore, for $\CF_1,\CF_2\in \fD(\Bun_G)$, the map
$$\Hom_{\fD(\Gr_{\Ran X})}\left(\pi^!(\CF_1),\pi^!(\CF_2)\right)\to \underset{(S,g)\in \affSch_{/\Bun_G}}{lim}\, 
\Hom_{\fD(S\underset{\Bun_G}\times \Gr_{\Ran X})}\left(g'{}^!\circ \pi^!(\CF_1),g'{}^!\circ \pi^!(\CF_2)\right)$$
is an isomorphism. 

\medskip

Hence, in order to show that 
$$\Hom_{\fD(\Bun_G)}(\CF_1,\CF_2)\to \Hom_{\fD(\Gr_{\Ran X})}\left(\pi^!(\CF_1),\pi^!(\CF_2)\right)$$
is an isomorphism, it suffices to show that for every $S$ as above, the map
\begin{equation} \label{e:fully faithful after base change}
\Hom_{\fD(S)}\left(g^!(\CF_1),g^!(\CF_2)\right)\to 
\Hom_{\fD(S\underset{\Bun_G}\times \Gr_{\Ran X})}\left(g'{}^!\circ \pi^!(\CF_1),g'{}^!\circ \pi^!(\CF_2)\right)
\end{equation}
is an isomorphism. 

\medskip

However, noting that $g'{}^!\circ \pi^!(\CF_i)\simeq \pi_S^!\circ g^!(\CF_i)$ for $i=1,2$ and taking
$\CF_S=g^!(\CF_1)$, we obtain that the isomorphism 
\eqref{e:fully faithful after base change} follows from \eqref{e:map after base change}.

\sssec{Step 2}

We fix a map $S\to \Bun_G$, and we wish to establish \eqref{e:map after base change}.

\medskip

It is easy to see, however, that (possibly after passing to an \'etale cover of $S$), there exists
a finite set $\sA$ and a map $x^{\sA}:S\to X^{\sA}$, such that the pullback of the universal $G$-bundle
under
$$S\times X\to \Bun_G\times X$$
admits a trivialization over $S\times X-\{x^\sA\}$. This follows, e.g., from Theorem 3 of \cite{DS}. 

\medskip

Recall the category $\fset_\sA$ (see \secref{sss:marked points}) and 
the functor $X^{\fset_\sA}:(\fset_\sA)^{op}\to \Sch_{/S}$ (see \secref{sss:relative version}). 

\medskip

Let $\Gr_{X_\sA^{\fset}}$ be the functor $(\fset_\sA)^{op}\to \indSch$ defined by sending $I\in \fset_{\sA/}$
to 
$$\Gr_{X^I_\sA}:=X^I_\sA\underset{(S\times X^I)}\times \left(S\underset{\Bun_G}\times \Gr_{X^I}\right)\simeq
S\underset{X^\sA\times \Bun_G}\times \Gr_{X^I}\simeq S\underset{S\times X^\sA}\times 
(S\underset{\Bun_G}\times \Gr_{X^I}).$$

\medskip

Set
$$\Gr_{\Ran X_\sA}:=\underset{(\fset)^{op}}{colim}\,\Gr_{X^{\fset}_\sA}\in \on{PreStk}.$$
Let $\pi_{S,\sA}:\Gr_{\Ran X_\sA}\to S$ denote the resulting map.

\medskip

We claim that the map
$$\Gr_{\Ran X_\sA}\hookrightarrow S\underset{\Bun_G}\times \Gr_{\Ran X}$$
induces an isomorphism
\begin{equation} \label{sss:Gr marked an umarked}
(\pi_{S,\sA})_{!}\circ (\pi_{S,\sA})^!(\CF_S)\simeq (\pi_S)_{!}\circ (\pi_S)^!(\CF_S).
\end{equation}
Indeed, this follows from \corref{c:marked points homology rel}.

\begin{remark}
For the reader who chose to skip \secref{s:unital}, here is an alternative
way to deduce the isomorphism \eqref{sss:Gr marked an umarked}:
namely, one can repeat the argument of \cite{BD1}, Proposition 4.4.2.
\end{remark}

\sssec{Step 3}

Consider now the functor 
$${\mathbf{Maps}}_S(X,G)^{\on{rat}}_{X^{\fset}_\sA}:(\fset_\sA)^{op}\to \indSch_{/S}$$
defined by 
$$I\rightsquigarrow X^I_\sA\underset{X^I}\times {\mathbf{Maps}}(X,G)^{\on{rat}}_{X^I}\simeq
S\underset{X^\sA}\times {\mathbf{Maps}}(X,G)^{\on{rat}}_{X^I}.$$
Set
$${\mathbf{Maps}}_S(X,G)^{\on{rat}}_{\Ran X_\sA}:=\underset{(\fset_\sA)^{op}}{colim}\,
{\mathbf{Maps}}_S(X,G)^{\on{rat}}_{X^{\fset}_\sA}.$$

\medskip

Note now that a choice of a trivialization of the pulled-back $G$-bundle on the open subscheme
$S\times X-\{x^\sA\}$ defines an isomorphism
$$\Gr_{X^{\fset}_\sA}\simeq {\mathbf{Maps}}_S(X,G)^{\on{rat}}_{X^{\fset}_\sA},$$
and hence 
$$\Gr_{\Ran X_\sA}\simeq {\mathbf{Maps}}_S(X,G)^{\on{rat}}_{\Ran X_\sA}.$$

\medskip

Let $p_\sA$ denote the projection
$${\mathbf{Maps}}_S(X,G)^{\on{rat}}_{\Ran X_\sA}\to S.$$ 
Thus, we have to show that for $\CF_S\in \fD(S)$, the trace map
$$(p_\sA)_{!}\circ (p_\sA)^!(\CF_S)\to \CF_S$$
is an isomorphism. 

\medskip

However, by the same logic as in Step 2, i.e., by \corref{c:marked points homology rel},
we can replace the pair
$$\left({\mathbf{Maps}}_S(X,G)^{\on{rat}}_{\Ran X_\sA},p_\sA\right)$$ by 
$$\left(S\times {\mathbf{Maps}}(X,G)^{\on{rat}}_{\Ran X}, \on{id}_S\times p\right),$$
where $p:{\mathbf{Maps}}(X,G)^{\on{rat}}_{\Ran X}\to \on{pt}$.

\medskip

Thus, we have to show that the trace map
$$(\on{id}_S\times p)_!\circ (\on{id}_S\times p)^!(\CF_S)\to \CF_S$$
is an isomorphism. However, the left-hand side is isomorphic to
$$H_\bullet\left({\mathbf{Maps}}(X,G)^{\on{rat}}_{\Ran X})\right)\otimes \CF_S,$$
and the desired assertion follows
from \thmref{t:contractibility}.

\qed

\ssec{Recollections on Verdier duality and D-modules on stacks}

In order to state the corollaries of \thmref{t:pullback fully faithful} pertaining to
cohomology of D-modules and quasi-coherent sheaves on $\Bun_G$, let
us recall some basics from the theory of D-modules. For a more detailed
treatment, the reader is referred to \cite[Sects. 5 and 6]{DrGa0}.

\sssec{}  \label{sss:Verdier basics}

First, let us review Verdier duality on schemes. For $Z\in \Sch$ consider
the category $\fD(Z)$. (We remind that according to our conventions, 
all schemes are assumed of finite type, and in particular, quasi-compact.) It is known that $\fD(Z)$
is compactly generated; its compact objects are bounded complexes, 
whose cohomologies are finitely generated.

\medskip

Verdier duality is a canonical equivalence 
\begin{equation} \label{e:Verdier on compact}
\BD_{\fD(Z)}:(\fD(Z)^c)^{op}\to \fD(Z)^c. 
\end{equation}
It is characterized by the property that for $\CF\in \fD(Z)^c$ and any $\CF_1\in \fD(Z)$,
\begin{equation} \label{e:characterize Verdier}
\Hom_{\fD(Z)}(\CF_1,\BD_{\fD(Z)}(\CF))=\Hom_{\fD(Z\times Z)}(\CF_1\boxtimes \CF,\Delta_{\dr,*}(\omega_Z)),
\end{equation}
where $\Delta_{\dr,*}$ denotes the D-module direct image functor (in this case for the diagonal morphism).

\medskip

By \cite{DG}, Sect. 2.3, the equivalence \eqref{e:Verdier on compact} defines an 
equivalence
\begin{equation} \label{e:self duality}
\fD(Z)^\vee\simeq \fD(Z),
\end{equation}
where for $\bC\in \StinftyCat$, we denote by $\bC^\vee$ the dual category, see, e.g., \cite{DG}, Sect. 2.1.

\medskip

Let now $g:Z_1\to Z_2$ be a morphism, and consider he functor $g^!:\fD(Z_2)\to \fD(Z_1)$.
The dual functor, $(g^!)^\vee$, which due to \eqref{e:self duality} can be though of
as a functor $\fD(Z_1)\to \fD(Z_2)$, identifies with the functor
$g_{\dr,*}$ of D-module (a.k.a. de Rham) direct image. 

\medskip

Note also that \cite{DG}, Lemma 2.3.3, formally implies the relationship
between $g_{\dr,*}$, which is the \emph{dual} of $g^!$, with $g_!$, which is
the left adjoint of $g^!$, whenever the latter is defined on all of $\fD(Z_1)$:

\medskip

Namely, for $g_!$ to be defined on all of $\fD(Z_1)$, it is necessary and sufficient that $g_{\dr,*}$
send $\fD(Z_1)^c\to \fD(Z_2)^c$. For a given object $\CF\in \fD(Z_1)^c$, the functor $g_!$
is defined on it if and only if $g_{\dr,*}\left(\BD_{\fD(Z_1)}(\CF)\right)$ belongs to 
$\fD(Z_2)^c$, and in the latter case we have:
$$g_!(\CF)\simeq \BD_{\fD(Z_2)}\circ g_{\dr,*} \circ \BD_{\fD(Z_1)}(\CF).$$

\sssec{}

In particular, it is easy to see that if $g$ is not proper, the functor $g_!$ is not defined on all
of $\fD(Z_1)$:

\medskip

Indeed, for a scheme $Z$, let 
$'\ind^{\fD}_Z:\QCoh(Z)\to \fD(Z)$
denote the functor left adjoint to the forgetful functor
$$'\oblv^{\fD}_Z:\fD(Z)\to \QCoh(Z),$$
where we are thinking of $\fD(-)$ in the ``right D-module realization". 
\footnote{We use the notation $'\ind^{\fD}_Z$ and $'\oblv^{\fD}_Z$ instead of
$\ind^{\fD}_Z$ and $\oblv^{\fD}_Z$, respectively, because the latter is reserved
for the corresponding adjoint pair of functors
$\IndCoh(Z)\rightleftarrows \fD(Z)$, where $\IndCoh(Z)$ is the category introduced in 
\cite{IndCoh}. The functor $'\oblv^{\fD}_Z$ is the composition of $\oblv^{\fD}_Z$
and the colocalization functor $\Psi_Z:\IndCoh(Z)\to \QCoh(Z)$. The functor
$\ind^{\fD}_Z$ is isomorphic to the composition of $\Psi_Z$ and 
$'\ind^{\fD}_Z$, see \cite[Sect. 5.1.10]{DrGa0}.}

\medskip

The functor $'\ind^{\fD}_Z$ sends $\Coh(Z)$ to the subcategory $\fD(Z)^c$ of $\fD(Z)$. 
Furthermore, it is easy to see that for $\CF\in \QCoh(Z)$, we have 
$'\ind^{\fD}_Z(\CF)\in \fD(Z)^c$ \emph{if and only if} $\CF\in \Coh(Z)$.

\medskip

For a morphism $g:Z_1\to Z_2$ we have:
\begin{equation} \label{e:dir image and induction}
g_{\dr,*}\circ {}'\ind^{\fD}_{Z_1}\simeq {}'\ind^{\fD}_{Z_2}\circ g_*.
\end{equation}

\medskip

However, it is known that unless $g$ is proper, the functor $g_*:\QCoh(Z_1)\to \QCoh(Z_2)$
does not send $\Coh(Z_1)$ to $\Coh(Z_2)$, hence
$\ind^{\fD}_{Z_2}\circ g_*(\Coh(Z_1))\notin \fD(Z_2)^c$.

\sssec{}

The above discussion is also applicable when a scheme $Z$ is replaced by a quasi-compact
Artin stack $\CY$, whose inertia stack, i.e., $\CY\underset{\CY\times \CY}\times \CY$,
is affine over $\CY$ (these facts are established in \cite[Sect. 7]{DrGa0}):

\medskip

\noindent (i) We have a pair of adjoint functors
$$'\ind^{\fD}_\CY:\QCoh(\CY)\rightleftarrows \fD(\CY):{}'\oblv^{\fD}_{\CY},$$
with $'\oblv^{\fD}_{\CY}$ being conservative.

\medskip

\noindent (ii) The category $\fD(\CY)$ is compactly generated. In fact, a set of compact generators
is obtained by applying the functor $'\ind^{\fD}_\CY$ to $\Coh(\CY)\subset \QCoh(\CY)$.

\medskip

\noindent (iii) Verdier duality, defined by formula \eqref{e:characterize Verdier}, 
is an equivalence $$\BD_\CY:(\fD(\CY)^c)^{op}\to \fD(\CY)^c,$$ and hence defines an identification
$\fD(\CY)^\vee\simeq \fD(\CY)$.

\medskip

The substantial difference from the case of schemes is the following: for a non-necessarily
schematic morphism morphism $g:\CY_1\to \CY_2$, where $\CY_1$ and $\CY_2$ are as
above, the functor dual to $g^!$ is no longer the naively defined functor $g_{\dr,*}$, but rather 
its renormalized version that we denote by $g_{\rendr,*}$, which is the 
ind-extension of the functor $g_{\dr,*}|_{\fD(\CY)^c}$. We refer the reader to \cite[Sect. 8.3]{DrGa0},
where the basic properties of the functor $g_{\rendr,*}$ are established. 

\medskip

In particular,
\begin{equation} \label{e:ren dR}
g_{\rendr,*}|_{\fD(\CY)^c}\simeq g_{\dr,*}|_{\fD(\CY)^c}.
\end{equation}

\medskip

The original functor
$$g_{\dr,*}:\fD(\CY_1)\to \fD(\CY_2)$$
may fail to be continuous, as can be seen in the example of $\CY_1=\on{pt}/B\BG_m$
and $\CY_2=\on{pt}$. (Whenever $g_{\dr,*}$ is continuous, it is canonically isomorphic
to $g_{\rendr,*}$.)

\medskip

The relationship between $g_!$, whenever the latter is defined, and the functor
$g_{\rendr,*}$ introduced
above, is the same as in the case of schemes: for $\CF\in \fD(\CY_1)^c$, the functor $g_!$ is defined on it if
and only if $g_{\rendr,*}(\BD_{\fD(\CY_1)}(\CF))\in \fD(\CY_2)^c$, and in the latter case we have:
\begin{equation} \label{e:two pushforwards}
g_!(\CF)\simeq \BD_{\fD(\CY_2)}\circ g_{\rendr,*} \circ \BD_{\fD(\CY_1)}(\CF).
\end{equation}

\medskip

Note, however, that formula \eqref{e:dir image and induction} still holds for the naive direct
image $g_{\dr,*}$, i.e.,
\begin{equation} \label{e:dir image and induction stacks}
g_{\dr,*}\circ {}'\ind^{\fD}_{\CY_1}\simeq {}'\ind^{\fD}_{\CY_2}\circ g_*.
\end{equation}
In fact, one can show that for $\CF\in \QCoh(\CY_1)$, the canonical map
$$g_{\rendr,*}\circ {}'\ind^{\fD}_{\CY_1}(\CF)\to g_{\dr,*}\circ {}'\ind^{\fD}_{\CY_1}(\CF)$$
is an isomorphism, see \cite[Corollary 8.3.9 and Example 8.2.4]{DrGa0}.

\medskip

We shall mostly apply the above discussion to $\CY_1=\CY$, $\CY_2=\on{pt}$ and
$g=p_\CY$. We shall use the notation
$$\Gamma_{\dr}(\CY,-):=(p_\CY)_{\dr,*}(-) \text{ and } \Gamma_{\rendr}(\CY,-):=(p_\CY)_{\rendr,*}(-).$$

\ssec{(Co)homology of D-modules on $\Bun_G$}

For the rest of the paper we will assume that $G$ is reductive. The main feature of this
situation is that in this case the indschemes $\Gr_{X^I}$ are ind-proper. In particular,
$\Gr_{\Ran X}$ is pseudo ind-proper (see \secref{sss:ind-schematic}, where the latter notion
is introduced). 

\medskip

By \secref{sss:when good}, from the peusdo ind-properness of $\Gr_{\Ran X}$, the functor
$$\Gamma_{\dr,c}(\Gr_{\Ran X},-):\fD(\Gr_{\Ran X})\to \Vect,$$
left adjoint to $p_{\Gr_{\Ran X}}^!$, is defined on all of $\fD(\Gr_{\Ran X})$.

\sssec{}

We claim: 

\begin{cor}  \label{c:homology of D-modules}
The functor
$$\Gamma_{\dr,c}(\Bun_G,-):\fD(\Bun_G)\to \Vect,$$
left adjoint to $p_{\Bun_G}^!:\Vect\to \fD(\Bun_G)$, is defined on all of $\fD(\Bun_G)$. Moreover,
we have a canonical isomorphism
$$\Gamma_{\dr,c}(\Bun_G,-)\simeq \Gamma_{\dr,c}\left(\Gr_{\Ran X},\pi^!(-)\right).$$
\end{cor}

In light of the discussion in \secref{sss:Verdier basics}, the 
meaning of this corollary is that, with respect to the functor $p^!_{\Bun_G}$, the stack 
$\Bun_G$ exhibits features of a proper scheme. 

\medskip

\noindent{\it Warning:} However, it is not true that the functor $\Gamma_{\dr,c}(\Bun_G,-)$
is isomorphic to either $(p_{\Bun_G})_{\dr,*}$ or its renormalized version.

\begin{proof}

We wish to prove that for $\CF\in \fD(\Bun_G)$ and $V\in \Vect$
there exists a functorial isomorphism
$$\Hom_{\fD(\Bun_G)}(\CF,p^!_{\Bun_G}(V))\simeq \Hom_{\Vect}
\left(\Gamma_{\dr,c}\left(\Gr_{\Ran X},\pi^!(\CF)\right),V\right).$$

\medskip

By \thmref{t:pullback fully faithful}, the left-hand side maps isomorphically to
\begin{equation} \label{e:on Gr}
\Hom_{\fD\left(\Gr_{\Ran X}\right)}\left(\pi^!(\CF),\pi^!\circ p^!_{\Bun_G}(V)\right)\simeq
\Hom_{\fD\left(\Gr_{\Ran X}\right)}\left(\pi^!(\CF),p^!_{\Gr_{\Ran X}}(V)\right).
\end{equation}

\medskip

However, the right-hand side in \eqref{e:on Gr} is isomorphic to
$$\Hom_{\Vect}\left(\Gamma_{\dr,c}\left(\Gr_{\Ran X},\pi^!(\CF)\right),V\right),$$
by definition. 

\end{proof}

\begin{remark}
Note that \corref{c:homology of D-modules} is specific to $\Bun_G$, i.e., 
it would not work for $\Bun_G^{\on{level}_D}$, because in the latter case the indschemes 
$\Gr^{\on{level}_D}_{X^I}$ are no longer ind-proper.
\end{remark}

\sssec{}  \label{sss:U alpha}

Being a left adjoint to a continuous functor, the functor $\Gamma_{\dr,c}(\Bun_G,-)$ sends compact objects to compact ones,
i.e., to bounded complexes of vector space with finite-dimensional cohomologies. 

\medskip

Let us recall now that the main theorem of \cite{DrGa1} asserts that $\fD(\Bun_G)$ is compactly generated.
In fact in {\it loc.cit.} a stronger assertion is established:

\medskip

It is shown that $\fD(\Bun_G)$ can be presented as a union of quasi-compact open substacks $\CU_\alpha$
that are \emph{co-truncative} (see \cite[Theorem 4.1.12]{DrGa1}). By definition (see \cite[Sect. 4.1]{DrGa1}),
this means  that for the open embedding
$$\CU_\alpha\overset{j_\alpha}\hookrightarrow \Bun_G,$$
the \emph{a priori partially defined functor} $(j_\alpha)_{!}$,
left adjoint to $j_\alpha^!$, is defined on all of $\fD(\CU_\alpha)$.

\medskip

A generating set of compact objects in $\fD(\Bun_G)$ is provided by
$$(j_\alpha)_{!}(\CF_\alpha)$$
for $\CF_\alpha\in \fD(\CU_\alpha)^c$. (The categories $\fD(\CU_\alpha)$ themselves are compactly
generated because $\CU_\alpha$'s are quasi-compact.)

\sssec{}

From \corref{c:homology of D-modules} we obtain:

\begin{cor}  \label{c:homology on U alpha}
For each open substack $\CU_\alpha$ as above, the functor
$$\Gamma_{\dr,c}(\CU_\alpha,-):\fD(\CU_\alpha)\to \Vect,$$
left adjoint to $p_{\CU_\alpha}^!$, is defined on all of $\fD(\CU_\alpha)$.
\end{cor}

Thus, each of the quasi-compact substacks $\CU_\alpha$ also
exhibits features of a proper scheme. 

\sssec{}

Applying \eqref{e:two pushforwards} and \eqref{e:ren dR} to $p_{\CU_\alpha}:\CU_\alpha\to \on{pt}$, 
we obtain:

\begin{cor}  \label{c:cohomology on opens}
For $\CU_\alpha$ as above, the functor
$$\Gamma_{\dr}(\CU_\alpha,-):\fD(\CU_\alpha)\to \Vect$$
sends compact objects to finite-dimensional vector spaces.
\end{cor}

\medskip

Applying \eqref{e:dir image and induction stacks} to $p_{\CU_\alpha}:U_{\alpha}\to \on{pt}$, we
obtain:

\begin{cor}
The functor $\Gamma:\QCoh(\CU_\alpha)\to \Vect$ sends 
$\Coh(\CU_\alpha)$ to $\Vect^c$, i.e., to bounded complexes of vector spaces
with finite-dimensional cohomologies.
\end{cor}

\ssec{The Hecke-equivariant category}

\sssec{}

We are going to consider another functor $(\fset)^{op}\to \on{PreStk}$, denoted $\CH_{X^{\fset}}$,
and referred to as ``the Hecke stack":

\medskip

For a finite set $I$ and a test scheme $S$ we set $\on{Maps}(S,\CH_{X^I})$ to be 
the groupoid of triples $(x^I,\CP',\CP'',\alpha)$, where $x^I:S\to X^I$, 
$\CP',\CP''$ are $G$-bundles on $S\times X$, and $\alpha$ is an isomorphism
$$\CP'_{S\times X-\{x^I\}}\simeq \CP''_{S\times X-\{x^I\}}.$$

\medskip

Let 
$$\CH_{\Ran X}:=\underset{(\fset)^{op}}{colim}\, \CH_{X^{\fset}}\in \on{PreStk}.$$

We have the natural projections
\begin{gather*}
\xy
(-20,0)*+{\Bun_G}="A";
(0,20)*+{\CH_{\Ran X}}="B";
(20,0)*+{\Bun_G,}="C";
(0,-10)*+{\Ran X}="D";
{\ar@{->}_{\hl} "B";"A"};
{\ar@{->}^{\hr} "B";"C"};
{\ar@{->}^{f_{\CH}} "B";"D"};
\endxy
\end{gather*}
where $\hl$ and $\hr$ remember the data of $\CP'$ and $\CP''$, respectively.

\medskip

An important property of the projections $\hl$ and $\hr$ is that for each $I\in \fset$, the
corresponding maps 
$$\hl|_{\CH_{X^I}},\hr|_{\CH_{X^I}}:\CH_{X^I}\to \Bun_G$$
are ind-schematic and ind-proper. Hence, the maps $\hl$ and $\hr$ are pseudo ind-proper. 

\medskip

In particular, the functors $\hl^!,\hr^!:\fD(\Bun_G)\to \fD(\CH_{\Ran X})$
admit left adjoints, denoted $\hl_!$ and $\hr_!$, respectively. 

\sssec{}

Note now that the projections $(\hl,\hr)$ make $\CH_{\Ran X}$ into a groupoid over $\Bun_G$. The
composition map is defined as follows:

\medskip 

We send an $S$-point $(x^{I_1},\CP'_1,\CP''_1,\alpha_1)$ of $\CH_{X^{I_1}}$ as above and
an $S$-point $(x^{I_2},\CP'_2,\CP''_2,\alpha_2)$ of $\CH_{X^{I_2}}$ with $\CP''_1\simeq \CP'_2$,
to the $S$-point of $\CH_{X^{I_1\sqcup I_2}}$ with $(x^{I_1\sqcup I_2},\CP',\CP'',\alpha)$ given by

\medskip

\begin{itemize}

\item $x^{I_1\sqcup I_2}:=(x^{I_1},x^{I_2})$, 

\item $\CP':=\CP'_1$, $\CP'':=\CP''_2$, and

\item $\alpha$ equal to the composition
$$\CP'_1|_{S\times X-\{x^{I_1},x^{I_2}\}}\overset{\alpha_1}\simeq
\CP''_1|_{S\times X-\{x^{I_1},x^{I_2}\}}\simeq \CP'_2|_{S\times X-\{x^{I_1},x^{I_2}\}}
\overset{\alpha_2}\simeq \CP'_2|_{S\times X-\{x^{I_1},x^{I_2}\}}.$$

\end{itemize}

\sssec{}  \label{sss:equiv general}

In general, whenever 
\begin{gather*}
\xy
(-15,0)*+{\CY}="A";
(0,15)*+{\CG}="B";
(15,0)*+{\CY}="C";
{\ar@{->}_{\hl} "B";"A"};
{\ar@{->}^{\hr} "B";"C"};
\endxy
\end{gather*}
is a groupoid in $\on{PreStk}$, we define the $\CG$-equivariant category $\fD(\CY)^\CG$
as the totalization of the corresponding cosimplicial category $\fD(\CG^\bullet)$,
where 
$$\CG^n:=\underset{n+1}{\underbrace{\CG\underset{\CY}\times...\underset{\CY}\times \CG}}.$$

\medskip

In other words,
$$\fD(\CY)^\CG\simeq \fD(|\CG^\bullet|),$$
where $|\CG^\bullet|$ is the geometric realization of $\CG^\bullet$ in $\on{PreStk}$.

\sssec{}

We define the Hecke equivariant category
$$\fD(\Bun_G)^{\on{Hecke}}:=\fD(\Bun_G)^{\CH_{\Ran X}}.$$

By definition, we have a conservative continuous functor
$$\oblv^{\on{Hecke}}:\fD(\Bun_G)^{\on{Hecke}}\to \fD(\Bun_G).$$

\medskip

We will prove:

\begin{prop} \label{p:descr of monad}
The functor $\oblv^{\on{Hecke}}$ admits a left adjoint (denoted $\ind^{\on{Hecke}}$). The resulting
monad $\oblv^{\on{Hecke}}\circ \ind^{\on{Hecke}}$, viewed as an endo-functor of $\fD(\Bun_G)$
is canonically isomorphic to $\hl_{!}\circ \hr^!$.
\end{prop}

We will prove the proposition in the general framework of \secref{sss:equiv general}, when
the morphisms $\hl$ and $\hr$ are pseudo ind-proper (see \secref{sss:ind-schematic},
where this notion is introduced).  We will show that in this case the forgetful functor
$$\oblv^{\CG}:\fD(\CY)^\CG\to \fD(\CY)$$
admits a left adjoint, denoted $\ind^{\CG}$, and that the resulting endo-functor
$\oblv^{\CG}\circ \ind^{\CG}$ of $\fD(\CY)$ is isomorphic to $\hl_{!}\circ \hr^!$.
\footnote{We are sure that the assertion of the proposition 
is known in this general context. We are supplying a proof for completeness.}

\sssec{Digression: the Beck-Chevalley condition}

Let us note the following feature of pseudo ind-proper maps in $\on{PreStk}$. 
Let $g:\CY_1\to \CY_2$ be a pseudo ind-proper map. In this case, the functor
$$g^!:\fD(\CY_2)\to \fD(\CY_1),$$ admits a left adjoint, which we denote by $g_{!}$. 
Moreover, for a Cartesian square
\begin{equation} \label{e:Cart diag of stacks}
\CD
\CY'_1  @>{g'}>> \CY'_2 \\
@V{t_1}VV   @VV{t_2}V   \\
\CY_1  @>{g}>> \CY_2.
\endCD
\end{equation}
the map
\begin{equation} \label{e:Beck-Chevalley}
g'_{!}\circ t_1^!\to t_2^!\circ g_{!}
\end{equation} that
arises by adjunction from the isomorphism
$$t_1^!\circ g^!\simeq g'{}^!\circ t_2^!,$$
is an isomorphism, i.e., the above adjunction satisfies what is sometimes referred to as ``the
Beck-Chevalley condition". 

\medskip

The isomorphism \eqref{e:Beck-Chevalley} follows from fact that the Beck-Chevalley condition
is satisfied for schemes: 

\medskip

For a Cartesian diagram of schemes  
$$
\CD
S'_1  @>{g'}>> S'_2 \\
@V{t_1}VV   @VV{t_2}V   \\
S_1  @>{g}>> S_2.
\endCD
$$
with $g$ proper, the resulting map $g'_{!}\circ t_1^!\to t_2^!\circ g_{!}$ equals the base change isomorphism
$$g'_{\dr,*}\circ t_1^!\to t_2^!\circ g_{\dr,*}.$$

\sssec{Proof of \propref{p:descr of monad}}

Consider the shifted 
simplicial object $\CG^{1+\bullet}$ of $\on{PreStk}$. We have a
canonical map of simplicial objects
$$g^\bullet:\CG^{1+\bullet}\to \CG^\bullet.$$
Pullback defines a map of cosimplicial categories
$$(g^\bullet)^!:\fD(\CG^\bullet)\to \fD(\CG^{1+\bullet}).$$
However, the fact that the Beck-Chevalley condition holds
implies that the term-wise left adjoint $(g^\bullet)_{!}$ of
$(g^\bullet)^!$ is also a map of cosimplicial categories.
In particular, we obtain adjoint functors between the totalizations:
\begin{equation} \label{e:functors in totalization}
\on{Tot}((g^\bullet)_{!}):\on{Tot}\left(\fD(\CG^{1+\bullet})\right)\rightleftarrows 
\on{Tot}\left(\fD(\CG^{\bullet})\right):\on{Tot}((g^\bullet)^!).
\end{equation}

\medskip

Note now that the simplicial object $\CG^{1+\bullet}$ is augmented by $\CY$ and split. 
Hence, the category $\on{Tot}\left(\fD(\CG^{1+\bullet})\right)$ identifies with
$\fD(\CY)$. In particular, the above functor $\on{Tot}((g^\bullet)_{!})$
provides a left adjoint to $\oblv^{\CG}:\fD(\CY)^\CG\to \fD(\CY)$. 

\medskip

Moreover, it is easy to see that the composition
$$\fD(\CY)\overset{\on{\text{splitting}}} \simeq 
\on{Tot}\left(\fD(\CG^{1+\bullet})\right)\overset{\on{Tot}((g^\bullet)_{!})}\longrightarrow 
\on{Tot}\left(\fD(\CG^{\bullet})\right)\overset{\on{Tot}((g^\bullet)^!)}\longrightarrow 
\on{Tot}\left(\fD(\CG^{1+\bullet})\right) \overset{\on{\text{augmentation}}} \simeq \fD(\CY)$$
is isomorphic to $\hl_{!}\circ \hr^!$, as required.

\qed

\sssec{}

We are now ready to formulate the main theorem of this subsection. Let recall that $\one$ denotes
the unit point of $\Bun_G$, and $\imath_{\one}$ the corresponding map $\on{pt}\to \Bun_G$.

\medskip

We have:

\begin{thm} \label{t:descr of equiv}
The composed functor
$$\fD(\Bun_G)^{\on{Hecke}}\overset{\oblv^{\on{Hecke}}}\longrightarrow \fD(\Bun_G) \overset{\imath^!_{\one}}\to
\Vect$$
is an equivalence.
\end{thm}

\begin{remark}
Let us observe that, like \thmref{t:pullback fully faithful}, there should not be
much surprise in  \thmref{t:descr of equiv} either:

\medskip

Note that the Cartesian square \eqref{e:fibration over Bun} can be extended to a diagram
in which all squares are Cartesian:
\begin{equation} \label{e:big fibration over Bun}
\CD
{\mathbf{Maps}}(X,G)^{\on{rat}}_{\Ran X}  @>{\imath'_{\one}}>>     \Gr_{\Ran X}  @>{p_{\Gr_{\Ran X}}}>>   \on{pt}   \\
& & @V{\imath''_{\one}}VV  @VV{\imath_{\one}}V  \\
@VVV   \CH_{\Ran X}   @>{\hr}>>  \Bun_G  \\
& & @V{\hl}VV \\
\on{pt}  @>{\imath_{\one}}>>  \Bun_G,
\endCD
\end{equation}
where the composed arrow $\hl\circ \imath''_{\one}:\Gr_{X^{\fset}}\to \Bun_G$ is $\pi$.

\medskip

Thus, if we imagine $\CH_{\Ran X}$ as a groupoid acting transitively on $\Bun_G$, we obtain that
the stabilizer of the point $\one$ is ${\mathbf{Maps}}(X,G)^{\on{rat}}_{\Ran X}$ which is contractible. 
In such situation it is natural to expect that the category of groupoid-equivariant objects
be equivalent to $\Vect$.

\end{remark}

\sssec{}  \label{sss:omega equiv}

Let us note that the object
$$\omega_{\Bun_G}\in \fD(\Bun_G)$$
naturally upgrades to an object of $\fD(\Bun_G)^{\on{Hecke}}$. Namely, it corresponds to the object
$$\omega_{|(\CH_{\Ran X})^\bullet|}\in \fD(|(\CH_{\Ran X})^\bullet|)=:\fD(\Bun_G)^{\on{Hecke}}.$$

\medskip

Since $\imath^!_{\one}(\omega_{\Bun_G})\simeq k$, \thmref{t:descr of equiv} implies that
under the equivalence of the theorem, $\omega_{\Bun_G}$ corresponds to $k\in \Vect$.

\medskip

Therefore, the inverse to the functor
of the theorem is given by $V\mapsto V\otimes \omega_{\Bun_G}$.

\sssec{}

We claim that under the identification $\fD(\Bun_G)^{\on{Hecke}}\simeq \Vect$
of \thmref{t:descr of equiv}, the functor
$$\ind^{\on{Hecke}}:\fD(\Bun_G)\to \fD(\Bun_G)^{\on{Hecke}}$$
corresponds to the functor
$$\Gamma_{\dr,c}(\Bun_G,-):\fD(\Bun_G)\to \Vect.$$

\medskip

Indeed, by \secref{sss:omega equiv} above, the right adjoint functor
$$\Vect\simeq \fD(\Bun_G)^{\on{Hecke}}\overset{\oblv^{\on{Hecke}}}\longrightarrow \fD(\Bun_G)$$
identifies with $p_{\Bun_G}^!$, and our
assertion follows from the $(\Gamma_{\dr,c}(\Bun_G,-),p_{\Bun_G}^!)$-adjunction. 

\sssec{Variant with level structure}

We can study a variant of $\fD(\Bun_G)^{\on{Hecke}}$ for the stack $\Bun_G^{\on{level}_D}$
(see \secref{sss:with level structure}). The corresponding groupoid is given by 
identifying the $G$-bundles generically, ignoring the level structure. 

\medskip

Note that in this case the functor
$$\oblv^{\on{Hecke}}:\fD(\Bun^{\on{level}_D}_G)^{\on{Hecke}}\to \fD(\Bun^{\on{level}_D}_G)$$
does not admit a left adjoint because the corresponding map
$$\hl:\CH^{\on{level}_D}_{\Ran X}\to \Bun_G^{\on{level}_D}$$
is no longer pseudo ind-proper.

\medskip

However, we claim that $\fD(\Bun^{\on{level}_D}_G)^{\on{Hecke}}$ is still equivalent to $\Vect$.

\medskip

Indeed, the groupoid $\CH^{\on{level}_D}_{\Ran X}$ is by definition the pullback of the groupoid
$\CH_{\Ran X}$ under the forgetful map $\Bun^{\on{level}_D}_G\to \Bun_G$, i.e.,
$$\CH^{\on{level}_D}_{\Ran X}\simeq (\Bun^{\on{level}_D}_G\times \Bun^{\on{level}_D}_G)
\underset{\Bun_G\times \Bun_G}\times \CH_{\Ran X}.$$

Now, since the map $\Bun^{\on{level}_D}_G\to \Bun_G$ is faithfully flat, and in particular, satisfies
descent of D-modules, we have:
$$\fD(\Bun^{\on{level}_D}_G)^{\CH^{\on{level}_D}_{\Ran X}}\simeq
\fD(\Bun_G)^{\CH_{\Ran X}}.$$

\sssec{Proof of \thmref{t:descr of equiv}}

Let us denote the functor in the theorem by $\sG$. Consider
its left adjoint, which is given by
$$\sF:\Vect\overset{(\imath_{\one})_{!}}\longrightarrow \fD(\Bun_G)\overset{\ind^{\on{Hecke}}}\longrightarrow
\fD(\Bun_G)^{\on{Hecke}}.$$

\medskip

By construction, both functors $\sF$ and $\sG$ are continuous (i.e., commute with colimits). We have:

\begin{lem} \label{l:conservativeness}
The functor $\sG$ is conservative.
\end{lem}

The proof will be given in \secref{sss:proof of cons}. Let us 
continue with the proof of the theorem. We obtain that the pair of functors
$$\sF:\Vect\rightleftarrows \fD(\Bun_G)^{\on{Hecke}}:\sG$$ satisfies the conditions of the
Barr-Beck-Lurie theorem, see e.g., \cite{DG}, Proposition 3.1.1. 

\medskip

Hence, to prove the theorem, it suffices to show that the adjunction map
$$k\to \sG\circ \sF(k)$$ 
is an isomorphism.

\medskip

We are going to calculate the object $\sG\circ \sF(k)$ explicitly. We have:
\begin{equation}  \label{e:calc composition} 
\sG\circ \sF(k)\simeq \imath^!_{\one}\circ \oblv^{\on{Hecke}}\circ \ind^{\on{Hecke}}\circ (\imath_{\one})_{!}(k).
\end{equation}

By \propref{p:descr of monad}, we have an isomorphism of endo-functors on $\fD(\Bun_G)$:
$$\oblv^{\on{Hecke}}\circ \ind^{\on{Hecke}}\simeq \hl_{!}\circ \hr^!,$$
so we have:
\begin{equation}  \label{e:calc composition one} 
\sG\circ \sF(k)\simeq \imath^!_{\one}\circ \hl_{!}\circ \hr^! \circ (\imath_{\one})_{!}(k).
\end{equation}

\medskip

From the right Cartesian square in \eqref{e:big fibration over Bun}, we obtain a map
\begin{equation} \label{e:ULA}
(\imath''_{\one})_{!}\circ (p_{\Gr_{\Ran X}})^!(k)
\to \hr{}^!\circ (\imath_{\one})_!(k)\in \fD(\CH_{\Ran X}).
\end{equation}

\begin{lem} \label{l:ULA}
The map \eqref{e:ULA} is an isomorphism.
\end{lem}

The proof of the lemma is given below, see \secref{sss:proof of ULA}. Assuming the lemma, we obtain that 
$$\sG\circ \sF(k)\simeq \imath^!_{\one}\circ \hl_{!}\circ (\imath''_{\one})_{!}\circ p_{\Gr_{\Ran X}}^!(k)\simeq
\imath^!_{\one}\circ \pi_{!}(\omega_{\Gr_{\Ran X}}).$$

\medskip

However, by \corref{c:dualizing on Bun}, $\pi_{!}(\omega_{\Gr_{\Ran X}})\simeq \omega_{\Bun_G}$.
Hence, we obtain that there exists \emph{an} isomorphism 
\begin{equation} \label{e:an iso}
k\simeq \sG\circ \sF(k).
\end{equation}

\medskip

By tracing through the sequence of isomorphisms that led to \eqref{e:an iso}, one could show
directly that it equals the adjunction map $k\to \sG\circ \sF(k)$. However, instead of doing so, we
will use a shortcut:
 
\medskip

From \eqref{e:an iso}, we obtain that the unit map of the adjunction 
\begin{equation} \label{e:unit}
\on{Id}_{\Vect}\to \sG\circ \sF
\end{equation}
is either zero or an isomorphism. 

\medskip

However, we claim that \eqref{e:unit} is non-zero. Indeed, whenever we have a pair of adjoint functors
between stable categories
$$\sF:\bC_1\rightleftarrows \bC_2:\sG,$$
if the unit of the adjunction is zero, this implies that for any $\bc_1\in \bC_1$ and $\bc_2\in \bC_2$,
$$\Hom(\bc_1,\sG(\bc_2))=0.$$
However, this is clearly not the case in our situation: take $\bc_1=k$ and $\bc_2=\omega_{\Bun_G}$.

\medskip

Thus, we conclude that the unit map \eqref{e:unit} is an isomorphism. Applying the
Barr-Beck-Lurie theorem mentioned above, we conclude that $\sG$ is an equivalence. 

\qed

\sssec{Proof of \lemref{l:ULA}}  \label{sss:proof of ULA}

We will prove the following generalization: 

\medskip

Instead of $\imath_{\one}:\on{pt}\to \Bun_G$
we will take an arbitrary schematic morphism $g:\CY\to \Bun_G$, and instead of $k\in \fD(\on{pt})=\Vect$
an arbitrary $\CF\in \fD(\CY)$, for which $g_!(\CY)\in \fD(\Bun_G)$ is defined. 

\medskip

We will show that in the notation of the following Cartesian diagram
$$
\CD
\CY\underset{\Bun_G}\times \CH_{\Ran X}  @>{g'}>>  \CH_{\Ran X}   \\
@V{\pi_\CY}VV    @VV{\pi}V   \\
\CY  @>{g}>>  \Bun_G
\endCD
$$
the map
$$g'_{!}\circ \pi_\CY^!(\CF')\to \pi^!\circ g_{!}(\CF)\in \fD(\CH_{\Ran X})$$
is an isomorphism (in particular, the left-hand side is defined). 

\medskip

In fact, we will show that the isomorphism takes place for every individual finite set $I$, i.e.,
for the diagram
$$
\CD
\CY\underset{\Bun_G}\times \CH_{X^I}  @>{g'}>>  \CH_{X^I}  \\
@V{\pi(I)_\CY}VV    @VV{\pi(I)}V   \\
\CY  @>{g}>>  \Bun_G
\endCD
$$
and the corresponding functors.

\medskip

For each $I$-tuple $\lambda^I$ of dominant coweights of $G$, let $\CH_{X^I}^{\lambda^I}$ be the correspoding
closed substack of $\CH_{X^I}$. We have
$$\CH_{X^I}\simeq \underset{\lambda^I\in (\Lambda^+)^I}{colim}\, \CH_{X^I}^{\lambda^I},$$
where the set $(\Lambda^+)^I$ is given the standard ordering. 

\medskip

We will show that the indicated isomorphism takes place for every individual $\lambda^I$,
for the functors in the Cartesian diagram
$$
\CD
\CY\underset{\Bun_G}\times \CH_{X^I}^{\lambda^I}  @>{g'}>>  \CH_{X^I}^{\lambda^I}    \\
@V{\pi(I)^{\lambda^I}_\CY}VV    @VV{\pi(I)^{\lambda^I}}V   \\
\CY  @>{g}>>  \Bun_G.
\endCD
$$

The latter follows from the fact that, locally in the smooth topology on $X^I\times \Bun_G$, the 
stack $\CH_{X^I}^{\lambda^I}$ is isomorphic to the product $\Gr_{X^I}^{\lambda^I}\times \Bun_G$.
(In fact $\CH_{X^I}^{\lambda^I}$ is a fiber bundle over $X^I\times \Bun_G$ obtained by
twisting $\Gr_{X^I}^{\lambda^I}$ by a torsor with respect to a certain group-scheme acting 
on it.)

\qed

\sssec{Proof of \lemref{l:conservativeness}}  \label{sss:proof of cons}

The idea of the proof is the following: if $\CF\in \fD(\Bun_G)^{\on{Hecke}}$ is such that
$\imath_{\one}^!(\CF)=0$, then $\imath_{\bg}^!(\CF)=0$ for any other point $\bg\in \Bun_G$,
because any two points of $\Bun_G$ can be connected by a Hecke correspondence.

\medskip

We articulate this idea as follows. Let $\CF\in \fD(\Bun_G)^{\on{Hecke}}$ be an object
such that $\imath^!_{\one}\left(\oblv^{\on{Hecke}}(\CF)\right)=0$. We wish to show
that $\CF=0$. Since the functor $\oblv^{\on{Hecke}}$ is conservative by definition,
it suffices to show that $\CF':=\oblv^{\on{Hecke}}(\CF)\in \fD(\Bun_G)$ vanishes.

\medskip

Let $S$ be a scheme equipped with a map $g:S\to \Bun_G$. We need
to show that $g^!(\CF')=0$.

\medskip

As in the proof of \thmref{t:pullback fully faithful}, possibly after passing to an \'etale cover of
$S$, there exists a finite set $I$ and a map $x^I:S\to X^I$, such that the pullback of
the universal bundle to $S\times X-\{x^I\}$ admits a trivialization. 

\medskip

A choice of such trivialization is the same as a lift of the map $(g,x^I):S\to \Bun_G\times X^I$ to a map
$g':S\to \CH_{X^I}$, 
such that
$$\hl\circ g'=g \text{ and } \hr\circ g'=\imath_{\one}\circ p_S.$$

Hence, 
$$g^!(\CF')\simeq g'{}^!\left(\hl{}^!(\CF')\right)\simeq g'{}^!\left(\hr{}^!(\CF')\right)\simeq
p_S^!(\imath_{\one}^!(\CF'))=0,$$
where the isomorphism
$$\hl{}^!(\CF')\simeq \hr{}^!(\CF')$$
takes place because $\CF'$ came from a Hecke-equivariant object $\CF$.

\qed

\ssec{Another version of the Hecke category}

Let us note that one could define the category $\fD(\Bun_G)^{\on{Hecke}}$
slightly differently. 

\sssec{}

Namely, instead of taking the simplicial object of $\on{PreStk}$ to be
$(\CH_{\Ran X})^\bullet$ defined above, we can take 
$$(\CH^\bullet)_{\Ran X},$$
where we take iterated Cartesian products of $\CH_{\Ran X}$
``over'' $\Ran X$, as well as over $\Bun_G$. 
\footnote{Quotation marks for ``over" are due
to the fact that we are not actually taking Cartesian products over $\Ran X$ in $\on{PreStk}$,
but rather at the level of functors $(\fset)^{op}\to \on{PreStk}$, and then pass to the colimit. The former
is an ill-behaved operation because the category $(\fset)^{op}$ is not filtered.}

\medskip

I.e., each term $(\CH^n)_{\Ran X}$ is by definition
$$\underset{I\in (\fset)^{op}}{colim}\, (\CH^n)_{X^I},$$
where 
\begin{multline*} 
\on{Maps}(S,(\CH^n)_{X^I})=\{x^I\in \on{Maps}(S,X^I),\CP^0,...,\CP^n\in \on{Maps}(S,\Bun_G),\\
\CP^0|_{S\times X-\{x^I\}}\overset{\alpha_1}\simeq
\CP^1|_{S\times X-\{x^I\}},...,\CP^{n-1}|_{S\times X-\{x^I\}}\overset{\alpha_n}\simeq
\CP^n|_{S\times X-\{x^I\}}\}.
\end{multline*}

Let us denote the version of the Hecke-equivariant category defined using $(\CH^\bullet)_{\Ran X}$
by $'\fD(\Bun_G)^{\on{Hecke}}$. We claim that is in fact equivalent to $\fD(\Bun_G)^{\on{Hecke}}$.

\sssec{}

Note that there are maps of simplicial objects of $\on{PreStk}$:
\begin{equation} \label{e:maps between two versions of Hecke}
(\CH^\bullet)_{\Ran X}\leftrightarrows (\CH_{\Ran X})^\bullet
\end{equation}

By definition, the right-hand side is
$$\underset{(I_1,...,I_{n})\in (\fset^{\times n})^{op}}{colim}\, 
\CH_{X^{I_1}}\underset{\Bun_G}\times...\underset{\Bun_G}\times \CH_{X^{I_n}}.$$
It receives a map from the left-hand side via the identification
$$(\CH^n)_{X^I}\simeq X^I\underset{(X^I)^{\times n}}\times (\CH_{\Ran X})^n,$$
where 
$$X^I\to (X^I)^{\times n}=\underset{n}{\underbrace{X^I\times...\times X^I}}$$
is the diagonal map.

\medskip

The map $\leftarrow$ in \eqref{e:maps between two versions of Hecke} is
constructed via the procedure of \secref{sss:expl map}:
we send $$(\fset^{\times n})^{op}\to (\fset)^{op}$$ by the disjoint union functor
$$(I_1,...,I_n)\mapsto I_1\sqcup...\sqcup I_n.$$
The corresponding map of indschemes sends the datum of
\begin{multline*}
\{(x^{I_1},\CP^0,\CP^1,\CP^0|_{S\times X-\{x^{I_1}\}}\overset{\alpha_1}\simeq 
\CP^1|_{S\times X-\{x^{I_1}\}}),...\\
...,(x^{I_n},\CP^{n-1},\CP^n,\CP^{n-1}|_{S\times X-\{x^{I_n}\}}\overset{\alpha_n}\simeq 
\CP^n|_{S\times X-\{x^{I_n}\}}\}
\end{multline*}
to the tautologically defined point of $(\CH^n)_{\Ran X}$, where we regard each $\alpha_k$, $k=1,...,n$
as defined over $S\times X-\{x^{I_1},...,x^{I_n}\}$.

\medskip

Pullback along the maps in \eqref{e:maps between two versions of Hecke} defines functors
\begin{equation} \label{e:maps between two versions of Hecke category}
'\fD(\Bun_G)^{\on{Hecke}}\leftrightarrows \fD(\Bun_G)^{\on{Hecke}},
\end{equation}
and we claim that these functors are mutually inverse equivalences.

\sssec{}

The objects $(\CH_{\Ran X})^\bullet$ and $(\CH^\bullet)_{\Ran X}$, viewed as simplicial objects 
of $\on{PreStk}$ equipped with pseudo ind-proper maps to $\Bun_G$, define simplicial monads
$$(\hl^\bullet)_!\circ (\hr^\bullet)^! \text{ and } ({}'\hl^\bullet)_!\circ ({}'\hr^\bullet)^!,$$
acting on $\fD(\Bun_G)$, where $(\hl^\bullet,\hr^\bullet)$ and $({}'\hl^\bullet,{}'\hr^\bullet)$ are the maps
$$(\CH_{\Ran X})^\bullet\rightrightarrows \Bun_G  \text{ and } (\CH^\bullet)_{\Ran X}\rightrightarrows \Bun_G,$$
respectively. 

\medskip

The monads
$$\oblv^{\on{Hecke}}\circ \ind^{\on{Hecke}} \text{ and } {}'\oblv^{\on{Hecke}}\circ {}'\ind^{\on{Hecke}},$$
acting on $\fD(\Bun_G)$ and responsible for $\fD(\Bun_G)^{\on{Hecke}}$ and $'\fD(\Bun_G)^{\on{Hecke}}$,
are given by $$|(\hl^\bullet)_!\circ (\hr^\bullet)^!| \text{ and } |({}'\hl^\bullet)_!\circ ({}'\hr^\bullet)^!|.$$
respectively. 

\medskip

However, we claim that the maps \eqref{e:maps between two versions of Hecke} define simplicial
isomorphisms of monads
$$(\hl^\bullet)_!\circ (\hr^\bullet)^! \text{ and } ({}'\hl^\bullet)_!\circ {}'\hr^\bullet)^!.$$
This follows from a relative and iterated version of \corref{c:prod}.

\medskip

Alternatively, for the reader who skipped \secref{s:unital}, one can prove this directly,
by repeating the argument of \cite{BD1}, Theorem 4.3.6.

\section{A rederivation of the Atiyah-Bott formula}  \label{s:Atiyah-Bott}

\ssec{A local expression for the homology of $\Bun_G$}

The goal of this section is to explain how the isomorphism of \corref{c:homology of D-modules}
reproduces the Atiyah-Bott formula for $$\Gamma^\bullet_{\dr}(\Bun_G,k)=:H^\bullet(\Bun_G).$$

\sssec{}

Note that \corref{c:homology of D-modules} implies the following isomorphism:

\begin{cor} \label{c:homology of Bun G}
$\Gamma_{\dr,c}(\Bun_G,\omega_{\Bun_G})\simeq \Gamma_{\dr,c}\left(\Gr_{\Ran X},\omega_{\Gr_{\Ran X}}\right)$.
\end{cor}

\sssec{}

First, we observe that the left-hand side in \corref{c:homology of Bun G} is really the
homology of $\Bun_G$ in the sense that
$$\Hom^\bullet_{\Vect}\left(\Gamma_{\dr,c}(\Bun_G,\omega_{\Bun_G}),k\right)\simeq 
H^\bullet(\Bun_G).$$

\medskip

Indeed, for an Artin stack $\CY$, let $\underline{k}_\CY\in \fD(\CY)$ denote the 
``constant sheaf" object, i.e., one for which
$$\Hom^{\bullet}\left(\underline{k}_\CY,-\right)\simeq \Gamma_{\dr}(\CY,-).$$
Then by definition 
$$H^\bullet(\CY):=\Gamma_{\dr}(\CY,\underline{k}_\CY)\simeq
\Hom^{\bullet}_{\fD(\CY)}\left(\underline{k}_\CY,\underline{k}_\CY\right).$$

\medskip

Now, since the left adjoint $(p_\CY)_{!}=\Gamma_{\dr,c}(\CY,-)$ to $p_\CY^!$ is defined on $\omega_\CY$,
we have by definition
$$\Hom^\bullet(\Gamma_{\dr,c}(\CY,\omega_\CY),k)\simeq \Hom^\bullet_{\fD(\CY)}(\omega_\CY,\omega_\CY).$$

\medskip

However, it is easy to see that  
$$\Hom^\bullet_{\fD(\CY)}(\omega_\CY,\omega_\CY)\simeq 
\Hom^{\bullet}_{\fD(\CY)}\left(\underline{k}_\CY,\underline{k}_\CY\right).$$
Indeed, by pulling back to schemes $Z$ mapping smoothly to $\CY$, the
latter isomorphism follows from the fact that $\BD_{\fD(Z)}(\underline{k}_Z)\simeq \omega_{Z}$. 

\sssec{}

Let us now observe that the right-hand side in \corref{c:homology of Bun G} can be rewritten
as follows: 
\begin{equation} \label{e:cohomology via Ran space}
\Gamma_{\dr,c}\left(\Ran X,f_!(\omega_{\Gr_{\Ran X}})\right).
\end{equation}

\medskip

Recall now that the basic feature of the functor $\Gr_{X^{\fset}}:(\fset)^{op}\to \indSch$ is {\it factorization}
with respect to $X^{\fset}$, see \cite{BD2}, Sect. 5.3.12.

\medskip

This implies that the object
$$\CB:=f_!(\omega_{\Gr_{\Ran X}})\in \fD(\Ran X)$$
has a structure of {\it factorization algebra} in the terminology
of \cite{BD1}, Sect. 3.4.4 (or {\it factorization D-module} in the terminology 
of \cite{FrGa}, Sect. 2.4).

\begin{remark}
Let us observe that unlike the situation of Remark \ref{r:not naive}, the object $f_!(\omega_{\Gr_{\Ran X}})\in \fD(\Ran X)$
can be calculated term-wise, i.e.,
$$(\Delta^I)^!\left(f_!(\omega_{\Gr_{\Ran X}})\right)\simeq f(I)_!(\omega_{\Gr_{X^I}}).$$
This is due to the fact that the properness of $\Gr_{X^I}$ over $X^I$ and the proper base change formula.
\end{remark}

\sssec{}

Recall now the notion of chiral algebra on $X$, see \cite{BD1}, Sect. 3.3 (see also
\cite{FrGa}, Sect. 2.4, where the definition is spelled out in the setting of higher
categories; note that in {\it loc.cit.}, chiral algebras on $X$ are referred to as 
``chiral Lie algebras on $X$"). 

\medskip

We shall denote by $\on{Lie-alg}^{\on{ch}}(X)$ the 
category of chiral algebras on $X$.  

\medskip

Thus, by Theorem 1.2.4 of \cite{FrGa}, $\CB$ corresponds to a \emph{chiral algebra}
$B$ on $X$. 

\sssec{}

Recall that for a chiral algebra $B$, its chiral homology $\underset{X}\int\, B$ is defined as 
$\Gamma_{\dr,c}(\Ran X,\CB)$, where $\CB$ is the corresponding factorization algebra,
viewed as an object of $\fD(\Ran X)$. 

\medskip

Thus, we obtain that
\begin{equation} \label{e:AB for homology}
H_\bullet(\Bun_G)\simeq \underset{X}\int\, B,
\end{equation}
for the above chiral algebra $B$ on $X$. 



\sssec{}

One can regard the operation of taking chiral homology as a local-to-global 
principle on $X$. In this sense, \eqref{e:AB for homology} gives a ``local on $X$"
expression for the homology of $\Bun_G$.

\ssec{The Atiyah-Bott formula}

\sssec{}

Recall that for any space $\CY$, the category $\fD(\CY)$ has a natural symmetric monoidal
structure: the monoidal operation corresponds to the composed functor
$$\fD(\CY)\otimes \fD(\CY)\to \fD(\CY\times \CY)\overset{\Delta^!_\CY}\longrightarrow \fD(\CY).$$

\medskip

Recall also (\cite{BD1}, Sect. 3.3.1) that commutative algebras in $\fD(X)$
in the above symmetric monoidal structure give rise to chiral algebras. 

\medskip

If $A$ is a commutative chiral algebra, the vector space
$\underset{X}\int\, A$ has a natural structure of commutative algebra, 
which satisfies the following universal property (see \cite{BD1}, Sect. 4.6.1):

\medskip

Note that the functor $p_X^!:\Vect\to \fD(X)$ has a natural symmetric
monoidal structure (this is true for the functor $g^!$ for any morphism
of prestacks $g:\CY_1\to \CY_2$). Thus, $p_X^!$ gives rise to the (eponymous) functor
$$p_X^!:\on{Com-alg}(\Vect)\to \on{Com-alg}(\fD(X)).$$

\medskip

For $A\in \on{Com-alg}(\fD(X))$, viewed as a chiral algebra, and $A'\in \on{Com-alg}(\Vect)$ we have:
\begin{equation} \label{e:adjunction for commutative conformal blocks}
\Hom_{\on{Com-alg}(\Vect)}\Bigl(\underset{X}\int\, A,A'\Bigr)\simeq
\Hom_{\on{Com-alg}(\fD(X))}\bigl(A,p_X^!(A')\bigr).
\end{equation}

\medskip

For $A\in \on{Com-alg}(\Vect)$ we will use a short-hand notation
$\underset{X}\int\, A$ for $\underset{X}\int\, p^!_X(A)$.

\sssec{}

We apply the above discussion to $A:=H^\bullet(BG)$ and $A'=H^\bullet(\Bun_G)$,
where $BG$ is the stack $\on{pt}/G$. 

\medskip

Pullback along the universal map $\Bun_G\times X\to BG$ gives rise to a map in
$\on{Com-alg}(\fD(X))$
\begin{equation} \label{e:A-B local}
p_X^!(H^\bullet(BG))\to p_X^!(H^\bullet(\Bun_G)).
\end{equation}

Thus, from \eqref{e:adjunction for commutative conformal blocks} we obtain a map
\begin{equation} \label{e:A-B}
\underset{X}\int\, H^\bullet(BG)\to H^\bullet(\Bun_G).
\end{equation}

The Atiyah-Bott formula says that when $G$ is semi-simple and simple connected,
the map \eqref{e:A-B} is an isomorphism. 

\begin{remark}
We emphasize that although the map \eqref{e:A-B} is an isomorphism only when $G$
is semi-simple and simply connected, formula \eqref{e:AB for homology} is valid for any
reductive $G$.
\end{remark}

\sssec{}  \label{sss:usual form}

Let us bring the above version of the Atiyah-Bott formula to a more familiar form. Let us recall
that the commutative algebra $H^\bullet(BG)$ is free, i.e., it is isomorphic to
$\on{Sym}(\fa_G)$ for some particular object $\fa_G\in \Vect$.

\medskip

In fact $$\fa_G\simeq \underset{e}\oplus\, k[-2\cdot e],$$ where $e$ runs through the set of
exponents of $G$. 

\medskip

By \cite{BD1}, Proposition 4.6.2, chiral homology of a free commutative chiral algebra
$\on{Sym}(V)$ is
computed by 
$$\underset{X}\int\, \on{Sym}(V)\simeq \on{Sym}(V\otimes H_\bullet(X)).$$

\medskip

Taking $V=\fa_G$, we obtain that \eqref{e:A-B} gives rise to an isomorphism
\begin{equation} \label{e:A-B usual}
\Sym(\fa_G\otimes H_\bullet(X))\simeq H^\bullet(\Bun_G),
\end{equation}
which is the more usual form of the Atiyah-Bott formula.

\ssec{The rederivation of the formula}

Let us now explain the equivalence of the isomorphisms \eqref{e:A-B} and
\eqref{e:AB for homology}. We will only give a sketch; a detailed proof will
appear in the forthcoming joint paper of the author and J.~Lurie, \cite{GaLu}.

\sssec{}

Let $\fb_G$ be the Lie algebra in $\Vect$ that governs the homotopy type 
of $BG$ tensored with $k$. I.e., by definition, $\fb_G$ is the Lie algebra such
that
$$\on{C}^\bullet(\fb_G)\simeq H^\bullet(BG),$$
where $\on{C}^\bullet$ denotes the cohomological Chevalley complex of a Lie algebra.

\medskip

In the case of $BG$, the Lie algebra $\fb_G$ is abelian, which corresponds to the fact
that $H^\bullet(BG)$ is free as a commutative algebra:

\medskip 

We have
$$\fa_G\simeq \fb_G^\vee[-1],$$
where $\fa_G\in \Vect$ is the vector space from \secref{sss:usual form}.

\sssec{}

Recall the notion of Lie-* algebra on $X$, see \cite{BD1}, Sect. 2.5
(this is what is called a $\star$-Lie algebra on $X$ in \cite{FrGa}, Sect. 2.4).
Let $\on{Lie-alg}^\star(X)$ denote the category of Lie-* algebras on $X$.

\medskip

Recall also that if $L$ is a Lie algebra in $\Vect$, the object $p_X^{\dr,*}(L)\in \fD(X)$ 
naturally upgrades to one in $\on{Lie-alg}^\star(X)$. Here 
$$p_X^{\dr,*}:\Vect\to \fD(X)$$ is the functor of *-pullback, left adjoint to the 
direct image functor $(p_X)_!\simeq (p_X)_{\dr,*}$.

\medskip

Finally, recall that the forgetful functor
$$\on{Lie-alg}^{\on{ch}}(X)\to \on{Lie-alg}^{\star}(X)$$
admits a left adjoint, called the functor of \emph{chiral universal envelope},
and denoted by $U^{\on{ch}}$.

\medskip

We have:

\begin{prop}  \label{p:descr of chiral algebra}
For $G$ semi-simple and simply connected, the chiral algebra $B$ of \eqref{e:AB for homology} 
is canonically isomorphic to $U^{\on{ch}}(p_X^{\dr,*}(\fb_G))$.
\end{prop}

\begin{remark}
The assertion of the above proposition is well-known in topology. At the level
of !-stalks at points of $X$ it says that 
$$H_\bullet(\Gr_x)\simeq \on{C}_\bullet(\fb_G[-2]),$$
where $x$ is any point of $X$, and $\Gr_x$ is the fiber of $\Gr_{X}$ over
$x\in X$, and $\fb_G[-2]$ is the Lie algebra obtained by looping $\fb_G$
twice. This results from the fact that $\Gr_x$ is homotopy-equivalent to
the double-loop space of $BG$. 
\end{remark}

The proof will be sketched in \secref{ss:ident of chiral algebra}. Let us 
proceed to showing how \eqref{e:AB for homology}, combined with 
\propref{p:descr of chiral algebra}, reproduces the isomorphism
\eqref{e:A-B}.

\sssec{}

By \cite{BD1}, Theorem 4.8.1 (see also \cite{FrGa}, Corollary 6.3.4), 
for a Lie-* algebra $L$ on $X$ we have:
$$\underset{X}\int\, U^{\on{ch}}(L)\simeq \on{C}_\bullet(\Gamma_{\dr,c}(X,L)),$$
where $\Gamma_{\dr,c}(X,L)\in \Vect$ acquires a natural structure of Lie
algebra by \cite{FrGa}, Sect. 6.2.1. We apply this to $L=p_X^{\dr,*}(\fb_G)$,
and we obtain that
\begin{equation} \label{e:homology}
H_\bullet(\Bun_G)\simeq \underset{X}\int\, B\simeq 
\on{C}_\bullet\left(\Gamma_{\dr,c}(X,p^{\dr,*}_X(\fb_G))\right).
\end{equation}

Hence,
\begin{multline} \label{e:cohomology}
H^\bullet(\Bun_G)\simeq \left(H_\bullet(\Bun_G)\right)^\vee\simeq  \\
\simeq \left(\on{C}_\bullet\left(\Gamma_{\dr,c}(X,p^{\dr,*}_X(\fb_G))\right)\right)^\vee\simeq
\on{C}^\bullet\left(\Gamma_{\dr,c}(X,p^{\dr,*}_X(\fb_G))\right).
\end{multline}

\sssec{}

On the other hand, we have $H^\bullet(BG)\simeq \on{C}^\bullet(\fb_G)$,
and therefore we can identify the commutative chiral algebra 
$p_X^!(H^\bullet(BG))$ with 
$$\on{C}_{\fD(X)}^\bullet(p^{\dr,*}_X(\fb_G)),$$
where for a $L\in \on{Lie-alg}^{\star}(X)$ we denote by 
$$\on{C}_{\fD(X)}^\bullet(L)\in \on{Com-alg}(\fD(X))$$
the corresponding Chevalley algebra, see \cite{BD1}, Sect. 1.4.14. 

\medskip

By \cite{BD1}, Proposition 4.7.1, we have:
$$\underset{X}\int\, \on{C}_{\fD(X)}^\bullet(L)\simeq \on{C}^\bullet(\Gamma_{\dr,c}(X,L)).$$
Thus, we obtain that
\begin{equation} \label{e:cohomology again}
\underset{X}\int\, H^\bullet(BG)\simeq \on{C}^\bullet\left(\Gamma_{\dr,c}(X,p^{\dr,*}_X(\fb_G))\right).
\end{equation}

Comparing \eqref{e:cohomology} with \eqref{e:cohomology again}, we deduce the desired
isomorphism
\begin{equation} \label{e:A-B again}
H^\bullet(\Bun_G)\simeq \underset{X}\int\, H^\bullet(BG).
\end{equation}

\begin{remark}
The construction of the isomorphism $\CB\simeq U^{\on{ch}}(p_X^{\dr,*}(\fb_G))$ sketched in
\secref{ss:ident of chiral algebra} implies that the isomorphism \eqref{e:A-B again} coincides
with that of \eqref{e:A-B}.
\end{remark}

\ssec{Proof of \propref{p:descr of chiral algebra}} \label{ss:ident of chiral algebra}

We shall only sketch the proof; details will be supplied in \cite{GaLu}. 

\sssec{}

Let $Z$ be any functor $(\fset)^{op}\to \indSch$ as in \secref{sss:over Ran}, which is a unital 
\emph{factorization monoid}, see \cite{BD1}, Sect. 3.10.16. 
In this case $\CB:=f_!(\omega_{Z_{\Ran X}})$ is a unital and augmented factorization 
algebra in $\fD(\Ran X)$. Let $B$ be the corresponding chiral algebra on $X$. We let
$\ol{B}$ denote the augmentation ideal of $B$, and $\ol\CB$ the corresponding factorization
algebra on $\Ran X$.

\medskip

The diagonal map $Z\to Z\times Z$ defines on $B$ a structure of unital augmented
commutative co-algebra object in the category of chiral algebras. 

\medskip

Now, the functor $U^{\on{ch}}$
canonically factors through a functor
\begin{equation} \label{e:Lie * to fact}
\on{Lie-alg}^{\star}(X)\to \on{Com-coalg}(\on{Lie-alg}^{\on{ch}}_{\on{unital,augm}}(X)),
\end{equation}
followed by the forgetful functor
$$\on{Com-coalg}(\on{Lie-alg}^{\on{ch}}_{\on{unital,augm}}(X))\to 
\on{Lie-alg}^{\on{ch}}(X).$$
Moreover, the functor 
of \eqref{e:Lie * to fact} induces an equivalence on 
the subcategories of objects satisfying an appropriate connectivity hypothesis. In particular,
\secref{sss:connectivity} implies that this hypothesis is satisfied by $B$ if the fibers of $Z(\{1\})$ 
over $X$ are connected. 

\sssec{}

Let us return now to the situation when $Z=\Gr_{X^{\fset}}$. If $G$ is semi-simple and
simply connected, the fibers of $\Gr_{X}$ over $X$ are connected. Hence, we obtain
that in this case
$$B\simeq U^{\on{ch}}(L)$$
for a canonically defined Lie-* algebra $L$ on $X$.

\medskip

Let $\CA$ denote the factorization algebra corresponding to the commutative chiral
algebra $A:=p_X^!(H^\bullet(BG))$. Let $\ol{A}$ denote the augmentation ideal of $A$,
corresponding to the unit point of $BG$. Let $\ol{\CA}$ denote the corresponding 
factorization algebra.

\medskip

A local version of the map
\eqref{e:A-B local} gives rise to a map
\begin{equation} \label{e:A-B factor}
\on{union}_!(\ol{\CB}\boxtimes \ol{\CA})\to \omega_{\Ran X},
\end{equation}
compatible with the commutative co-algebra structure on $\ol\CB$ and the commutative algebra
structure on $\ol\CA$.

\medskip

Since 
$$p_X^!(H^\bullet(BG))\simeq \on{C}^\bullet_{\fD(X)}(p^{\dr,*}_X(\fb_G)),$$ 
the above properties
of the map \eqref{e:A-B factor} define a map of Lie-* algebras 
\begin{equation} \label{e:ident of Lie}
L\to  p^{\dr,*}_X(\fb_G),
\end{equation}

It remains to show that \eqref{e:ident of Lie} is an isomorphism. This is a local question,
hence we can assume that $X=\BP^1$. It is easy to see that it is sufficient to show that 
the map
$$\on{C}_\bullet(\Gamma_{\dr,c}(X,L))\to \on{C}_\bullet(\Gamma_{\dr,c}(X,p^{\dr,*}_X(\fb_G)))$$
is an isomorphism. I.e., that the composition
\begin{equation} \label{e:local to global comp}
\underset{X}\int\, H^\bullet(BG)\simeq  \on{C}^\bullet(\Gamma_{\dr,c}(X,p^{\dr,*}_X(\fb_G)))
\to \on{C}^\bullet(\Gamma_{\dr,c}(X,L))\simeq \left(\underset{X}\int\, B\right)^\vee\simeq H^\bullet(\Bun_G)
\end{equation}
is an isomorphism for $X=\BP^1$. We shall do so by reversing the above manipulations. 

\sssec{}

The main observation is that the map \eqref{e:A-B factor} makes the following diagram commute:
\begin{equation} \label{e:local to global}
\CD
\Gamma_{\dr,c}(\Ran X,\on{union}_!(\ol\CB\boxtimes \ol\CA))  
@>{\sim}>> \Gamma_{\dr,c}(\Ran X,\ol\CB)  \otimes \Gamma_{\dr,c}(\Ran X,\ol\CA) \\ 
@VVV     @VVV   \\ 
\Gamma_{\dr,c}(\Ran X,\omega_{\Ran X})  & &  \ol{H}_\bullet(\Bun_G)\otimes \ol{H}^\bullet(\Bun_G)   \\
@VVV    @VVV  \\
k   @>{\on{id}}>>  k,
\endCD
\end{equation} 
where $\ol{H}_\bullet(\Bun_G)$ (resp., $\ol{H}^\bullet(\Bun_G)$) denotes the reduced
homology (resp., cohomology) of $\Bun_G$ with respect to the base point ${\mathbf 1}\in \Bun_G$.
This property is true for any $G$, not necessarily semi-simple 
and simply connected.

\medskip

Now, for a Lie-* algebra $L'$, if we denote by $\CB'$ the factorization algebra corresponding to
the chiral algebra $U^{\on{ch}}(L')$, and by $\CA'$ the factorization algebra corresponding to
the chiral algebra $\on{C}_{\fD(X)}^\bullet(L')$, and by $\ol\CB'$ and $\ol\CA'$ their reduced counterparts
(i.e., the factorization algebras corresponding to the augmentation ideals in the corresponding chiral
algebras), the canonical map
$$\on{union}_!(\ol\CB'\boxtimes \ol\CA')\to \omega_{\Ran X}$$
makes the following diagram commute:
$$
\CD
\Gamma_{\dr,c}(\Ran X,\on{union}_!(\ol\CB'\boxtimes \ol\CA'))  
@>{\sim}>> \Gamma_{\dr,c}(\Ran X,\ol\CB')  \otimes \Gamma_{\dr,c}(\Ran X,\ol\CA') \\ 
@VVV    @VV{\sim}V   \\
\Gamma_{\dr,c}(\Ran X,\omega_{\Ran X})     & &  \ol{\on{C}}_\bullet(\Gamma_{\dr,c}(X,L'))\otimes 
\ol{\on{C}}^\bullet(\Gamma_{\dr,c}(X,L')) \\
@VVV    @VVV  \\
k   @>{\on{id}}>>  k.
\endCD
$$

This implies that the map \eqref{e:local to global comp} equals the map \eqref{e:A-B}. 

\sssec{}

Thus, it remains
to show that the map \eqref{e:A-B} is an isomorphism for $X=\BP^1$, which is an easy verification: it 
follows e.g., from the fact that the map 
$$\Gr_0/G\to \Bun_G(\BP^1)$$
induces an isomorphism on cohomology, and the computation of the $G$-equivariant cohomology of $\Gr_0$.

\qed

\section{Appendix: contractibility of the Ran space}  \label{s:proof contr Ran}

Here is the promised proof of \thmref{t:contractibility of Ran}. By \secref{sss:connectivity},
$H_\bullet(\Ran X)$ is connective, and $H_0(\Ran X)$ maps isomorphically to $k$. 

\medskip

Assume by
contradiction that for some $n>0$ we have $H_n(\Ran X)\neq 0$. With no restriction of
generality, we can assume that $n$ is minimal such integer. By the K\"unneth formula, we obtain that 
the maps $\on{id}\times p_{\Ran X}$ and $p_{\Ran X}\times \on{id}$
$$\Ran X\times \Ran X\rightrightarrows \Ran_X$$
define an isomorphism
$$H_n(\Ran X\times \Ran X)\to H_n(\Ran X)\oplus H_n(\Ran X).$$
Denote $M:=H_n(\Ran X)$. 

\medskip

The map $$\on{union}:\Ran X\times \Ran X\to \Ran X$$ defines, therefore, a map
$$M\oplus M\to M,$$
which by symmetry must be of the form $u\oplus u$ for some map $u:M\to M$. By associativity,
for any $k\geq 2$, the corresponding map
$$(\Ran X)^{\times k}\overset{\on{union}^k}\longrightarrow \Ran X$$
acts on $H_n$ as $u^{\oplus k}$. 

\medskip

For an integer $k$, consider now the diagonal map $\Ran X\to (\Ran X)^{\times k}$. It induces on
$H_n$ the diagonal map $M\to M^{\oplus k}$. 

\medskip

By \secref{sss:sq of diag}, the composition
$$\Ran X \to (\Ran X)^{\times k}\overset{\on{union}^k}\longrightarrow \Ran X$$
is the identity map. Hence, we obtain that $k\cdot u=\on{id}$ for any $k\geq 2$. Taking $k$
to be, e.g., $2$ and $3$, we obtain that $2\cdot u=3\cdot u$, i.e., $u=0$. Hence, $\on{id}:M\to M$ 
is the zero map, which is a contradiction. 

\qed


\begin{thebibliography}{99}





\bibitem[Ba]{Ba} J.$\sim$Barlev, {\em Moduli spaces of generic data over a curve}, forthcoming.

\bibitem[BD1]{BD1} A.$\sim$Beilinson and V.$\sim$Drinfeld, 
{\it Chiral algebras}, American Mathematical Society Colloquium Publications {\bf 51} (2004). 

\bibitem[BD2]{BD2} A.$\sim$Beilinson and V.$\sim$Drinfeld, 
{\it Quantization of Hitchin's integrable system and Hecke eigensheaves}, \newline
available from http://math.uchicago.edu/$\sim$mitya/langlands.html$\sim$

\bibitem[DrGa0]{DrGa0} V.$\sim$Drinfeld and D.$\sim$Gaitsgory, {\it On some finiteness questions
for algebraic stacks}, arXiv:1108.5351.

\bibitem[DrGa1]{DrGa1} V.$\sim$Drinfeld and D.$\sim$Gaitsgory, {\it Compact generation of the category of D-modules on $\Bun_G$},
\newline
arXiv:1112.2402.

\bibitem[DS]{DS} V.$\sim$Drinfeld and C.$\sim$Simpson, {\it $B$-structures on $G$-bundles and local
triviality}, Mathematical Research Letters {\bf 2} (1995), 823--829. 

\bibitem[FrGa]{FrGa} J.$\sim$Francis and D.$\sim$Gaitsgory, {\it Chiral Koszul duality}, 
arXiv:1103.5803. 




\bibitem[GL:DG]{DG} Notes on Geometric Langlands, {\it Generalities on DG categories}, \newline
available from http://www.math.harvard.edu/$\sim$gaitsgde/GL/

\bibitem[GL:Ran]{Ran} Notes on Geometric Langlands, {\it Categories over the Ran space}, \newline
available from http://www.math.harvard.edu/$\sim$gaitsgde/GL/

\bibitem[GL:Stacks]{Stacks} Notes on Geometric Langlands, {\it Stacks}, \newline
available from http://www.math.harvard.edu/$\sim$gaitsgde/GL/

\bibitem[GL:QCoh]{QCoh} Notes on Geometric Langlands, {\it Quasi-coherent sheaves on stacks}, \newline
available from http://www.math.harvard.edu/$\sim$gaitsgde/GL/

\bibitem[GL:IndCoh]{IndCoh} Notes on Geometric Langlands, {\it Ind-coherent sheaves}, arXiv:1105.4857.

\bibitem[GL:IndSch]{IndSch} Notes on Geometric Langlands, {\it DG indschemes}, arXiv:1108.1738.

\bibitem[GL:Crystals]{Crystals} Notes on Geometric Langlands, {\it Crystals and D-modules}, arXiv:1111.2087.

\bibitem[GaLu]{GaLu} D.$\sim$Gaitsgory and J.$\sim$Lurie, {\it On the Tamagawa number formula}, forthcoming.











\bibitem[Lu0]{Lu0} J.$\sim$Lurie, {\it Higher Topos Theory}, Princeton Univ. Press (2009). 

\bibitem[Lu1]{Lu1} J.$\sim$Lurie, {\it Higher Algebra}, available from http://www.math.harvard.edu/$\sim$lurie

\bibitem[MV]{MV} I.$\sim$Mirkovic and K.$\sim$Vilonen, 
"Geometric Langlands duality and representations of algebraic groups over commutative rings", 
Annals of Math. {\bf 166} (2007), 95--143.



\end{thebibliography}
\end{document}